\documentclass[11pt, dvipsnames]{amsart}
\usepackage{amsmath,amssymb,amsfonts, mathtools}
\usepackage[mathscr]{eucal}
\usepackage[margin=1in, marginparwidth=2cm, marginparsep=2mm]{geometry}
\usepackage[all]{xy}
\usepackage{bbm}
\usepackage{xparse}
\allowdisplaybreaks

\usepackage{tikz, tikz-cd}
\usetikzlibrary{matrix, arrows, decorations.pathmorphing, cd, patterns, calc, backgrounds, patterns}
\usetikzlibrary{decorations.pathreplacing, calc}
\usepackage{extarrows}
\tikzset{dummy/.style= {circle,fill,draw,inner sep=0pt,minimum size=1.2mm}}
\tikzset{vertex/.style={fill, circle, minimum size=.1cm, inner sep=0pt}}

\newcommand{\HG}{\mathrm{H}\mathbf{G}}
\newcommand{\TG}{\mathrm{T}\mathbf{G}}

\usepackage[pagebackref, colorlinks, citecolor=DarkOrchid, linkcolor=black, urlcolor=NavyBlue]{hyperref}
\hypersetup{linktoc=all,} %set to all if you want both sections and subsections linked

\hypersetup{
  colorlinks   = true, %Colours links instead of ugly boxes
  urlcolor     = blue, %Colour for external hyperlinks
  linkcolor    = PineGreen, %Colour of internal links
colorlinks, citecolor    = DarkOrchid %Colour of citations
}

\usepackage{comment}

\usepackage[final]{microtype}

\usepackage{enumerate}

\usepackage{appendix}

\numberwithin{equation}{subsection} 
\numberwithin{figure}{subsection}
\usepackage[nameinlink,capitalise,noabbrev]{cleveref}

\newcommand{\newrefformat}[2]{}

\theoremstyle{plain}
\newtheorem{theorem}[equation]{Theorem}
\newtheorem{corollary}[equation]{Corollary}
\newtheorem{proposition}[equation]{Proposition}
\newtheorem{lemma}[equation]{Lemma}

\newtheorem{introtheorem}{Theorem}

\theoremstyle{definition}
\newtheorem{definition}[equation]{Definition}
\newtheorem{example}[equation]{Example}

\newtheorem{observation}[equation]{Observation}
\newtheorem{construction}[equation]{Construction}
\newtheorem{remark}[equation]{Remark}

\newtheorem{notation}[equation]{Notation}

\newcommand{\DG}{{\Delta \mathbf{G}}}

\crefname{lemma}{Lemma}{Lemmas}
\crefname{theorem}{Theorem}{Theorems}
\crefname{definition}{Definition}{Definitions}
\crefname{proposition}{Proposition}{Propositions}
\crefname{remark}{Remark}{Remarks}
\crefname{corollary}{Corollary}{Corollaries}
\crefname{question}{Question}{Questions}
\crefname{equation}{Equation}{Equations}
\crefname{construction}{Construction}{Constructions}
\crefname{ex}{Example}{Examples}
\crefname{appsec}{Appendix}{Appendices}
\crefname{observation}{Observation}{Observations}
\crefname{subsection}{Subsection}{Subsections}
\crefname{section}{Section}{Sections}
\crefname{conjecture}{Conjecture}{Conjectures}
\crefname{lettertheorem}{Theorem}{Theorems}
\crefname{notation}{Notation}{Notations}

\usepackage[textwidth=18mm, textsize=tiny]{todonotes} %%% for 1in margin
\usepackage{graphicx,scalerel}
\newcommand\temp[1][.8]{\mathbin{\ThisStyle{\vcenter{\hbox{%
  \scalebox{#1}{$\SavedStyle\bullet$}}}}}%
  }
\newcommand{\smbullet}{{\temp}}

\newcommand{\HC}{\mathrm{HC}}
\newcommand{\HR}{\mathrm{HR}}

\newcommand{\Fun}{\mathrm{Fun}}

\newcommand{\HS}{\mathrm{H}\Sigma}
\newcommand{\DS}{{\Delta \Sigma}}
\newcommand{\DB}{{\Delta \mathrm{B}}}

\newcommand{\pull}{\arrow[dr, phantom, "\lrcorner", very near start]}
\newcommand{\DC}{{\Delta \mathrm{C}}}

\newcommand{\BDB}{B_{\DB}^\smbullet}
\newcommand{\BDC}{B_{\DC}^\smbullet}
\newcommand{\BDCu}{B^{\DC}_\smbullet}
\newcommand{\BDS}{B_{\DS}^\smbullet}
\newcommand{\BDG}{B_{\Delta \mathbf{G}}^\smbullet}
\newcommand{\HB}{\mathrm{HB}}

\newcommand{\cO}{\mathscr{
O
}}

\newcommand{\Bcyc}{B^\mathrm{cyc}}

\newcommand{\sdual}{\mathscr{I}}

\newcommand{\lax}{\mathrm{lax}}
\newcommand{\TS}{\mathrm{T}\Sigma}
\newcommand{\TB}{\mathrm{TB}}

\newcommand{\TF}{\mathrm{T}\varphi}
\newcommand{\THH}{\mathrm{THH}}
\newcommand{\TC}{\mathrm{TC}}
\newcommand{\HH}{\mathrm{HH}}
\newcommand{\Hom}{\mathrm{Hom}}

\newcommand{\Pic}{\mathrm{Pic}}
\newcommand{\Free}{\mathrm{Free}}

\DeclareMathOperator{\Env}{\mathrm{Env}}
\DeclareMathOperator{\Tor}{\mathrm{Tor}}

\newcommand{\op}{\mathrm{op}}

\newcommand{\Assoc}{\mathrm{Assoc}}

\newcommand{\Sp}{\mathrm{Sp}}

\newcommand{\Alg}{\mathrm{Alg}}

\newcommand{\Vect}{\mathrm{Vect}}

\newcommand{\cC}{\mathscr{C}}
\newcommand{\cD}{\mathscr{D}}

\newcommand{\id}{\mathrm{id}}

\newcommand{\Map}{\mathrm{Map}}

\newcommand{\Fin}{\mathrm{Fin}}

\tikzcdset{
  diagrams={>={Straight Barb[scale=0.5]}}
}
\tikzset{
  altstackar/.style={decorate, decoration={show path construction,
    lineto code={
      \path (\tikzinputsegmentfirst); \pgfgetlastxy{\xstart}{\ystart}
      \path (\tikzinputsegmentlast); \pgfgetlastxy{\xend}{\yend}
      \path ($(0,0)!1.5pt!(\ystart-\yend,\xend-\xstart)$); \pgfgetlastxy{\xperp}{\yperp}
      \foreach \n[evaluate=\n as \k using .5*#1-\n+.5] in {1,...,#1}{
        \ifodd\n{\draw[->, shorten <=2pt, shift={($\k*(\xperp,\yperp)$)}](\xstart,\ystart)--(\xend,\yend);}
        \else{\draw[<-, shorten >=2pt, shift={($\k*(\xperp,\yperp)$)}](\xstart,\ystart)--(\xend,\yend);}\fi
      }
    }
  }}, altstackar/.default={1}
}

\NewDocumentCommand{\tens}{e{_^}}{%
  \mathbin{\mathop{\otimes}\displaylimits
    \IfValueT{#1}{_{#1}}
    \IfValueT{#2}{^{#2}}
  }%
}

\keywords{Topological symmetric homology, topological braid homology, crossed simplicial group, free $\mathbb{E}_\infty$-ring, free $\mathbb{E}_2$-ring, monoidal envelope}
\subjclass[2020]{
Primary: 16E40, %(Co)homology of rings and associative algebras (e.g., Hochschild, cyclic, dihedral, etc.)
55P43, %Spectra with additional structure
57T30. %% Bar and cobar constructions
Secondary: 
18M75 %Topological and simplicial operads
18N60, %infinity categories
55P42. %Stable homotopy theory, spectra
}

\author[G.\ Angelini-Knoll]{Gabriel Angelini-Knoll}
\address{Department of Mathematics, Applied Mathematics, and Statistics, Case
Western Reserve University, Cleveland, OH, USA}
\email{gja39@case.edu}

\author[D.\ Chan]{David Chan}
\address{Department of Mathematics, Michigan State University, 619 Red Cedar Road, East Lansing, MI 48824, USA}
\email{chandav2@msu.edu}

\author[T.\ Gerhardt]{Teena Gerhardt}
\address{Department of Mathematics, Michigan State University, 619 Red Cedar Rd, East Lansing, MI 48824, USA}
\email{teena@math.msu.edu}

\author[M.\ Merling]{Mona Merling} 
\address{Department of Mathematics, University of Pennsylvania, 209 33rd St, Philadelphia, PA 19104, USA}
\email{mmerling@math.upenn.edu}

\author[M.\ P\'eroux]{Maximilien P\'eroux}
\address{Department of Mathematics, Michigan State University, 619 Red Cedar Rd, East Lansing, MI 48824, USA}
\email{peroux@msu.edu}

\title[Topological symmetric and braid homologies]{Topological symmetric and braid homologies}

\begin{document}

\begin{abstract}
We identify topological symmetric homology as the free $\mathbb{E}_\infty$-algebra on an $\mathbb{E}_1$-algebra and topological braid homology as the free $\mathbb{E}_2$-algebra on an $\mathbb{E}_1$-algebra. 
In this way, topological symmetric homology and topological braid homology can be regarded as 
variants of $1$-dimensional representation homology. 
In order to identify topological braid homology as the free $\mathbb{E}_2$-algebra on an $\mathbb{E}_1$-algebra, we prove that the $\mathbb{E}_2$-monoidal envelope of the associative operad can be identified with the braided crossed simplicial group. 
Using this, we also compute the topological braid homology of grouplike $\mathbb{E}_1$-spaces. Further, we develop computational tools for topological symmetric and braid homologies and prove that topological symmetric homology is not Morita invariant.  
We also compute the topological $\Delta \mathbf{G}$-homology of Thom spectra in general and produce explicit formulas in the case of topological symmetric and braid homologies. 
\end{abstract}

\maketitle

\section{Introduction}

Topological Hochschild homology is an invariant of rings and ring spectra, which plays a key role in the trace methods approach to algebraic $K$-theory. In ~\cite{MSV97}, McClure--Schw\"anzl--Vogt proved that topological Hochschild homology enjoys a universal property: it is the free commutative ring spectrum with circle action on a commutative ring spectrum. This universal property was strengthened to incorporate genuine equivariant structure in~\cite{ABGHLM18}. Understanding this universal property of topological Hochschild homology facilitates the definition of topological cyclic homology, which plays an essential role in many computations of algebraic $K$-theory~\cite{DGM13}. 

% Topological Hochschild homology is an invariant of rings and ring spectra, which plays a key role in the trace methods approach to algebraic $K$-theory. In ~\cite{MSV97}, McClure--Schw\"anzl--Vogt proved that the topological Hochschild homology of a commutative ring spectrum is the free commutative ring spectrum with circle action; this result was strengthened to incorporate genuine equivariant structure in~\cite{ABGHLM18}. Understanding this universal property of topological Hochschild homology facilitates the definition of topological cyclic homology, which plays an essential role in many computations of algebraic $K$-theory~\cite{DGM13}. 

One way to define topological Hochschild homology is as the geometric realization of the cyclic bar construction associated to Connes' cyclic category. In~\cite{FL91}, the authors generalize the cyclic category by replacing the role of the cyclic groups by more general automorphism groups. 
These generalizations are called \emph{crossed simplicial groups} since the various automorphisms form a simplicial set where each level is a group and the face and degeneracy maps are akin to crossed homomorphisms. 
In \cite{THG}, the first, fourth and fifth authors produce topological homology theories for arbitrary crossed simplicial groups, and establish foundational properties of these constructions. 
In addition to new constructions, such as topological quaternionic and hyperoctohedral homologies, the authors recover Real topological Hochschild homology as an instance of this construction associated to the dihedral crossed simplicial group.

% Topological Hochschild homology is the geometric realization of the cyclic bar construction, which is a bar construction associated to Connes' cyclic category. In~\cite{FL91}, the authors generalize the cyclic category by replacing the role of the cyclic groups by more general automorphism groups. 
% These generalizations are called \emph{crossed simplicial groups} since the various automorphisms form a simplicial set where each level is a group and the face and degeneracy maps are akin to crossed homomorphisms. 
% In \cite{THG}, the first, fourth and fifth authors produce topological homology theories for arbitrary crossed simplicial groups, and establish foundational properties of these constructions. 
% In addition to new constructions, such as topological quaternionic and hyperoctohedral homologies, the authors recover Real topological Hochschild homology as an instance of this construction associated to the dihedral crossed simplicial group.

It is natural to wonder if the homology theories associated to crossed simplicial groups satisfy universal properties analogous to the universal property of THH.  In some cases this is known. For instance in~\cite{AKGH25}, the first author, third author, and Hill prove that Real topological Hochschild homology has a universal property as the free commutative ring spectrum with $O(2)$-action on a genuine commutative $C_2$-ring spectrum. In this paper, we focus on two of the new topological theories developed in \cite{THG}: topological symmetric homology $\TS$, and topological braid homology $\TB$. These theories produce spectra from any $\mathbb{E}_1$-ring spectrum, and are computed as explicit homotopy colimits. Our first main result identifies the universal property of topological symmetric homology. 

\begin{introtheorem}[\cref{Einfty-structure,topological symmetric homology is topological representation homology}]\label{intro theorem A}
For an  $\mathbb{E}_1$-ring spectrum $R$,  the topological symmetric homology $\TS(R)$ has a canonical $\mathbb{E}_{\infty}$-ring spectrum structure, and it is the free $\mathbb{E}_\infty$-ring spectrum on $R$.
\end{introtheorem}

The universal property of $\TS$ we prove can be viewed as a topological refinement of work of Berest and Ramadoss~\cite{Berest-Ramadoss}. In~\cite{Berest-Ramadoss}, they characterize (algebraic) symmetric homology, denoted $\HS$, in terms of $1$-dimensional representation homology in characteristic zero. Here $1$-dimensional representation homology is the derived free commutative algebra on an associative algebra. Representation homology was defined in~\cite{BKR13} and has been studied extensively, e.g.~\cite{BFPRW17,BRY22,Berest-Ramadoss,BR24,Li26}.  In~\cite{Berest-Ramadoss}, Berest and Ramadoss also produce a character map from cyclic homology to symmetric homology. The topological analogue of cyclic homology is the homotopy orbits of the $S^1$-action on topological Hochschild homology, which we denote by $\mathrm{TC}^+$. We produce the topological analogue of the character map 
\[ 
    \TC^+(R)\to \TS(R) 
\]
in \cref{prop:character-map}.

The starting point for \cref{intro theorem A} is a result of~\cite{THG}, which identifies the symmetric crossed simplicial group with the symmetric monoidal envelope of the associative operad. To obtain similar results about topological braid homology, we prove an analogous interpretation of the braided crossed simplicial group. We write $\DB_+$ for the augmented braided category, see~\cref{augmented-braid}.  

\begin{introtheorem}[\cref{thm:envelope}]
The $\mathbb{E}_2$-monoidal envelope of the associative operad is equivalent to $\DB_+$.  
\end{introtheorem}

In other words, given any $\mathbb{E}_2$-monoidal $\infty$-category $\cC$, an $\mathbb{E}_2$-monoidal functor $\DB_+\to \cC$ is precisely the data of an $\mathbb{E}_1$-algebra in $\cC$. 
Since we have proven that $\TS(R)$ is the free $\mathbb{E}_{\infty}$-ring spectrum on an $\mathbb{E}_1$-ring spectrum $R$, it is natural to ask whether there are analogous constructions for the free $\mathbb{E}_n$-ring spectra on an $\mathbb{E}_1$-ring spectrum. In the case of $\mathbb{E}_2$, we show that topological braid homology plays this role. This result extends unpublished work of \cite{Fie} to the topological setting. 

\begin{introtheorem}[\cref{TB-rep-homology}]
 For an $\mathbb{E}_1$-ring spectrum $R$, the topological braid homology $\TB(R)$ has a canonical $\mathbb{E}_2$-ring spectrum structure and it is the free $\mathbb{E}_2$-ring spectrum  on $R$. 
\end{introtheorem}

Besides these qualitative results, we also provide computational tools for topological symmetric homology and topological braid homology. The construction of $\TS(R)$ as a homotopy colimit of a symmetric bar construction gives us access to a spectral sequence which we can use to understand the low-degree homotopy groups. In particular, for a connective $\mathbb{E}_1$-ring spectrum $R$ we show in \cref{cor:low-degrees} that there is a canonical map 
\[
    \TS_s(R)\to \HS_s(\pi_0R)
\]
that induces an isomorphism in degree $0$. As a consequence, we prove in \cref{cor:not-morita} that topological symmetric homology is not Morita invariant as a functor on connective $\mathbb{E}_1$-ring spectra. As part of our tool kit, in \cref{prop:Bokstedt} we produce  a B\"okstedt-type spectral sequence computing the homology of $\TS(R)$. The $E^2$-page of this spectral sequence can be described in terms of (algebraic) symmetric homology. In fact, this computational tool exists in the greater generality of twisted symmetric homologies as constructed in~\cite{THG}. 

When the input ring spectrum is a Thom spectrum, we are able to give more complete computations. Recall that when $R$ is a multiplicative Thom spectrum, a theorem of Blumberg, Cohen, and Schlichtkrull \cite{BCS} gives an efficient computation of the topological Hochschild homology of $R$. Essentially, they show that the Thom spectrum functor commutes with the cyclic bar construction, in an appropriate sense.  We extend their result to the (positive) topological $\DG$-homology of Thom spectra for all crossed simplicial groups $\DG$. The definition of (positive) topological $\DG$-homology is recalled in \cref{def: DG-homology}.  Let $BF$ denote the classifying space of stable spherical fibrations.

\begin{introtheorem}[{\cref{BCS}}]
    Let $\DG$ be a crossed simplicial group.  Let $f\colon X\to BF$ be a map of $\mathbb{E}_1$-spaces with twisted $G_0$-action and let $Mf$ denote the Thom spectrum of $f$.  Then there is an equivalence of spectra
\[
    M(\HG^+(f))\simeq \TG^+(Mf)\,.
\]    
Moreover, if $\DG$ is self-dual then there is an equivalence of spectra with $|\mathbf{G}_{\smbullet}|$-action
\[
    M(\mathrm{TH}\mathbf{G}(f))\simeq \mathrm{TH}\mathbf{G}(Mf)\,.
\] 
\end{introtheorem}
In the case where $\DG = \Delta \mathrm{D}$, the dihedral crossed simplicial group, this recovers some results from \cite{HKZ} on Real topological Hochschild homology of Thom spectra. 

Blumberg, Cohen, and Schlichtkrull further show that if $f\colon X\to BF$ is a $3$-fold loop map classifying a stable spherical fibration on a base space $X$, then $\THH(Mf)\simeq Mf\otimes BX_+$~\cite{BCS}. We prove analogous results when $\DG = \DS$ or $\DB$. For any pointed space $Y$, let $QY = \Omega^{\infty}\Sigma^{\infty}Y$.  When $Y$ is an $\infty$-loop space, there is a natural map $QY\to Y$, and we denote the homotopy fiber by $\widetilde{Q}Y$.  Similarly, we write $JY = \Omega \Sigma Y$, and if $Y$ is a loop space there is a natural map $JY\to Y$ with homotopy fiber $\widetilde{J}Y$. 
\begin{introtheorem}[{\cref{thm:TS-thomspectra} and \cref{thm:TB-thomspectra}}]
    If $X$ is an $\infty$-loop space and $f\colon X\to BF$ is an $\infty$-loop map, with Thom spectrum $Mf$, then there is an equivalence of spectra
    \[
        \TS(Mf)\simeq Mf\otimes \Omega \widetilde{Q}BX_+.
    \]
    Similarly, if $X$ is a two-fold loop space and $f:X\to BF$ is a two-fold loop map then there is an equivalence of spectra
    \[
        \TB(Mf)\simeq Mf\otimes \Omega \widetilde{J}BX_+.
    \]
\end{introtheorem}
 
These results allow us to identify $\TS(R)$ and $\TB(R)$ for many multiplicative Thom spectra of interest, for example $R = \mathrm{MU}$ in the case of $\TS$ as well as $R=H\mathbb{F}_p$ and $H\mathbb{Z}$ in the case of $\TB$; see \cref{cor: TS thom spectra examples,cor:TB-Thom-spectra-examples}. Computations of homotopy groups require more in-depth analysis, but we note that the stable homotopy types of $\widetilde{Q}BX$ and $\widetilde{J}BX$ can be understood using the Snaith and James splittings, respectively.  This, together with the Hilton--Milnor theorem for computing the loops of a wedge of spaces makes it possible to compute some homotopy groups explicitly.

\subsection{Acknowledgments} The authors would like to thank Liam Keenan, Tyler Lawson, and Cary Malkiewich for helpful conversations. The second author was supported by NSF grant DMS-2135960. The third author was supported by NSF grant DMS-2404932. The fourth author was supported by NSF grant DMS-2506429 and a Simons Fellowship. The authors would also like to thank the Isaac Newton Institute for Mathematical Sciences, Cambridge, for support and hospitality during the programme ``Equivariant homotopy theory in context'' where some work on this paper was undertaken. 

\section{Symmetric homology}\label{section: symmetric homology}

In this section we recall the definition of topological symmetric homology from \cite{THG}, and give a comparison to representation homology of \cite{Berest-Ramadoss}. In particular, we give topological refinements of some of the results in \cite{Berest-Ramadoss}.  We also describe a character map relating these constructions to cyclic homology.

\subsection{Symmetric homology}\label{subsection: symmetric homology} The notion of crossed simplicial groups was introduced independently in \cite{FL91} and \cite{Krasauskas}.
In particular, the \emph{symmetric category} $\DS$, extensively studied in \cite{Fie, Aul10, Aul14}, extends the simplex category $\Delta$ via the addition of automorphisms on the objects $[n]$ through the symmetric groups $\Sigma_{n+1}$. We now give the full definition.

\begin{definition}
    
 The objects of $\DS$ are the finite ordered sets $[n] = \{0,\dots,n\}$, and morphisms $\DS([m],[n])$ are pairs $(\phi,\gamma)$ where $\phi\in \Delta([m],[n])$ is an order preserving map and $\gamma\in \Sigma_{m+1}^\op$. The morphisms in $\DS$ are generated by the coface and codegeneracy maps in $\Delta$, $\delta^i_n\colon [n-1]\to [n]$ and $\sigma^i_n\colon [n+1]\to [n]$ for $0\leq i \leq n$, as well as maps $t^i_n\colon [n]\to [n]$ for $0\leq i \leq n-1$ viewed as the transposition which swaps $i$ and $i+1$. These morphisms are subject to the usual cosimplicial identities, Moore relations (i.e., $t^i_nt^i_n=\mathrm{id}$, $t^i_n t^j_n=t^j_n t^i_n$ for $|i-j|>1$ and $t^i_n t^{i+1}_n t^i_n=t^{i+1}_n t^i_n t^{i+1}_n$), as well as the following \textit{mixed relations}: 
\begin{equation}\label{eq:mixes relations}
t_n^i\,\delta_n^j=
\begin{cases}
\delta_n^j\,t_{n-1}^i, & i<j-1,\\[2pt]
\delta_n^{i}, & i=j-1,\\[2pt]
\delta_n^{i+1}, & i=j,\\[2pt]
\delta_n^j\,t_{n-1}^{i-1}, & i>j,
\end{cases}
\quad \quad 
t_n^i\,\sigma_n^j=
\begin{cases}
\sigma_n^j\,t_{n+1}^i, & i<j-1,\\[2pt]
\sigma_n^{i}\,t_{n+1}^{i+1}\,t_{n+1}^{i}, & i=j-1,\\[2pt]
\sigma_n^{i+1}\,t_{n+1}^{i}\,t_{n+1}^{i+1}, & i=j,\\[2pt]
\sigma_n^j\,t_{n+1}^{i+1}, & i>j.
\end{cases}
\end{equation} 
See \cite[E.6.1.7]{Lod92} for more discussion.
\end{definition}

 \begin{remark}
     Note that the symmetric category can also be described as the category of non-empty finite sets, where morphisms are maps of finite sets with choices of orderings on the fibers, cf.~\cite[\S~1.2]{PR02}.
 \end{remark}

 Note that composition in $\DS$ is defined so that $\Delta\subset\DS$ is a subcategory.  Moreover, the definition is such that a functor $\DS\to \mathscr{C}$ is precisely the data of a cosimplicial object $X^{\smbullet}$ in $\mathscr{C}$, together with symmetric group actions of $\Sigma_{n+1}$ on $X^n$ subject to the compatibility described by the mixed relations.
 
 \begin{remark} By \cite[Proposition 1.6]{FL91}, the category $\DS$ is entirely determined by functions $\gamma^*\colon \Hom_{\Delta}([m], [n])\to \Hom_{\Delta}([m], [n])$ for each $\gamma\in\Sigma_{n+1}$ and functions $\phi^*\colon \Sigma_{n+1}\to \Sigma_{m+1}$ for each map $\phi\colon [m]\to [n]$ in $\Delta$ such that the following diagram (of sets) commutes
\[
  \xymatrix{
    [m] \ar[d]_-{\phi^\ast \gamma} \ar[r]^-{\gamma^\ast \phi } & [n] \ar[d]^-\gamma\\
    [m] \ar[r]_-{\phi } &[n].
    }
\]
The maps $\gamma^*(\phi)$ and $\phi^*(\gamma)$ are defined as follows.  Note that any map $\phi\colon [m]\to [n]$ in $\Delta$ is determined by the sequence of numbers $(|\phi^{-1}(1)|,\dots,|\phi^{-1}(n)|)$. The map $\gamma^*\phi$ is the unique order preserving function such that $(\gamma^*\phi)^{-1}(k) =\phi^{-1}(\gamma(k))$.  The map $\phi^*(\gamma)$ is the unique permutation which makes the square commute and is order preserving when restricted  to $(\gamma^*\phi)^{-1}(k)$ for all $k$.  See an example below in \cref{figure-symmetric-operation}.
From this perspective, the composition of two morphisms  $(\phi, \gamma)$ and $(\psi, \tau)$ in $\DS$ is defined as:
\[
(\phi, \gamma)\circ (\psi, \tau)= (\phi\circ \gamma^*\psi, \psi^*\gamma \circ \tau).
\]
\end{remark}
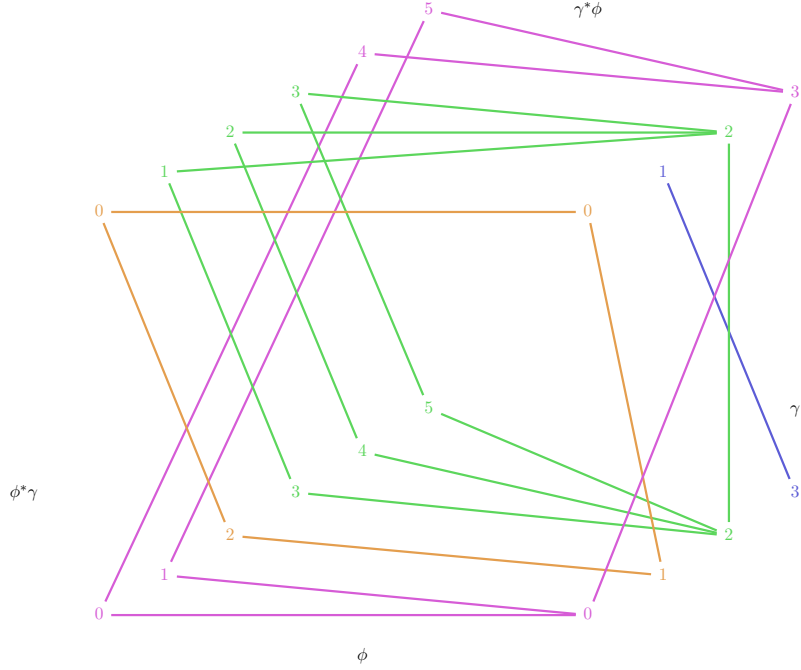
\begin{figure}[ht!]
 \centering
\begin{tikzpicture}[baseline= (a).base]
\node[scale=0.60] (a) at (1,1){
\begin{tikzcd}[row sep=small]
&&&&&& |[alias=TL5]|\textcolor{rgb,255:red,214;green,92;blue,214}{5} &&& {\gamma^\ast \phi} \\
	&&&&& |[alias=TL4]|\textcolor{rgb,255:red,214;green,92;blue,214}{4} \\
	&&&& |[alias=TL3]|\textcolor{rgb,255:red,92;green,214;blue,92}{3} &&&&&&&& |[alias=TR3]|\textcolor{rgb,255:red,214;green,92;blue,214}{3} \\
	&&& |[alias=TL2]|\textcolor{rgb,255:red,92;green,214;blue,92}{2} &&&&&&&& |[alias=TR2]|\textcolor{rgb,255:red,92;green,214;blue,92}{2} \\
	&& |[alias=TL1]|\textcolor{rgb,255:red,92;green,214;blue,92}{1} &&&&&&&& |[alias=TR1]|\textcolor{rgb,255:red,92;green,92;blue,214}{1} \\
	& |[alias=TL0]|\textcolor{rgb,255:red,228;green,158;blue,78}{0} &&&&&&&& |[alias=TR0]|\textcolor{rgb,255:red,228;green,158;blue,78}{0} \\
    &&&&&&&\\
&&&&&&& \\
&&&&&&&&\\
&&&&&&&&\\
&&&&&&&\\
&&&&&&&&\\
&&&&&&&&\\
&&&&&&&&\\
&&&&&&&&&\\
&&&&&&&&&\\
&&&&&& |[alias=BL5]|\textcolor{rgb,255:red,92;green,214;blue,92}{5} &&&&& & {\gamma} \\
	&&&&& |[alias=BL4]|\textcolor{rgb,255:red,92;green,214;blue,92}{4} \\
	{\phi^\ast \gamma} &&&& |[alias=BL3]|\textcolor{rgb,255:red,92;green,214;blue,92}{3} &&&&&&&& |[alias=BR3]|\textcolor{rgb,255:red,92;green,92;blue,214}{3} \\
	&&& |[alias=BL2]|\textcolor{rgb,255:red,228;green,158;blue,78}{2} &&&&&&&& |[alias=BR2]|\textcolor{rgb,255:red,92;green,214;blue,92}{2} \\
	&& |[alias=BL1]|\textcolor{rgb,255:red,214;green,92;blue,214}{1} &&&&&&&& |[alias=BR1]|\textcolor{rgb,255:red,228;green,158;blue,78}{1} \\
	& |[alias=BL0]|\textcolor{rgb,255:red,214;green,92;blue,214}{0} &&&&&&&& |[alias=BR0]|\textcolor{rgb,255:red,214;green,92;blue,214}{0} \\
	&&&&& {\phi}
	\arrow[draw={rgb,255:red,214;green,92;blue,214}, no head, from=TL5, to=TR3, line width=1.4pt]
	\arrow[draw={rgb,255:red,214;green,92;blue,214}, no head, from=TL5, to=BL1, line width=1.4pt]
	\arrow[draw={rgb,255:red,214;green,92;blue,214}, no head, from=TL4, to=TR3, line width=1.4pt]
	\arrow[draw={rgb,255:red,214;green,92;blue,214}, no head, from=TL4, to=BL0, line width=1.4pt]
	\arrow[draw={rgb,255:red,92;green,214;blue,92}, no head, from=TL3, to=BL5, line width=1.4pt]
	\arrow[draw={rgb,255:red,92;green,214;blue,92}, no head, from=TL2, to=TR2, line width=1.4pt]
	\arrow[draw={rgb,255:red,92;green,214;blue,92}, no head, from=TL2, to=BL4, line width=1.4pt]
	\arrow[draw={rgb,255:red,92;green,214;blue,92}, no head, from=TR2, to=TL3, line width=1.4pt]
	\arrow[draw={rgb,255:red,92;green,214;blue,92}, no head, from=TL1, to=TR2, line width=1.4pt]
	\arrow[draw={rgb,255:red,92;green,214;blue,92}, no head, from=TL1, to=BL3, line width=1.4pt]
	\arrow[draw={rgb,255:red,228;green,158;blue,78}, no head, from=TL0, to=TR0, line width=1.4pt]
	\arrow[draw={rgb,255:red,228;green,158;blue,78}, no head, from=TL0, to=BL2, line width=1.4pt]
	\arrow[draw={rgb,255:red,92;green,214;blue,92}, no head, from=BL5, to=BR2, line width=1.4pt]
	\arrow[draw={rgb,255:red,92;green,214;blue,92}, no head, from=BL4, to=BR2, line width=1.4pt]
	\arrow[draw={rgb,255:red,92;green,214;blue,92}, no head, from=BL3, to=BR2, line width=1.4pt]
	\arrow[draw={rgb,255:red,92;green,92;blue,214}, no head, from=TR1, to=BR3, line width=1.4pt]
	\arrow[draw={rgb,255:red,228;green,158;blue,78}, no head, from=BL2, to=BR1, line width=1.4pt]
	\arrow[draw={rgb,255:red,92;green,214;blue,92}, no head, from=BR2, to=TR2, line width=1.4pt]
	\arrow[draw={rgb,255:red,214;green,92;blue,214}, no head, from=BL1, to=BR0, line width=1.4pt]
	\arrow[draw={rgb,255:red,228;green,158;blue,78}, no head, from=BR1, to=TR0, line width=1.4pt]
	\arrow[draw={rgb,255:red,214;green,92;blue,214}, no head, from=BL0, to=BR0, line width=1.4pt]
	\arrow[draw={rgb,255:red,214;green,92;blue,214}, no head, from=BR0, to=TR3, line width=1.4pt]
\end{tikzcd}
};  
\end{tikzpicture}
\caption{Depiction of the operations $\phi^\ast$ and $\gamma^\ast$}
\label{figure-symmetric-operation}
\centering
\end{figure}

 Let $k$ be a field, and let $\Vect_k$ denote the category of $k$-vector spaces and linear transformations. 

\begin{definition}\label{symmetric bar construction}
The \textit{symmetric bar construction} of a $k$-algebra $R$ is denoted $\BDS(R)$ and is defined as a functor $\DS\to \Vect_k$
\[
\begin{tikzcd}
    R  \arrow[r, altstackar=3] & R^{\otimes 2} \ar[out=120, in=60, loop, looseness=6]{}{\Sigma_2} \arrow[altstackar=5]{r} & R^{\otimes 3} \ar[out=120, in=60, loop, looseness=6]{}{\Sigma_{3}} \arrow[altstackar=7]{r} & \cdots .
\end{tikzcd}
\]
The cofaces are defined as:
\begin{align*}
    \delta^i\colon R^{\otimes n+1}&\longrightarrow R^{\otimes n+2}\\*
    r_0\otimes \cdots \otimes r_n &\longmapsto r_0 \otimes \cdots 
    r_{i-1} \otimes 1_R \otimes r_i \otimes \cdots \otimes r_n.
\end{align*}
The codegeneracies are defined by
\begin{align*}
    \sigma^i\colon R^{\otimes n+1} & \longrightarrow R^{\otimes n}\\*
    r_0\otimes \cdots \otimes r_n & \longmapsto r_0\otimes \cdots \otimes r_{i-1} \otimes r_ir_{i+1} \otimes r_{i+2} \otimes \cdots \otimes r_n.
\end{align*}
The symmetric group $\Sigma_{n+1}$ acts on $R^{n+1}$ for all $g\in \Sigma_{n+1}$ via
\begin{align*}
    R^{\otimes n+1} & \longrightarrow R^{\otimes n+1}\\*
    r_0 \otimes \cdots \otimes r_n & \longmapsto r_{g^{-1}(0)}\otimes \cdots r_{g^{-1}(n)}.
\end{align*}

\end{definition}

There is an analogous bar construction defined using Connes' cyclic category $\DC\subset \DS$ (often also denoted $\Lambda$), for which automorphisms are given by the cyclic groups $C_{n+1}\subset \Sigma_{n+1}$. That is, the morphisms in  $\DC$ are generated by the coface and codegeneracy maps $\delta^i_n$ and $\sigma^i_n$ in $\Delta$, as well as maps $t_n\colon [n]\to [n]$ viewed as the generators of the groups $C_{n+1}$. The morphisms are subject to the usual cosimplicial identities, the cyclic relation $t_n^{n+1}=\id$, as well as the following mixed relations: 
\[
    t_n \,\delta_n^i=
        \begin{cases}
            \delta_n^n, & i=0,\\[2pt]
            \delta_n^{i-1}\,t_{n-1}, & i>0,
        \end{cases}
        \quad \quad 
        t_n \,\sigma_n^i=
        \begin{cases}
            \sigma_n^n\,t_{n+1}^2, & i=0,\\[2pt]
            \sigma_n^{i-1}\,t_{n+1}, & i>0.
    \end{cases}
\]

\begin{definition}
    The \textit{cocyclic bar construction} of a $k$-algebra $R$, denoted $\BDC(R)$, is defined as the functor $\DC\to \Vect_k$ obtained by precomposing $\BDS(R)\colon \DS\to \Vect_k$ with the inclusion functor $\DC\hookrightarrow \DS$:
    \[
\begin{tikzcd}
    R  \arrow[r, altstackar=3] & R^{\otimes 2} \ar[out=120, in=60, loop, looseness=6]{}{C_2} \arrow[altstackar=5]{r} & R^{\otimes 3} \ar[out=120, in=60, loop, looseness=6]{}{C_{3}} \arrow[altstackar=7]{r} & \cdots .
\end{tikzcd}
\]
Note that, $\BDC(R)$ and $\BDS(R)$ have the same underlying cosimplicial object. 
\end{definition}

The cyclic category $\DC$ is self-dual: there is an isomorphism of categories $\sdual\colon \DC^\op\overset{\cong}{\to} \DC$ that is the identity on objects, and on generating morphisms
\begin{align*}
   \sdual\left([n-1]\xrightarrow{\delta^i} [n]\right) &=\begin{cases}
       [n]\xrightarrow{\sigma^i} [n-1] & \text{if }0\leq i \leq n-1\\
       [n]\xrightarrow{\sigma^0 t_n^{-1}} [n-1] & \text{if }i=n,
   \end{cases}\\
   \sdual\left( [n+1]\xrightarrow{\sigma^i}[n]\right) &=[n]\xrightarrow{\delta^{i+1}}[n+1],\\
   \sdual\left( [n]\xrightarrow{t_n}[n]\right) &=
    [n]\xrightarrow{t_n^{-1}}[n], 
   \end{align*}
   where $t_n$ is the generator of $C_{n+1}$, see \cite{Dun89}. Using this duality, one recovers the more classical cyclic bar construction from the cocyclic bar construction.

   \begin{definition}
       The \textit{cyclic bar construction} of a $k$-algebra $R$, denoted $\BDCu(R)$, is  the functor $\BDC(R)\circ \sdual\colon \DC^\op\to \Vect_k$.  
       Explicitly, the faces are given by:
\begin{align*}
    R^{\otimes n+1} & \stackrel{\delta_i}\longrightarrow R^{\otimes n}\\
    r_0\otimes \cdots \otimes r_n &\longmapsto \begin{cases}
    r_0\otimes \cdots \otimes r_ir_{i+1} \otimes \cdots \otimes r_n & \text{if }0\leq i \leq n-1\\
    r_nr_0\otimes r_1\otimes \cdots \otimes r_{n-1} & \text{if }i=n
    \end{cases}
\end{align*}
while the degeneracies are given by:
\begin{align*}
    R^{\otimes n+1} & \stackrel{\sigma_i}\longrightarrow R^{\otimes n+2}\\
    r_0\otimes \cdots \otimes r_n &\longmapsto r_0 \otimes \dots \otimes r_{i} \otimes 1_R \otimes r_{i+1} \otimes \dots \otimes  r_n.
\end{align*}
Moreover, each cyclic group $C_{n+1}=\langle t_n \rangle$ acts on $R^{\otimes n+1}$ as follows:
\begin{align*}
   R^{\otimes n+1} & \stackrel{t_n}\longrightarrow R^{\otimes n+1}\\
    r_0\otimes \cdots \otimes r_n & \longmapsto r_n\otimes r_0 \otimes \cdots \otimes r_{n-1}.
\end{align*}
       We denote the underlying simplicial object of $\BDCu(R)$ by $\Bcyc_\smbullet(R)\colon \Delta^\op\to \Vect_k$.
   \end{definition}

The symmetric and (co)cyclic bar constructions give rise to homology theories for $k$-algebras.

\begin{definition}
Let $R$ be a $k$-algebra.  The \textit{symmetric homology groups $\HS_*(R)$ of $R$} are defined as the functor homology groups
\[
    \HS_n(R)=\Tor_n^\DS(\underline{k}, \BDS(R)),
\]
where $\underline{k}\colon \DS^\op\to \Vect_k$ denotes the constant diagram at $k$.
The \textit{cyclic homology groups $\HC_*(R)$ of $R$} are defined equivalently as:
\[
\HC_n(R)=\Tor_n^\DC(\underline{k}, \BDC(R))\cong \Tor_n^{\DC^{\op}}(\underline{k}, \BDCu(R)).
\]
The \textit{Hochschild homology groups $\HH_*(R)$ of $R$} are defined as:
\[
\HH_n(R)=\Tor_n^{\Delta^{\op}}(\underline{k}, \Bcyc_\smbullet(R))  \,.
\]
\end{definition}

\begin{remark}
    Unlike $\DC$, the symmetric category is not self-dual: there is no equivalence of categories between $\DS$ and $\DS^\op$ \cite[1.4]{Dun89}.
    Moreover, by \cite[Theorem 6.16]{FL91}, given a functor $X_\smbullet\colon \DS^\op\to \Vect_k$, there is an isomorphism $\Tor_n^\DS(X_\smbullet, \underline{k})\cong \pi_n(|X_\smbullet|)$,  
    where $|X_\smbullet|$ is the geometric realization of the underlying simplicial $k$-vector space $X_\smbullet$ in the category of $k$-vector spaces. 
    Therefore a contravariant version of the symmetric bar construction would not lead to a new invariant.
\end{remark}

\begin{remark}\label{derived HS}
    We can define $\HS$ and $\HC$ more generally in the category of simplicial $R$-modules, for example, in which case we define this in the same way and will always mean the derived version of the construction. In particular, we can make sense of  $\HS$ and $\HC$ of a connective differential graded algebra. 
\end{remark}

\begin{remark}\label{remark: HS as homotopy groups}
    Using the previous remark, if $R$ is a graded $k$-algebra, viewed as a simplicial module, then $\HS_n(R)$ is equivalent to the $n$-th homotopy group of the homotopy colimit of the symmetric bar construction $B_\DS^\smbullet(R)\colon \DS\to \mathrm{sMod}_k$, see \cite[4.24, 4.26]{THG}.
\end{remark}

\subsection{Topological symmetric homology}\label{sec:TS}

In this section, we discuss a topological refinement of symmetric homology which produces a spectrum from any $\mathbb{E}_1$-ring spectrum. We show that this spectrum has a canonical $\mathbb{E}_{\infty}$-structure. We begin with some preliminary observations about the symmetric category which will prove useful in the construction of the $\mathbb{E}_{\infty}$-structure.

One goal of this section is to argue that the symmetric category has an analogous universal property to the monoidal category $T$, the terminal category with unique object $*$ and unique morphism. 
Indeed, the terminal category $T$ has a unique monoidal structure $*+*=*$, and given any monoidal category $(\cC,\otimes,1_{\cC})$ there is an equivalence of categories between lax monoidal functors $T\to \cC$ and monoid objects in $\cC$.
A functor $F\colon T\to \cC$ is then the data of an object $A\coloneqq F(*)\in \cC$, and a lax monoidal structure on $F$ is the data of a morphism 
\[
A\otimes A=F(*)\otimes F(*)\longrightarrow F(*+*)=F(*)=A
\]
in $\cC$ and a unit $1_\cC\to F(*)=A$, which is associative and unital, as $F$ is lax monoidal. We will first observe that the augmented simplex category enjoys an analogous universal property. 

\begin{definition}\label{def: augmented simplex cat}
  Let $\Delta_+$ be the augmented simplex category: it is the simplex category $\Delta$ with an additional initial object. Denote by $[-1]$ the initial object with unique maps $[-1]\to [n]$ for all $n$. We sometimes shift the notation and write $\langle n \rangle = [n-1]$ for all $n\geq 0$. 
\end{definition}

The augmented simplex category $\Delta_+$ is a (non-symmetric) monoidal category with monoidal product:
\[
\langle n \rangle \sqcup \langle m \rangle = \langle n+m \rangle 
\]
for $n,m\geq 0$.
Given  morphisms $f\colon \langle n \rangle \to \langle n' \rangle$ and $g\colon \langle m \rangle \to \langle m' \rangle $ in $\Delta_+$, define $f\sqcup g$ to be the  morphism $\langle n+m\rangle \to \langle n'+m' \rangle$ determined by $f$ on the first $n$ elements and by $g$ on the last $m$ elements. This determines a bifunctor $-\sqcup -\colon \Delta_+\times \Delta_+ \to \Delta_+$ that is a monoidal product, with monoidal unit given by $\langle 0 \rangle$.
Notice that while $\langle n+m \rangle$ and $\langle m+n \rangle$ are the same set, the reordering bijection which swaps the first $n$ elements with the last $m$ is not order preserving. That is, the naive choice of symmetric structure is not a morphism in $\Delta_+$. Therefore $\Delta_+$ is monoidal, but not symmetric monoidal. 

\begin{remark}\label{rem: universal prop of Delta augmented}
The monoidal structure on $\Delta_+$ has the following universal property. 
Given any monoidal category $\cC$, there is an equivalence of categories between strong monoidal functors $\Delta_+\to \cC$ and monoid objects in $\cC$, see \cite[VII \S5 Proposition 1]{MacLane}. There is an adjunction 
\[
    \begin{tikzcd}[column sep=large]
       \mathrm{MonCat}^\mathrm{lax}\ar[bend left]{r}{}\ar[phantom, "\perp" description, xshift=0.25ex]{r} & \mathrm{MonCat}^\mathrm{strong}\ar[bend left]{l}{U}
    \end{tikzcd}
\]
where $\mathrm{MonCat}^\mathrm{lax}$ is the category of small monoidal categories with lax monoidal functors, while $\mathrm{MonCat}^\mathrm{strong}$ is the category of small monoidal categories with strong monoidal functors. The functor $U$ is the forgetful functor, which has a left adjoint that sends the terminal category $T$ with its unique monoidal structure to $\Delta_+$. 
\end{remark}

The symmetric category $\DS$ plays the role of $\Delta$ in the symmetric version of the above story.
Let $\DS_+$ be the augmented symmetric category, obtained from $\Delta\Sigma$ via the addition of an initial object as in \cref{def: augmented simplex cat}.
Then one can extend the monoidal structure on $\Delta_+$ to $\DS_+$, and in fact, $\DS_+$ is a \textit{symmetric} monoidal category \cite[Lemma 2.25]{THG}.

In \cite[Proposition 3.10]{THG}, it was proved that (the nerve of) $\DS_+$ is equivalent, as a symmetric monoidal $\infty$-category, to the symmetric monoidal envelope $\Env(\Assoc)$ of the associative operad, which is just the subcategory of $\Assoc^\otimes$ spanned by the active morphisms \cite[2.2.4.3]{HA}. In \cref{section on representation}, we recall the more general $\cO$-monoidal envelope of $\Assoc$ where $\cO$ is an $\infty$-operad, see \cref{Monoidal envelopes without fibration}. For instance, the (nerve of) $\Delta_+$ is the $\mathbb{E}_1$-monoidal envelope $\Env_{\mathbb{E}_1}(\Assoc)$ of the associative operad, in the sense of \cite[2.2.4.1]{HA}; see also \cite[Construction 3.4]{Keenan25}. 

As in \cref{rem: universal prop of Delta augmented}, the monoidal structure on $\DS_+$ has a universal property: given any symmetric monoidal category $\cC$, the category of strong symmetric monoidal functors $\DS_+\to \cC$ and monoidal natural transformations is equivalent to the category of (non-commutative) monoid objects in $\cC$. 
Moreover, there is an adjunction 
\[
    \begin{tikzcd}[column sep=large]
    \mathrm{MonCat}^\mathrm{strong}\ar[bend left]{r}{}\ar[phantom, "\perp" description, xshift=0.75ex]{r} & \mathrm{SymMonCat}^\mathrm{strong},\ar[bend left]{l}{U}
    \end{tikzcd}
\]
where $\mathrm{SymMonCat}^\mathrm{strong}$ is the category of symmetric monoidal small categories with strong monoidal functors. Here $U$ is the forgetful functor, which has a left adjoint that sends $\Delta_+$ to $\DS_+$. 

The above observations also show that 
given any symmetric monoidal $\infty$-category $\cC$, the $\infty$-category of strong symmetric monoidal functors from the nerve of $\DS_+$ to $\cC$ is equivalent to the $\infty$-category of $\mathbb{E}_1$-algebras in $\cC$. 
In particular, any $\mathbb{E}_1$-ring spectrum $R$, viewed as a map of $\infty$-operads $R^\otimes\colon \Assoc^\otimes\to \Sp^\otimes$, can be equivalently regarded as a strong symmetric monoidal functor $F_R^\otimes\colon \Env(\Assoc)^\otimes\to \Sp^\otimes$ by the universal property of the symmetric monoidal envelope \cite[2.2.4.9]{HA}. Note that this agrees with the composite 
\[
\begin{tikzcd}
   \Env(\Assoc)^\otimes \ar{r}{\Env(R)}& [2em]\Env(\Sp)^\otimes \ar{r}{\otimes} &\Sp^\otimes \,. 
\end{tikzcd}
 \]

We now turn to the topological analogues of the homology theories from \cref{subsection: symmetric homology}. Recall that the positive topological cyclic homology $\TC^+(R)$ of an $\mathbb{E}_1$-ring spectrum $R$ is the homotopy colimit of the functor 
    \[
    \begin{tikzcd}
        \DC\ar[hook]{r} &  \DS \ar[hook]{r} & \DS_+\simeq \Env(\Assoc) \ar{r}{F_R}  & \Sp\,,
    \end{tikzcd}
    \]
    while topological Hochschild homology $\THH(R)$ is defined in \cite{NS18} to be the homotopy colimit of the functor:
     \[
\begin{tikzcd}
    \Delta^\op \ar[hook]{r} &  \DC^\op\simeq\DC\ar[hook]{r} &  \DS \ar[hook]{r} & \DS_+\simeq \Env(\Assoc) \ar{r}{F_R} & \Sp.
    \end{tikzcd}
    \]
    Note that $\THH(R)$ comes equipped with an $S^1$-action and $\TC^+(R)$ is the $S^1$-homotopy orbits $\THH(R)_{hS^1}$, see \cite[B.5]{NS18} and \cite[4.12, 4.21]{THG}. Unpacking this, note that in~\cite[4.21]{THG} the authors used $\TC^+(R)$ to denote the colimit of the functor
        \[
    \begin{tikzcd}
 \DC^\op\simeq\DC\ar[hook]{r} &  \DS \ar[hook]{r} & \DS_+\simeq \Env(\Assoc) \ar{r}{F_R} & \Sp.
    \end{tikzcd}
    \]
and~\cite[4.12 (iv)]{THG} implies that this colimit agrees with the $S^1$-homotopy orbits $\THH(R)_{hS^1}$.

\begin{definition}[{\cite[Example 4.33]{THG}}]\label{TS-def}
    The \textit{topological symmetric homology $\TS(R)$} of an $\mathbb{E}_1$-ring spectrum $R$ is the homotopy colimit of the functor 
    \[
    \begin{tikzcd}
        \DS \ar[hook]{r} & \DS_+\simeq \Env(\Assoc) \ar{r}{F_R} & \Sp.
    \end{tikzcd}
    \]
    For $n\in \mathbb{Z}$, we write $\TS_n(R)$ for $\pi_n(\TS(R))$. The above composite of functors gives a homotopy coherent analog in spectra of the symmetric bar construction of \cref{symmetric bar construction}.
    The construction is natural in $R$, and defines a functor $\TS\colon \Alg_{\mathbb{E}_1}(\Sp)\to \Sp$ where $\Alg_{\mathbb{E}_1}(\Sp)$ denotes the $\infty$-category of $\mathbb{E}_1$-ring spectra.
\end{definition}

\begin{lemma}[{\cite[Lemma 6.7]{THG}}]\label{lemma:THG}
Given an $\mathbb{E}_1$-ring spectrum $R$,
    the spectrum $\TS(R)$ can equivalently be obtained as the homotopy colimit of $F_R\colon \Env(\Assoc)\to \Sp$, the adjunct of the map of $\infty$-operads $R\colon \Assoc\to \Sp$ defining $R$.
\end{lemma}

\begin{proof}
   By \cite[Theorem 3.17]{THG}, there is an equivalence  $\Env(\Assoc)\simeq \DS_+$, and we show that the inclusion $\DS\to \DS_+$ is homotopy cofinal. By Quillen's Theorem A for $\infty$-categories \cite[Theorem 4.1.3.1]{HTT}, it suffices to show that for every $n\geq -1$, the $\infty$-category $\DS\times_{\DS_+}(\DS_+)_{[n]/}$ is weakly contractible \cite[4.1.1.8]{HTT}.
   If $n\geq 0$, then this category is equivalent to $\DS_{[n]/}$ since morphisms in $\DS_+$ are the same as the ones for $\DS$ between the non-initial objects. The category $\DS_{[n]/}$ is contractible as the identity on $[n]$ is initial. 
   
   If $n=-1$, then the category  $\DS\times_{\DS_+}(\DS_+)_{[-1]/}$ is equivalent to $\DS$. Indeed, an object $([m], [-1]\to [m])$ for $m\geq 0$ in $\DS\times_{\DS_+}(\DS_+)_{[-1]/}$ is precisely determined by picking an object $[m]$ in $\DS$, and any morphism $[m]\to [\ell]$ in $\DS$ is compatible with the unique maps $[-1]\to [m]$ and $[-1]\to [\ell]$. Since $B\DS\simeq *$ by~\cite[Example 6 in Section 1.5]{FL91} and the proof of \cite[Proposition 5.8]{FL91}, the proof is complete.
\end{proof}

In general, topological Hochschild homology and its analogues take as input a ring spectrum, and output a spectrum. We now show that the output of  topological symmetric homology is always an $\mathbb{E}_\infty$-ring spectrum.

\begin{theorem}
\label{Einfty-structure}
Given an $\mathbb{E}_1$-ring spectrum $R$, the topological symmetric homology $\TS(R)$ has a canonical  $\mathbb{E}_\infty$-ring spectrum structure. 
In fact, the functor $\TS\colon \Alg_{\mathbb{E}_1}(\Sp)\to \Sp$ lifts to a functor $\Alg_{\mathbb{E}_1}(\Sp)\to \Alg_{\mathbb{E}_\infty}(\Sp)$ where $\Alg_{\mathbb{E}_\infty}(\Sp)$ is the $\infty$-category of $\mathbb{E}_\infty$-ring spectra. 
\end{theorem}

\begin{proof}
    Using \cref{lemma:THG}, this is a formal argument from Day convolution.  If $\cC$ is a small symmetric monoidal $\infty$-category then the (homotopy) colimit of any lax symmetric monoidal functor $\cC\to \Sp$ is an $\mathbb{E}_\infty$-ring spectrum (see \cite[Proposition 3.16 and Example 3.18]{Keenan-Peroux}).
\end{proof}

\begin{remark}
We expect that this extends the graded commutative algebra structure on symmetric homology $\HS$ produced in~\cite{Aul14} called the Pontryagin product therein.
\end{remark}

In the next section we prove that this $\mathbb{E}_{\infty}$-ring structure is free; that is, we show $\TS\colon \Alg_{\mathbb{E}_1}(\Sp)\to \Alg_{\mathbb{E}_{\infty}}(\Sp)$ is left adjoint to the forgetful functor.

\section{Representation homology}\label{section on representation}
For $k$ a characteristic zero field and a $k$-algebra $R$, the (one-dimensional) representation homology of $R$, $\HR_*(R;k)$, is defined in \cite{BKR13} as a kind of derived abelianization, and has applications to the study of character varieties. In \cite{Berest-Ramadoss}, $\HR_*(R;k)$ is identified with symmetric homology. They also construct a map relating cyclic homology and representation homology, $\HC_*(R/k)\to \HR_*(R;k)$, called the derived 1-dimensional character map. In this section, we develop tools to study generalizations of representation homology in new settings, and prove topological refinements of some results in \cite{Berest-Ramadoss}.

\subsection{The monoidal envelope and representation homology}

In this subsection we recall the construction of the $\cO$-monoidal envelope of $\mathrm{Assoc}$ for an $\infty$-operad $\cO$. For $\cC$ an $\cO$-monoidal $\infty$-category and $R$ an $\mathbb{E}_1$-algebra in $\cC$, we then define the $1$-dimensional $\mathscr{O}$-representation homology of $R$. 

\begin{construction}[{\cite[2.2.4.1]{HA}}]\label{Monoidal envelopes without fibration}
Let $\cO^\otimes$ be an $\infty$-operad. 
Let $\mathrm{Act}(\cO^\otimes)$ denote the full subcategory of $\Fun(\Delta^1, \cO^\otimes)$ spanned by the active morphisms.
Let $f^\otimes\colon \Assoc^\otimes \to \cO^\otimes$ be a fibration of $\infty$-operads \cite[2.1.2.10]{HA}.
The $\cO$-monoidal envelope of $\Assoc$ is the $\cO$-monoidal $\infty$-category $\Env_\cO(\Assoc)^\otimes=\Assoc^\otimes\times_{\Fun(\{0\}, \cO^\otimes)}\mathrm{Act}(\cO^\otimes)$. When $\cO=\mathbb{E}_{\infty}$, we simply write $\Env(\Assoc)^\otimes$. 

If $f^\otimes\colon \Assoc^\otimes\to  \cO^\otimes$ is a map of $\infty$-operads that is not a fibration, we may replace it by a fibration 
using the model structure on $\infty$-preoperads from \cite[2.1.4.6]{HA}. 
Recall that an $\infty$-preoperad is a marked simplicial set $(X,M)$ with a map of simplicial sets $X\to \Fin_\ast$ such that edges $e\in M$ are sent to inert morphisms. 
Let $\Assoc^{\otimes, \natural}$  
and $\cO^{\otimes,\natural}$ be the $\infty$-preoperads obtained from the corresponding $\infty$-operads by marking all the inert morphisms. 
Then we can replace the induced map $f^{\natural}$ by a fibration $\widetilde{f}^\natural\colon \widetilde{\Assoc}^{\otimes, \natural}\to \cO^{\otimes, \natural}$ in the category of $\infty$-preoperads. 
By \cite[2.1.4.6]{HA}, since $\cO^\otimes$ is fibrant as it is an $\infty$-operad, then  $\widetilde{\Assoc}^{\otimes, \natural}$ is also fibrant, and in turn corresponds to an $\infty$-operad $\widetilde{\Assoc}^{\otimes}$. 
Moreover, the fibration of $\infty$-preoperads $\widetilde{f}^\natural$ corresponds to a fibration of $\infty$-operads $$\widetilde{f}^\otimes\colon \widetilde{\Assoc}^{\otimes}\to \cO^\otimes.$$ 
Additionally, we note that since $\Assoc^{\otimes,\natural}\xrightarrow{\simeq} \widetilde{\Assoc}^{\otimes,\natural}$ is a weak equivalence of $\infty$-preoperads between fibrant objects, there is a homotopy inverse, by \cite[B.2.4]{HA}.  
Thus we obtain an equivalence of $\infty$-categories $i^\otimes\colon \Assoc{^\otimes}\xrightarrow{\simeq} \widetilde{\Assoc}^{\otimes}$
which is a map of $\infty$-operads that satisfies $f=\widetilde{f}\circ i$. We define $\Env_\cO(\Assoc)^\otimes$ for any map of $\infty$-operads $f^\otimes\colon \Assoc^\otimes\to \cO^\otimes$ via the pullback diagram
\[
\begin{tikzcd}
    \Env_{\cO}(\Assoc)^\otimes \pull \ar{r} \ar{d} & \mathrm{Act}(\cO^\otimes) \ar{d}{\mathrm{ev}_0} \\
    \widetilde{\Assoc}^\otimes \ar{r}{\widetilde{f}^\otimes} & \cO^\otimes.
\end{tikzcd}
\]
By \cite[2.2.4.4]{HA}, evaluation at $1$ induces a map $\mathrm{ev}_1\colon\mathrm{Act}(\cO^\otimes)\to \cO^\otimes$ and there is a coCartesian fibration of $\infty$-operads $p^\otimes\colon\Env_\cO(\Assoc)^\otimes\to \cO^\otimes$ given by the composite
\[\Env_\cO(\Assoc)^\otimes\longrightarrow \mathrm{Act}(\cO^\otimes)\stackrel{\mathrm{ev}_1}\longrightarrow \cO^\otimes.\] 
The structure map $\cO^\otimes\to \Fin_\ast$ realizes $\Env_{\cO}(\Assoc)^\otimes$ as an $\infty$-operad and thus as an $\cO$-monoidal $\infty$-category. 
The fiber at $\mathbf{1}\in \Fin_\ast$ associated to this map is the underlying $\infty$-category of this $\infty$-operad and is denoted
$$\Env_{\cO}(\Assoc)\coloneqq \Env_{\cO}({\Assoc})^\otimes_{\mathbf{1}}.$$
\end{construction}

\begin{notation}\label{iota functor} 
Denote by $\iota^\otimes\colon \Assoc^\otimes \hookrightarrow \Env_{\cO}(\Assoc)^\otimes$ the map of $\infty$-operads obtained as the composite of inclusions $i^\otimes\colon \Assoc^\otimes \hookrightarrow \widetilde{\Assoc}^\otimes$ and the inclusion $\widetilde{\Assoc}^\otimes\hookrightarrow \Env_\cO(\Assoc)^\otimes$.
\end{notation}

\begin{remark}\label{universal property of envelope}
    By \cite[2.2.4.9]{HA}, if $\cC$ is an $\cO$-monoidal $\infty$-category \cite[2.1.2.13]{HA}, the universal property of $\Env_{\cO}(\Assoc)$ determines an equivalence of $\infty$-categories:
\[
\iota^*\colon \Fun_{\cO}^\otimes(\Env_{\cO}(\Assoc), \cC) \stackrel{\simeq}\longrightarrow \Alg_{\mathbb{E}_1}(\cC) 
\]
sending a strong $\cO$-monoidal functor $F^\otimes\colon \Env_{\cO}(\Assoc)^\otimes\to  \cC^\otimes$ to the map of $\infty$-operads $F^\otimes\circ \iota^\otimes \colon \Assoc^\otimes\to \cC^\otimes$ over $\cO^{\otimes}$. 
Here, when writing $\Alg_{\mathbb{E}_1}(\cC)$, we view $\cC$ as a monoidal $\infty$-category by the pullback of $\cC^\otimes\to \cO^\otimes$ along the map of $\infty$-operads $f^\otimes\colon \Assoc^\otimes\to \cO^\otimes$, it is equivalently the $\infty$-category $\Fun_\cO^\lax(\Assoc, \cC)$ of $\cO$-lax monoidal functors $\Assoc^\otimes\to\cC^\otimes$, i.e.\ a map of $\infty$-operads over $\cO^\otimes$.
Conversely, any map of $\infty$-operads $g^\otimes\colon\Assoc^\otimes \to \cC^\otimes$ 
over $\cO^\otimes$ defines a strong $\cO$-monoidal functor as the composite
\begin{equation}\label{equation 1}
\begin{tikzcd}
    \Env_\cO(\Assoc)^\otimes \ar{r}{\Env_\cO(g)} & [2em] \Env_{\cO}(\cC)^\otimes \ar{r}{\otimes} & \cC^\otimes. 
\end{tikzcd}
\end{equation}
\end{remark}
% \gaknoteil{I think these need to be fibrations over $\mathcal{O}$ and not just over $\mathcal{O}$ as the refereee suggests (or maybe that is what over $\mathcal{O}$ is supposed to mean). Can you confirm? Does that cause issues here? See remark above.}
% \mpnoteil{I think it is now handled}

\begin{notation}\label{strong monoidal functor associated to algebras}
Given any $\mathbb{E}_1$-algebra $R$ in $\cC$, viewed as a map of $\infty$-operads $R^\otimes\colon \Assoc^\otimes\to \cC^\otimes$, we denote by $F_R^{\cO}\colon \Env_\cO(\Assoc)\to \cC$ the underlying functor of \eqref{equation 1}.
When $\cO^\otimes$ is the commutative $\infty$-operad, we denote $F_R^\cO$ simply by $F_R$, recovering our previous notation. Note that $\iota^*{F_R^{\cO}}\simeq R$.
\end{notation}

\begin{construction}\label{free O-algebra and colimits construction}
Let $\kappa$ be an uncountable regular cardinal.
Suppose $\cO^\otimes$ is a $\kappa$-small $\infty$-operad. 
Let $\cC^\otimes$ be an $\cO$-monoidal $\infty$-category which is compatible with $\kappa$-small colimits \cite[3.1.1.19]{HA}.
Let $f^\otimes\colon \Assoc^\otimes\to \cO^\otimes$ be a map of $\infty$-operads.

\begin{itemize}
    \item Precomposition with $f^\otimes$ defines a forgetful functor $f^*\colon \Alg_{\cO}(\cC)\to \Alg_{\mathbb{E}_1}(\cC)$, and we denote its left adjoint by 
   \[  
   f_!\colon \Alg_{\mathbb{E}_1}(\cC)\to \Alg_{\cO}(\cC)
    \]
    which always exists by \cite[3.1.3.5]{HA}.
If $R$ is an $\mathbb{E}_1$-algebra in $\cC$, we refer to $f_!R$ as the \textit{free $\cO$-algebra in $\cC$ generated by $R$.} 

\item The coCartesian fibration $p^\otimes \colon \Env_{\cO}(\Assoc)^\otimes\to \cO^\otimes$ determines a functor $p^*\colon \Fun(\cO, \cC)\to  \Fun(\Env_{\cO}(\Assoc), \cC)$ by precomposing with the underlying functor $p\colon \Env_\cO(\Assoc)\to \cO$, and $p^*$ has a left adjoint $p_!\colon \Fun(\Env_{\cO}(\Assoc), \cC)\to \Fun(\cO, \cC)$ given by the left Kan extension along $p$. 
Given a functor 
\(F\colon \Env_{\cO}(\Assoc)\to\cC\,,\) 
the induced functor $p_!F\colon\cO\to\cC$  is given on objects $X\in\cO$ by
    \[
    (p_! F)(X)=\operatorname{colim}\left
    ( \Env_\cO(\Assoc)\times_\cO \cO_{/X} \to \Env_\cO(\Assoc)\stackrel{F}\to \cC\right).\] 
    Since $p$ underlies a strong $\cO$-monoidal functor, the adjunction $p_! \dashv p^*$ is monoidal using the induced $\cO$-monoidal structure given by Day convolution on the functor categories \cite[3.16]{Keenan-Peroux}. In particular, we obtain an adjunction:
    \[
    \begin{tikzcd}[column sep=large]
\Fun_{\cO}^\lax(\Env_{\cO}(\Assoc), \cC)\ar[bend left=20]{r}{p_!}\ar[phantom, "\perp" description, xshift=-3.5ex]{r} & [2em]\Alg_{\cO}(\cC),\ar[bend left=20]{l}{p^*}
    \end{tikzcd}
\]
where $\Fun_{\cO}^\lax(\Env_{\cO}(\Assoc), \cC)$ denotes the $\infty$-category of lax $\cO$-mono\-idal functors  of the form $\Env_{\cO}(\Assoc)\to \cC$.
\end{itemize}
\end{construction}

\begin{remark}\label{coCart p correspond to f}
    The coCartesian fibration $p^\otimes\colon \Env_\cO(\Assoc)^\otimes\to \cO^\otimes$ seen as a strong $\cO$-monoidal functor corresponds to the map of $\infty$-operads $f^\otimes\colon \Assoc^\otimes\to \cO^\otimes$ using \cref{universal property of envelope} for $\cC=\cO$.
\end{remark}

\begin{remark}
If $\cO^\otimes=\mathbb{E}_n^\otimes$, or more generally if the underlying $\infty$-category $\cO$ is a contractible Kan complex, then the induced functor $p_!F\colon \cO\to \cC$ of a functor $F\colon \Env_{\cO}(\Assoc)\to \cC$ is entirely determined by its value on $*\in \cO$:
    \[
    p_!F(\ast)\simeq \operatorname{colim}\left
    ( \Env_\cO(\Assoc)\stackrel{F}\to \cC\right).
    \]
  \end{remark}

\begin{definition}
Let $\kappa$ be an uncountable regular cardinal.
Suppose $\cO^\otimes$ is a $\kappa$-small $\infty$-operad. 
Let $f^\otimes\colon \Assoc^\otimes\to \cO^\otimes$  be a map of $\infty$-operads. Let $\cC$ be an $\cO$-monoidal $\infty$-category compatible with $\kappa$-small colimits and let $R$ be an $\mathbb{E}_1$-algebra in $\cC$. 
We define the \textit{$1$-dimensional $\mathscr{O}$-representation homology} of $R$ as 
\[ \mathrm{HR}^{\mathscr{O}}(R)\coloneqq f_{!}R\]
and we write $\mathrm{HR}_m^{\mathscr{O}}(R)=\pi_m\mathrm{HR}^{\mathscr{O}}(R)$.
\end{definition}

In the case $\mathscr{O}=\mathrm{Comm}$, this is a topological refinement of the definition of $1$-dimensional representation homology from~\cite[Theorem~2.2]{BKR13}. It is also closely related to the topological Andr\'e--Quillen homology of~\cite{KP17}. 

\subsection{Representation homology as topological symmetric homology}
We now prove that the free $\cO$-algebra generated by an associative algebra $R$ is precisely the colimit determined by its associated strong $\cO$-monoidal functor $F_R^\cO$ from \cref{strong monoidal functor associated to algebras}. The following result is a reformulation of the formula already appearing in \cite[3.1.3.3]{HA}; also see the recent work of \cite{BL26}.

\begin{theorem}\label{representation homology thm}
Let $\kappa$ be an uncountable regular cardinal.
Suppose $\cO^\otimes$ is a $\kappa$-small $\infty$-operad. 
Let $f^\otimes\colon \Assoc^\otimes\to \cO^\otimes$  be a map of $\infty$-operads.
Let $\cC$ be an $\cO$-monoidal $\infty$-category compatible with $\kappa$-small colimits.
Let $R$ be an $\mathbb{E}_1$-algebra in $\cC$, and let $F_R^\cO\colon \Env_{\cO}(\Assoc)\to \cC$ be its associated strong $\cO$-monoidal functor.
Then there is an equivalence of $\cO$-algebras
\[
    f_!R\simeq p_!F_R^\cO
\]
(see \cref{free O-algebra and colimits construction}).
In particular, when $\cO^\otimes=\mathbb{E}_n^\otimes$:
    \[
    \HR^{\mathbb{E}_n}(R) \simeq \operatorname{colim}\left( \Env_{\mathbb{E}_n}(\Assoc)\stackrel{F_R^{\mathbb{E}_n}}\longrightarrow \cC\right).
    \]
\end{theorem}

\begin{proof}
Recall from \cref{universal property of envelope} there is an equivalence
\[
\iota^*\colon \Fun_\cO^\otimes(\Env_\cO(\Assoc), \cC)\stackrel{\simeq}\rightarrow \Fun_\cO^\lax(\Assoc, \cC)\simeq \Alg_{\mathbb{E}_1}(\cC) 
\]
given by pre-composing with the map of $\infty$-operads $\iota^\otimes\colon\Assoc^\otimes\to \Env_\cO(\Assoc)^\otimes$ over $\cO^\otimes$.
This functor is actually the restriction of a functor
$\Fun_\cO^\lax(\Env_\cO(\Assoc), \cC)\to \Fun_\cO^\lax(\Assoc, \cC)$ also denoted by $\iota^*$ which is given by pre-composition with $\iota^\otimes$.

% Denote by $L$ the fully faithful inclusion $\Fun_\cO^\lax(\Assoc,\cC)\simeq \Fun_\cO^\otimes(\Env_\cO(\Assoc), \cC)\hookrightarrow \Fun_\cO^\lax(\Env_\cO(\Assoc), \cC)$, we obtain an adjunction 

The proof of \cite[Proposition~2.2.4.9]{HA} shows that every object of
$\Fun_{\cO}^{\lax}(\Assoc,\cC)$ admits a relative left Kan extension
along
\[
\iota^\otimes\colon
\Assoc^\otimes\to\Env_{\cO}(\Assoc)^\otimes,
\]
and that this relative left Kan extension is strong $\cO$-monoidal. Hence, by \cite[Proposition~4.3.2.17]{HTT}, relative left Kan extension defines a left adjoint $L$ in the adjunction
\[
    \begin{tikzcd}[column sep=large, ampersand replacement=\&]
\Alg_{\mathbb{E}_1}(\cC)  \ar[phantom, "\simeq" description]{r}\& [-3em]\Fun_{\cO}^\lax(\Assoc, \cC)\ar[bend left=20]{r}{L}\ar[phantom, "\perp" description, xshift=1.5ex]{r} \& [2em]\Fun_{\cO}^\lax(\Env_{\cO}(\Assoc), \cC),\ar[bend left=20]{l}{\iota^*}
    \end{tikzcd}
\]
% By the proof of \cite[2.2.4.9]{HA}, the left adjoint $L(R)$ of $\iota^*$ is given by the relative left Kan extension of $R\colon \Assoc^\otimes\to \cC^\otimes$ along $\iota^\otimes\colon\Assoc^\otimes\to \Env_\cO(\Assoc)^\otimes$ over $\cO^\otimes$, which in particular can be identified as $L(R)\simeq F^\cO_R$ as in \cref{strong monoidal functor associated to algebras}. 
Moreover, \cite[Proposition~2.2.4.9]{HA} identifies the essential image of $L$ with the strong $\cO$-monoidal functors. Since $L(R)$ restricts along $\iota^\otimes$ to $R$, the universal property of the envelope and \cref{strong monoidal functor associated to algebras} give a natural equivalence $L(R)\simeq F^\cO_R$.

Suppose $A$ is any $\cO$-algebra in $\cC$. Then from the adjunction 
\[
    \begin{tikzcd}[column sep=large]
\Fun_{\cO}^\lax(\Env_{\cO}(\Assoc), \cC)\ar[bend left=20]{r}{p_!}\ar[phantom, "\perp" description, xshift=-3.5ex]{r} & [2em]\Alg_{\cO}(\cC),\ar[bend left=20]{l}{p^*}
    \end{tikzcd}
\]
we obtain that $\iota^*(p^*(A))\simeq f^*(A)$ as we have the commutative diagram
\[
\begin{tikzcd}
    \Assoc^\otimes \ar{r}{\iota^\otimes} \ar{dr}[swap]{f^\otimes} & \Env_\cO(\Assoc)^\otimes \ar{d}{p^\otimes}\\
    & \cO^\otimes.
\end{tikzcd}
\]
In other words, $\iota^*\circ p^*\simeq f^*$.
Therefore, its left adjoint $p_!\circ L$ is equivalent to $f_!$, i.e.\ $f_!R\simeq p_!F_R^\cO$ as $\cO$-algebras. 
\end{proof}

\begin{remark}
    This result applies more generally: given any fibration of $\infty$-operads $\mathscr{P}^\otimes\to \cO^\otimes$, and any $\mathscr{P}$-algebra $R$ in an $\cO$-monoidal $\infty$-category $\cC$ with compatible colimits, the colimit of the associated strong $\cO$-monoidal functor $\Env_\cO(\mathscr{P})\to \cC$ is equivalent to the free $\cO$-algebra in $\cC$ generated by $R$, see also independent work~\cite{BL26}. 
\end{remark}

One of the main theorems of~\cite{Berest-Ramadoss} identifies $1$-dimensional representation homology with symmetric homology over a field of characteristic zero. The following result is a topological refinement of this theorem. Notably, in the topological setting we can remove the characteristic zero assumption. This proves a topological analogue of a conjecture of Ault, see~\cite[Conjecture 2]{Berest-Ramadoss}. 

\begin{corollary}\label{topological symmetric homology is topological representation homology}
Let $R$ be an $\mathbb{E}_1$-ring spectrum.
There is a natural equivalence of $\mathbb{E}_{\infty}$-ring spectra
\[ 
\TS(R)\simeq \mathrm{HR}^{\mathbb{E}_{\infty}}(R) \,.
\]
\end{corollary}

\begin{proof}
By \cref{representation homology thm} and \cref{lemma:THG},  we obtain an equivalence of spectra $\TS(R)\simeq \mathrm{HR}^{\mathbb{E}_{\infty}}(R)$. Moreover, the $\mathbb{E}_\infty$-structure on $\TS(R)$ constructed in \cref{Einfty-structure}, and that on $p_!F_R$ constructed in \cref{free O-algebra and colimits construction} are induced by the same Day-convolution left Kan extension of $F_R$; see \cite[Proposition~3.16 and Example~3.18]{Keenan-Peroux}. Consequently, the preceding equivalence of underlying spectra
refines naturally to an equivalence in $\Alg_{\mathbb{E}_\infty}(\Sp)$.
\end{proof}

Thus, $\TS\colon \Alg_{\mathbb{E}_1}(\Sp)\to \Alg_{\mathbb{E}_\infty}(\Sp)$ is a left adjoint to  the forgetful functor. This gives a refinement of the identification in \cite[Corollary~6.2]{THG}. 

\begin{remark}\label{rem-formula-TS}
Let $\mathrm{Free}_{\mathbb{E}_1}$ denote the left adjoint of the forgetful functor $\mathrm{Alg}(\Sp)\to \Sp$. Unpacking the Bousfield--Kan formula for the colimit in \cref{TS-def}, we have an equivalence 
\[
\TS(\mathrm{Free}_{\mathbb{E}_1}(X)) \simeq \coprod_{n\ge -1}([n]\downarrow \DS)\times \mathrm{Free}_{\mathbb{E}_1}(X)^{\times n+1}/\sim   
\]
and as a special case of~\cite[Lemma~6.8]{THG}, we can identify $([n]\downarrow \DS)=E\Sigma_{n+1}$\,. 
Since every map in $\DS$ factors as a composite of an element in the symmetric group followed by a map in $\Delta$,  
we have an explicit presentation
\[
\TS(\mathrm{Free}_{\mathbb{E}_1}(X))\simeq \coprod_{n\ge -1}\mathbb{E}_{\infty}(n+1)\times_{\Sigma_{n+1}} X^{\times n+1}
\]
of the homotopy colimit defining $\TS(\mathrm{Free}_{\mathbb{E}_1}(X))$, cf.~\cite[Proposition~6.10]{THG}. Since both sides commute with sifted colimits, one could use a free resolution of a more general $\mathbb{E}_1$-ring $R$ to give an independent proof of \cref{topological symmetric homology is topological representation homology}. 
\end{remark}

In the following corollary we write $QY \coloneqq \Omega^{\infty}\Sigma^{\infty}Y$ for a pointed space $Y$ and we write $BM$ for the delooping of $M$.  

\begin{corollary}\label{theorem: TS of M as a spectrum}
    Let $M$ be a group-like $\mathbb{E}_1$-space.  There is an equivalence of $\mathbb{E}_\infty$-ring spectra: $\TS(\Sigma^\infty_+ M)\simeq \Sigma^\infty_+ \Omega QBM$.  
\end{corollary}
In particular, this implies the following in light of May's recognition theorem~\cite{May72}. 

\begin{corollary}\label{theorem: TS of Omega X as a spectrum}
    Let $X$ be a pointed connected space. There is an equivalence of $\mathbb{E}_\infty$-ring spectra: $\TS(\Sigma^\infty_+ \Omega X)\simeq \Sigma^\infty_+ \Omega QX$.  
\end{corollary}

We can also reproduce a topological analogue of the character map studied in~\cite{BKR13,Berest-Ramadoss}. 

\begin{construction}\label{prop:character-map}
Let $R$ be an $\mathbb{E}_1$-ring spectrum. Then there is a canonical map of spectra
\[ 
\THH(R)\to \TS(R)\,.
\]
that factors through $\TC^{+}(R)$. To see this, note that since $\TS(R)$ is the free $\mathbb{E}_{\infty}$-ring on $R$, there is a unit map of $\mathbb{E}_1$-ring spectra $\eta : R\to \TS(R)$. Applying $\THH$ we obtain a composite 
\begin{equation}\label{eq: composite}
\THH(R)\xrightarrow{\THH(\eta)} \THH(\TS(R))\xrightarrow{\epsilon^{\THH}}\TS(R)
\end{equation}
where the second map $\varepsilon^{\THH}$ is the augmentation which arises from the fact that $\TS(R)$ is an $\mathbb{E}_{\infty}$-ring.  Explicitly, this map is obtained by considering a map of simplicial spectra from the cyclic bar construction on $\TS(R)$ to the constant simplicial spectrum on $\TS(R)$. Levelwise, this map is given by the multiplication of $\TS(R)$. The composite \eqref{eq: composite} is $S^1$-equivariant, where the target has trivial action, and therefore it factors through $\TC^+(R)=\THH(R)_{hS^1}$.
\end{construction}
\begin{remark}
There is another natural map 
\( \Xi_R \colon \THH(R)\to \TS(R) \)
that factors through $\TC^+(R)$ induced by the functors 
\( \Delta^{\op}\to \DC^{\op}\to \DS\),
and this agrees with the one constructed in \cref{prop:character-map} as maps of spectra by the commuting diagram 
\[
\begin{tikzcd}
\THH(R) \arrow[r,"\THH\circ \eta"] \arrow[d,"\Xi_R"] & \THH(\TS(R)) \arrow[d,"\Xi_{\TS(R)}"] \arrow[r,"\epsilon^{\THH}"] &  \TS(R)  \\
\TS(R) \arrow[r,"\TS \circ \eta"] & \TS(\TS(R)) \arrow[ur,"\epsilon \circ \TS",swap]  & 
\end{tikzcd}
\]
where the top composite is the topological character map from \eqref{eq: composite}, the left vertical arrow is the map induced by the the functors \( \Delta^{\op}\to \DC^{\op}\to \DS\) and the map 
\[ \TS(R)\to \TS(\TS(R)) \to \TS(R) \]
is the identity map by the triangle identity for adjunctions. Unpacking this further, the left square is produced by applying the natural transformation $\Xi :\THH\to \TS$ induced by the composite of functors $\Delta^{\op}\to \Delta C^{\op}\to \Delta \Sigma$ to the map $\TS\circ \eta : R\to \TS(R)$ of $\mathbb{E}_1$-rings produced from the unit of the adjunction. Since $\TS(R)$ is an $\mathbb{E}_\infty$-ring, there are are maps induced by the maps to the constant functor, which are compatible with $\Xi$ by construction. 
The authors would like to thank an anonymous referee for pointing out that these two maps in fact agree.  
\end{remark}

A similar map was also considered by Fiedorowicz~\cite[Remark~1.4]{Fie}. 

\section{Topological braid homology}

In this section we study topological braid homology, which is associated to the crossed simplicial group $\DB$ that extends $\Delta$ via the addition of the braid groups. 

\subsection{Braid homology}
We recall here the braided category $\DB$ and its associated homologies.  Let $B_n$ denote the braid group on $n$ strands.  We write $\pi\colon B_n\to \Sigma_n$ for the homomorphism which sends a braid to the underlying permutation of vertices.

\begin{definition}

The \emph{braided category}  $\DB$ \cite[3.8]{FL91} is a category with objects $[n]$ for $n\geq 0$, and morphisms generated by the coface and codegeneracy maps in $\Delta$, $\delta^i_n\colon [n-1]\to [n]$ and $\sigma^i_n\colon [n+1]\to [n]$ for $0\leq i \leq n$, as well as maps $t^i_n\colon [n]\to [n]$ for $0\leq i \leq n-1$ viewed as the generators of $B_{n+1}$ (elementary braids on $n+1$ strands in which the $i$-th strand crosses once over the $(i+1)$-st strand while all other strands stay straight). These morphisms are subject to the usual cosimplicial identities, Artin relations (i.e., $t^i_n t^j_n=t^j_n t^i_n$ for $|i-j|>1$ and $t^i_n t^{i+1}_n t^i_n=t^{i+1}_n t^i_n t^{i+1}_n$), as well as the same mixed relations \eqref{eq:mixes relations} as in $\DS$.
\end{definition}

\begin{remark}
    Every morphism $[n]\to [m]$ in $\DB$ factors uniquely as a composite $\phi\circ b$ where $b\in B_{n+1}^\op$ and $\phi\colon [n]\to [m]$ in $\Delta$. By \cite[Proposition 1.6]{FL91}, the category $\DB$ is entirely determined by functions $b^*\colon \Hom_\Delta([m], [n])\to \Hom_\Delta([m], [n])$ for each $b\in B_{n+1}$, and functions $\phi^*\colon B_{n+1}\to B_{m+1}$ for each $\phi\colon [m]\to [n]$ in $\Delta$, which we now recall. 

    For a braid $b\in B_{n+1}$, the operation $b^\ast\in \Hom_\Delta([m], [n])$ is given as follows. For an ordered map $\phi\colon [m]\to [n]$, the ordered map $b^\ast\phi$ is exactly $\pi(b)^\ast\phi$ from the definition of $\DS$. 

    For a map $\phi\colon [m]\to [n]$ in $\Delta$, we define the operation $\phi^\ast\colon B_{n+1}\to B_{m+1}$ as follows. To a braid $b\in B_{n+1}$, corresponding to a diagram with $n+1$ strands, we define $\phi^*b$ by first replacing the $i$-th strand in $b$ by $|\phi^{-1}(i)|$ parallel strands if $\phi^{-1}(i)$ is nonempty. We refer to these parallel strands as a block. If $\phi^{-1}(i)$ is empty, we omit this strand when forming $\phi^*b$. We then keep these blocks together, and define the crossings so that every strand in the $i$-th block of $\phi^*b$ crosses over/under each strand in the $j$-th block of $\phi^*b$ exactly as the original strand $i$ in $b$ crosses over/under the strand $j$ in $b$. Then $\phi^\ast(b)$ is the resulting braid on $m+1$ strands. See~\cref{figure-braid-operation} for an example.  
\end{remark}

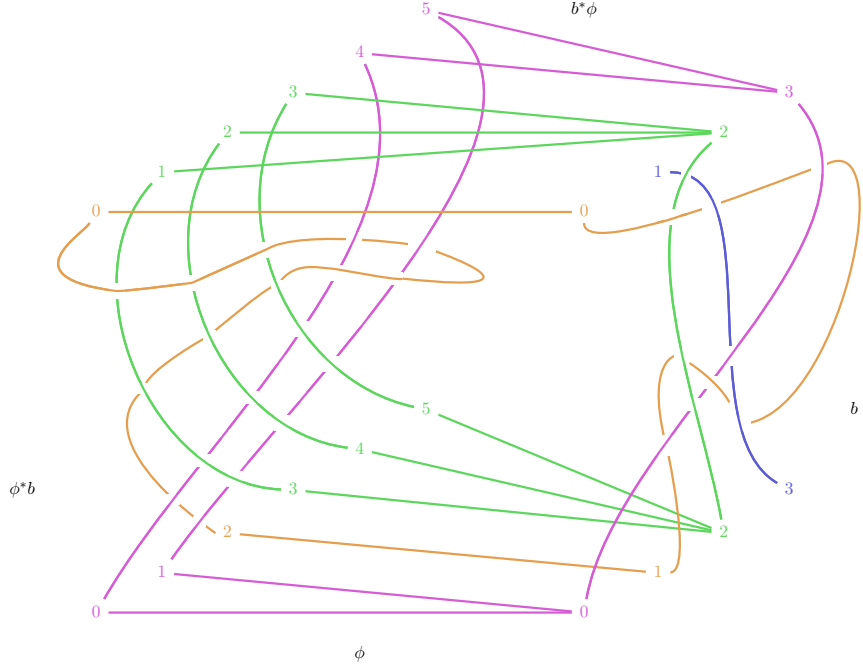
\begin{figure}[ht!]
\begin{tikzpicture}[baseline= (a).base]
\node[scale=0.6] (a) at (1,1){
\begin{tikzcd}[row sep=small]
	&&&&&& |[alias=TL5]|\textcolor{rgb,255:red,214;green,92;blue,214}{5} &&& {b^\ast \phi} \\
	&&&&& |[alias=TL4]|\textcolor{rgb,255:red,214;green,92;blue,214}{4} \\
	&&&& |[alias=TL3]|\textcolor{rgb,255:red,92;green,214;blue,92}{3} &&&&&&&& |[alias=TR3]|\textcolor{rgb,255:red,214;green,92;blue,214}{3} \\
	&&& |[alias=TL2]|\textcolor{rgb,255:red,92;green,214;blue,92}{2} &&&&&&&& |[alias=TR2]|\textcolor{rgb,255:red,92;green,214;blue,92}{2} \\
	&& |[alias=TL1]|\textcolor{rgb,255:red,92;green,214;blue,92}{1} &&&&&&&& |[alias=TR1]|\textcolor{rgb,255:red,92;green,92;blue,214}{1} \\
	& |[alias=TL0]|\textcolor{rgb,255:red,228;green,158;blue,78}{0} &&&&&&&& |[alias=TR0]|\textcolor{rgb,255:red,228;green,158;blue,78}{0} \\
&&&&&&&\\
&&&&&&& \\
&&&&&&&&\\
&&&&&&&&\\
&&&&&&&\\
&&&&&&&&\\
&&&&&&&&\\
&&&&&&&&\\
&&&&&&&&&\\
&&&&&&&&&\\
	&&&&&& |[alias=BL5]|\textcolor{rgb,255:red,92;green,214;blue,92}{5} &&&&&& & {b} \\
	&&&&& |[alias=BL4]|\textcolor{rgb,255:red,92;green,214;blue,92}{4} \\
	{\phi^\ast b} &&&& |[alias=BL3]|\textcolor{rgb,255:red,92;green,214;blue,92}{3} &&&&&&&& |[alias=BR3]|\textcolor{rgb,255:red,92;green,92;blue,214}{3} \\
	&&& |[alias=BL2]|\textcolor{rgb,255:red,228;green,158;blue,78}{2} &&&&&&&& |[alias=BR2]|\textcolor{rgb,255:red,92;green,214;blue,92}{2} \\
	&& |[alias=BL1]|\textcolor{rgb,255:red,214;green,92;blue,214}{1} &&&&&&&& |[alias=BR1]|\textcolor{rgb,255:red,228;green,158;blue,78}{1} \\
	& |[alias=BL0]|\textcolor{rgb,255:red,214;green,92;blue,214}{0} &&&&&&&& |[alias=BR0]|\textcolor{rgb,255:red,214;green,92;blue,214}{0} \\
	&&&&& {\phi}
    %%%%%%%%LEFT STRANDS%%%%%%%%%%%%
\arrow[draw={rgb,255:red,92;green,214;blue,92}, no head, from=TL1, to=BL3, line width=1.4pt, bend right=65, ""{pos=0.3, coordinate, name=C}, ""{pos=0.55, coordinate, name=J}] %Left green 1
\arrow[draw={rgb,255:red,92;green,214;blue,92}, no head, from=TL2, to=BL4, line width=1.4pt, bend right=60, ""{pos=0.38, coordinate, name=D}, ""{pos=0.52, coordinate, name=I}] %Left green 2
	\arrow[draw={rgb,255:red,92;green,214;blue,92}, no head, from=TL3, to=BL5, line width=1.4pt, bend right=55, ""{pos=0.4, coordinate, name=E}, ""{pos=0.5, coordinate, name=H}] % Left green 3
\arrow[draw={rgb,255:red,214;green,92;blue,214}, no head, from=TL4, to=BL0, line width=1.4pt, looseness=0.8, in=60,out =-65, ""{pos=0.8, coordinate, name=X}, ""{pos=0.38, coordinate, name=G}, ""{description,pos=0.68, fill=white, inner sep=3pt}, ""{description,pos=0.545, fill=white, inner sep=3pt}, ""{description,pos=0.39, fill=white, inner sep=3pt}]% Left rose 4
\arrow[draw={rgb,255:red,214;green,92;blue,214}, no head, from=TL5, to=BL1, line width=1.4pt, looseness=0.8, in=55,out =-30, ""{pos=0.47, coordinate, name=F}, ""{description,pos=0.825, fill=white, inner sep=3pt}, ""{description,pos=0.708, fill=white, inner sep=3pt}, ""{description,pos=0.59, fill=white, inner sep=3pt}, ""{pos=0.52, coordinate, name=Y}] %Left rose 5
\arrow[draw={rgb,255:red,228;green,158;blue,78}, no head, from=TL0, to=Y, line width=1.4pt, to path={
    .. controls +(-0.3,-0.5) and +(-2.4,0.3) .. (C)
    .. controls +(0.3,0) and +(0,0) .. (D)
    .. controls +(0, 0) and +(0:0) .. (E)
    .. controls +(0.8,0.5) and +(0.4, 0) .. node[pos=.46, fill=white, inner sep=2pt] {\phantom{x}} (F)
     .. controls +(-1.1,0.4) and +(4.5,-0.5) .. node[pos=.01, fill=white, inner sep=3pt, yshift=2pt, xshift=2pt] {\phantom{xx}} (\tikztotarget)
    \tikztonodes}]
    \arrow[draw={rgb,255:red,228;green,158;blue,78}, no head, from=G, to=BL2, line width=1.4pt, to path={
    .. controls +(-0.2,-0.12) and +(0.4,0.7) .. (H)
     .. controls +(0,0.2) and +(0,0) .. node[pos=.01, fill=white, inner sep=1pt] {\phantom{x}}(I)
     .. controls +(0:0) and +(0,0.4) .. node[pos=.99, fill=white, inner sep=1pt] {\phantom{x}} node[pos=.01, fill=white, inner sep=1pt] {\phantom{x}}(J)
    .. controls +(-1.4,-1.2) and +(-0.25, 0) .. node[pos=.78, fill=white, inner sep=1pt] {\phantom{x}} node[pos=.61, fill=white, inner sep=1pt] {\phantom{x}}(\tikztotarget)
    \tikztonodes
  }]
      %Left brown
  \arrow[draw={rgb,255:red,214;green,92;blue,214}, no head, from=TL4, to=BL0, line width=1.4pt, looseness=0.8, in=60,out =-65, ""{pos=0.8, coordinate, name=X}, ""{description,pos=0.73, fill=white, inner sep=4pt}, ""{description,pos=0.615, fill=white, inner sep=4pt}, ""{description,pos=0.49, fill=white, inner sep=5pt}, ""{description,pos=0.39, fill=white, inner sep=5pt}, ""{pos=0.258, coordinate, name=Q}]% Left rose 4 %started working on perspective but abandoned: ""{description,pos=0.07, fill=white, inner sep=5pt}, ""{description,pos=0.13, fill=white, inner sep=3pt}, ""{description,pos=0.18, fill=white, inner sep=5pt} ,""{description,pos=0.28, fill=white, inner sep=5pt}
\arrow[draw={rgb,255:red,214;green,92;blue,214}, no head, from=TL5, to=BL1, line width=1.4pt, looseness=0.8, in=55,out =-30, ""{pos=0.47, coordinate, name=F}, ""{description,pos=0.845, fill=white, inner sep=5pt}, ""{description,pos=0.75, fill=white, inner sep=5pt}, ""{description,pos=0.656, fill=white, inner sep=5pt}, ""{description,pos=0.52, fill=white, inner sep=4pt, xshift=-1}] %Left rose 5 %%%%FIX ROSE
\arrow[draw={rgb,255:red,228;green,158;blue,78}, no head, shorten <=-3, shorten >=-7, bend right=5, from=Y, to=G, out=20, looseness=0.2, yshift=-1, line width=1.4pt]
%%%%OLD
% \arrow[draw={rgb,255:red,228;green,158;blue,78}, no head, from=F, to=G, line width=1.4pt, to path={.. controls +(-1.1,0.7) and +(7,-0.5) .. (\tikztotarget)
%     \tikztonodes}]
\arrow[draw={rgb,255:red,92;green,214;blue,92}, no head, from=TL1, to=BL3, line width=1.4pt, bend right=65, ""{pos=0.3, coordinate, name=C}, ""{pos=0.55, coordinate, name=J}, ""{description,pos=0.3, fill=white, inner sep=5pt}] %Left green 1
\arrow[draw={rgb,255:red,92;green,214;blue,92}, no head, from=TL2, to=BL4, line width=1.4pt, bend right=60, ""{pos=0.38, coordinate, name=D}, ""{pos=0.52, coordinate, name=I}, ""{description,pos=0.38, fill=white, inner sep=5pt}] %Left green 2
	\arrow[draw={rgb,255:red,92;green,214;blue,92}, no head, from=TL3, to=BL5, line width=1.4pt, bend right=55, ""{pos=0.42, coordinate, name=H}, ""{description,pos=0.4, fill=white, inner sep=5pt}] % Left green 3
    \arrow[draw={rgb,255:red,228;green,158;blue,78}, no head, from=TL0, to=E, line width=1.4pt, to path={
    .. controls +(-0.3,-0.5) and +(-2.4,0.3) .. (C)
    .. controls +(0.3,0) and +(0,0) .. (D)
    .. controls +(0:0) and +(0:0) .. (\tikztotarget)
    \tikztonodes}]
    \arrow[draw={rgb,255:red,228;green,158;blue,78}, no head, shorten <=-1, shorten >=66, from=E, to=Q, line width=1.4pt]
%%%%%%%%%% other strands%%%%%%%
\arrow[draw={rgb,255:red,214;green,92;blue,214}, no head, from=TL5, to=TR3, line width=1.4pt] %% pink top
	\arrow[draw={rgb,255:red,214;green,92;blue,214}, no head, from=TL4, to=TR3, line width=1.4pt] %pink top
	\arrow[draw={rgb,255:red,92;green,214;blue,92}, no head, from=TL2, to=TR2, line width=1.4pt] %green topfroom2
	\arrow[draw={rgb,255:red,92;green,214;blue,92}, no head, from=TR2, to=TL3, line width=1.4pt]%green top from 
	\arrow[draw={rgb,255:red,92;green,214;blue,92}, no head, from=TL1, to=TR2, line width=1.4pt] %green top
	\arrow[draw={rgb,255:red,228;green,158;blue,78}, no head, from=TL0, to=TR0, line width=1.4pt] %orannge top
	\arrow[draw={rgb,255:red,92;green,214;blue,92}, no head, from=BL5, to=BR2, line width=1.4pt] %green bottom
	\arrow[draw={rgb,255:red,92;green,214;blue,92}, no head, from=BL4, to=BR2, line width=1.4pt] %green bottom
	\arrow[draw={rgb,255:red,92;green,214;blue,92}, no head, from=BL3, to=BR2, line width=1.4pt] %green bottom
    \arrow[draw={rgb,255:red,228;green,158;blue,78}, no head, from=BL2, to=BR1, line width=1.4pt] %orange bottom
    \arrow[draw={rgb,255:red,214;green,92;blue,214}, no head, from=BL1, to=BR0, line width=1.4pt] %pink bottom
    \arrow[draw={rgb,255:red,214;green,92;blue,214}, no head, from=BL0, to=BR0, line width=1.4pt]
%%%%%%%%%%%%%%%%%%%%%%%%RIGHT STRANDS%%%%%%%%%%%%%%%%%%%%%
	\arrow[draw={rgb,255:red,92;green,92;blue,214}, no head, from=BR3, to=TR1, line width=1.4pt, ""{pos=0.22, coordinate, name=B},  looseness=0.8, in=0,out =150]%right blue
%\arrow[draw={rgb,255:red,92;green,92;blue,214}, no head, from=BR3, to=TR1, thick, ""{description,pos=0.62, fill=white, inner sep=2pt}]
    \arrow[draw={rgb,255:red,214;green,92;blue,214}, no head, from=BR0, to=TR3, line width=1.4pt, looseness=0.8, in=-50,out =80,""{pos=0.85, coordinate, name=Z}, ""{pos=0.3, coordinate, name=T}] %%right pink
	\arrow[draw={rgb,255:red,92;green,214;blue,92}, no head, from=BR2, to=TR2, line width=1.4pt, ""{pos=0.4, coordinate, name=A},looseness=0.8, in=-140,out =100 ] %% right green
    %bLUE FIX \arrow[draw={rgb,255:red,92;green,92;blue,214}, no head, from=BR3, to=TR1, line width=1.4pt, shorten <=52pt, shorten >=49pt] %%% blue fix
	\arrow[draw={rgb,255:red,228;green,158;blue,78}, no head, from=BR1, to=A, line width=1.4pt, ""{description,pos=0.12, fill=white, inner sep=5pt}, to path={
    .. controls +(2:0.8) and +(90:0.1) .. node[pos=.9, fill=white, inner sep=4pt] {\phantom{x}} (T)
    .. controls +(-0.5,2) and +(90:0.2) .. node[pos=.9, fill=white, inner sep=2pt] {\phantom{x}}(A)
    }]
    \arrow[draw={rgb,255:red,228;green,158;blue,78}, no head, from=B, to=TR0, line width=1.4pt, ""{description,pos=0.12, fill=white, inner sep=5pt}, to path={
    .. controls +(0.5:2.2cm) and +(1.8, 1.1) .. node[pos=.01, fill=white, inner sep=2pt] {\phantom{x}}(Z)
    .. controls +(0.5,0.3) and +(-90:1.2cm) .. node[pos=.5, fill=white, inner sep=3pt]{\phantom{x}}(\tikztotarget)
    \tikztonodes
  }
  ] %%%right brown
  \arrow[draw={rgb,255:red,92;green,92;blue,214}, no head, from=BR3, to=TR1, line width=1.4pt, looseness=0.8, in=0,out =150,  ""{description,pos=0.945, fill=white, inner sep=3pt}, ""{description,pos=0.39, fill=white, inner sep=6pt}]%right blue fixxxxx deleted ""{description,pos=0.41, fill=blue, inner sep=3pt},
  \arrow[draw={rgb,255:red,214;green,92;blue,214}, no head, from=BR0, to=TR3, line width=1.4pt, looseness=0.8, in=-50,out =80, ""{description,pos=0.415, fill=white, inner sep=4pt}, ""{description,pos=0.373, fill=white, inner sep=4pt}] %%FIX right pink over blue
  \arrow[draw={rgb,255:red,92;green,214;blue,92}, no head, from=BR2, to=TR2, line width=1.4pt, ""{pos=0.4, coordinate, name=A},looseness=0.8, in=-140,out =100, ""{description,pos=0.75, fill=white, inner sep=5pt}] %% Fix right green over pink
    \arrow[draw={rgb,255:red,228;green,158;blue,78}, no head, from=A, to=B, bend left=15, shorten <=5, shorten >=15, line width=1.4pt] %% fix brown crossing with pink
    \arrow[draw={rgb,255:red,228;green,158;blue,78}, no head, from=TR1, to=TR2, yshift=-33.5, xshift=-8, bend right=18, shorten >=15.5, line width=1.4pt] %% fix brown crossing with green
\end{tikzcd}
};  
\end{tikzpicture}
\caption{Depiction of the operations $\phi^\ast$ and $b^\ast$}
\label{figure-braid-operation}
\centering
\end{figure}

The homomorphisms $\pi\colon B_n\to \Sigma_n$ assemble into a functor $\DB\to \DS$ which is the identity on objects and the identity on the subcategory $\Delta$.

\begin{definition}[{\cite{Fie}}]\label{defn: braid homology}
Let $k$ be a field.
    The \textit{braided bar construction} of a $k$-algebra $R$, denoted $\BDB(R)$, is defined as the functor $\DB\to \Vect_k$ obtained by precomposing the symmetric bar construction $\BDS(R)\colon \DS\to \Vect_k$  of \cref{symmetric bar construction} with the natural projection $\DB\to \DS$:
    \[
\begin{tikzcd}
    R  \arrow[r, altstackar=3] & R^{\otimes 2} \ar[out=120, in=60, loop, looseness=6]{}{B_2} \arrow[altstackar=5]{r} & R^{\otimes 3} \ar[out=120, in=60, loop, looseness=6]{}{B_{3}} \arrow[altstackar=7]{r} & \cdots .
\end{tikzcd}
\]
In particular, $\BDB(R)$ and $\BDS(R)$ have the same underlying cosimplicial object. 
The \textit{braid homology $\HB_*(R)$ of a $k$-algebra $R$} is defined as 
\[
    \HB_n(R)=\Tor_n^\DB(\underline{k}, \BDB(R))\,
\]
where $\underline{k}\colon \DB^\op\to \Vect_k$ denotes the constant diagram at $k$.
\end{definition}

\begin{definition}{\cite[Example 4.34]{THG}}\label{TB-def}
   The \textit{topological braid homology $\TB(R)$} of an $\mathbb{E}_1$-ring spectrum $R$ is defined as the homotopy colimit of the functor
    \[
    \begin{tikzcd}
        \BDB(R) \colon \DB \ar{r} & \DS \ar[hook]{r} & \DS_+\simeq \Env(\Assoc) \ar{r}{F_R} & \Sp.
    \end{tikzcd}
    \]
    The above composite of functors gives a homotopy coherent analog in spectra of the braided bar construction of \cref{defn: braid homology}.
    Note that $\TB$ is the special case of topological $\DG$-homology $\mathrm{T}\mathbf{G}^+(R)$ where $\DG=\DB$ of \cref{def: DG-homology} below. 
    The construction is natural in $R$ and defines a functor $\TB\colon \Alg_{\mathbb{E}_1}(\Sp)\to \Sp$.
\end{definition}

\subsection{Universal property of braided envelopes}\label{section: braided envelopes}

As discussed in \cref{section: symmetric homology}, the symmetric category $\DS_+$ is equivalent to the symmetric monoidal envelope of the associative operad, and it  has the symmetric monoidal analogue of the universal property of the terminal category regarded as a monoidal category as described in \cref{sec:TS}. 
In this section we establish the braided analogue. Abstractly, the universal property of the augmented braid category $\Delta B_+$ and that of the $\mathbb{E}_2$-monoidal envelope imply, once we show that the latter is 1-categorical, that the two categories are equivalent. Our main purpose in the next section is to unravel this equivalence concretely. The explicit identification of the two categories will also allow us to compare the canonical map out of the $\mathbb{E}_2$-monoidal envelope with the projection $\Delta B_+\to \Delta \Sigma_+$.

\begin{definition}\label{augmented-braid}
    Let $\DB_+$ be the augmented braided category: it is the braided category $\DB$ with an additional, initial object $[-1]$. We will sometimes write $\langle 0\rangle$ for $[-1]$, in analogy with our notation $\langle n\rangle$ for the object $[n-1]$. Using the same notation as in \cref{def: augmented simplex cat}, there is a monoidal structure on $\DB_+$, extending that of $\Delta_+$ such that the inclusion $\Delta_+\hookrightarrow \DB_+$ is strictly monoidal. Moreover, the monoidal category $\DB_+$ is braided in the sense that for any $n,m\geq 0$, there is a natural isomorphism $$b_{n,m}\colon \langle n\rangle \sqcup \langle m \rangle\to \langle m\rangle \sqcup \langle n\rangle $$ satisfying the usual hexagon identities. There are two natural options for the braid $b_{n,m}$.  We make the choice, working from bottom to top, that the strands on the left go over the strands on the right. For example, the braid $b_{2,3}\colon \langle 2\rangle \sqcup \langle 3 \rangle\to \langle 3\rangle \sqcup \langle 2\rangle $ is 
    \[
    \begin{tikzcd}[column sep=small]
	\bullet & \bullet & \bullet & \bullet & \bullet \\
	\\
	\bullet & \bullet & \bullet & \bullet & \bullet\,.
	\arrow[from=1-1, to=3-4, no head]
	\arrow[from=1-2, to=3-5, no head]
	\arrow[from=1-3, to=3-1,crossing over, no head]
	\arrow[from=1-4, to=3-2,crossing over, no head]
	\arrow[from=1-5, to=3-3,crossing over, no head]
\end{tikzcd}
\]
\end{definition}

\begin{observation}\label{braided-observation}
    The braided monoidal structure on $\DB_+$ is universal in the following sense. 
    % Given any braided monoidal category $\cC$, the category of strong braided monoidal functors $\DB_+ \to \cC$ and braided monoidal natural transformations is equivalent to the category of (non-braided) monoid objects in $\cC$, see for instance \cite[6.2]{Graves22}. 
    % Moreover, 
    There is an adjunction
    \[
    \begin{tikzcd}[column sep=large]
    \mathrm{MonCat}^\mathrm{strong}\ar[bend left]{r}{L}\ar[phantom, "\perp" description, xshift=0ex]{r} & \mathrm{BrMonCat}^\mathrm{BrSt},\ar[bend left]{l}{U}
    \end{tikzcd}
\]
where $\mathrm{BrMonCat}^\mathrm{BrSt}$ is the category of small braided monoidal categories with strong braided monoidal functors. Here $U$ is the forgetful functor, and the left adjoint sends $\Delta_+$ to $\Delta B_+$ and hence, for every
braided monoidal category $\cC$:
\[
\begin{aligned}
\Fun_\mathrm{br}^\otimes(\Delta B_+,\cC)
&\simeq
\Fun_\mathrm{br}^\otimes(L(\Delta_+),\cC)
\simeq
\mathrm{Fun}^\otimes(\Delta_+,U\cC)
\simeq
\mathrm{Mon}(\cC),
\end{aligned}
\]
giving an equivalence between the category of strong braided monoidal functors and braided monoidal natural transformations and the category of monoid objects in $\cC$ (see~\cite[6.2]{Graves22}).
\end{observation}

This can be used to identify $\Delta B_+$ with the $\mathbb{E}_2$-monoidal envelope of $\mathrm{Assoc}$. 

\begin{theorem}\label{thm:envelope}
There is an $\mathbb{E}_2$-monoidal equivalence of $\infty$-categories  
\[ \mathrm{Env}_{\mathbb{E}_2}(\mathrm{Assoc})\simeq \DB_+ \,.
\] 
\end{theorem}

Before proving the the theorem,  we recall the envelope construction in this case. First, we choose a model for the $\mathbb{E}_2$-operad in simplicial sets. Let $\Sigma_n//B_n$ denote the translation groupoid of the action of $B_n$ on $\Sigma_n$ through the canonical homomorphism $\pi\colon B_n\to \Sigma_n$.  That is, the objects of this groupoid are the elements of $\Sigma_n$, and a morphism $\sigma\to \tau$ is a element $b\in B_n$ such that $\pi(b)\sigma=\tau$.  Composition is multiplication in $B_n$.  We will sometimes use the pair $(\sigma,b)$ to denote the morphism $\sigma\to \pi(b)\sigma$ determined by $b$.

The groupoids $\Sigma_n//B_n$ form a categorical operad in the sense of \cite{wahl}. The operad structure maps are functors
$$\Sigma_n//B_n \times \Sigma_{j_1}//B_{j_1} \times \dots \times \Sigma_{j_n}//B_{j_n}\longrightarrow \Sigma_{j_1+\dots +j_n}//B_{j_1+\dots +j_n},$$
defined as follows. On objects each of these functors is given by the same formula as for $\Assoc$: for $\sigma\in \Sigma_n$ and $\tau_i\in \Sigma_{j_i}$, the operad action map is defined as 
$$(\sigma, \tau_1, \dots, \tau_n)\mapsto \sigma(j_1,\dots,j_n)\cdot (\tau_1\oplus.\dots\tau_n),$$ where $\sigma(j_1,\dots,j_n)$ is the permutation in $\Sigma_{j_1+\dots+ j_n}$ which permutes the $n$ blocks of lengths $j_1, \dots, j_n$ according to $\sigma$. On morphisms, the operad structure map is given by
$$\big( (\sigma, b), (\tau_1, b_1), \dots, (\tau_n, b_n)\big)\longmapsto \omega(b, b_{\sigma^{-1}(1)}, \dots, b_{\sigma^{-1}(n)}), $$ 
where $b\in B_n$ and $b_i\in B_{j_i}$, and $\omega(b, b_{\sigma^{-1}(1)}, \dots, b_{\sigma^{-1}(n)})$ is the braid in $B_{j_1+\dots+j_n}$ obtained by replacing the $i$-the strand of the braid $b$ by the braid $b_{\sigma^{-1}(i)}$.

The simplicial operad $\mathbb{E}_2$ is defined by letting $\mathbb{E}_2(n)$ be the nerve of the translation groupoid $\Sigma_n//B_n$ for the action of $B_n$ on $\Sigma_n$ via the projection map $\pi\colon B_n\to \Sigma_n$.

%Explicitly, the objects are elements of $\Sigma_n$, and a morphism $\sigma\to \tau$ is a braid $b\in B_n$ such that $\pi(b)\sigma=\tau$. 

\begin{remark}
    We note that for all $i>1$, and all $n\geq 0$ we have $\pi_i(\mathbb{E}_2(n))=0$.  This is immediate from the fact that $\mathbb{E}_2(n)$ is the nerve of a connected groupoid.  In particular, the geometric realization of $\mathbb{E}_2(n)$ is a model for $K(PB_n,1)$, where $PB_n$ is the pure braid group on $n$ strands.
\end{remark}

Associated to the simplicial operad $\mathbb{E}_2$ is the simplicial category of operators $\mathbb{E}_2^{\otimes}$ with objects $\mathbf{n} = \{0,\dots,n\}\in \Fin_*$, and morphism simplicial sets given by
\begin{equation}\label{eq: maps in E_2}
    \Map_{\mathbb{E}_2^\otimes}(\mathbf{m}, \mathbf{n})= \coprod_{\phi\colon \mathbf{m}\to \mathbf{n}}\ \prod_{1\leq i\leq n} \mathbb{E}_2\big(|\phi^{-1}(i)|\big)
\end{equation}
where $\phi\colon \mathbf{m}\to \mathbf{n}$ ranges over maps in $\Fin_\ast$. Note that the underlying category of the simplicial category $\mathbb{E}_2^\otimes$ is precisely the category of operators $\Assoc^\otimes$. Additionally, since each $\mathbb{E}_2(n)$ is the nerve of a groupoid, we have that the mapping simplicial sets in $\mathbb{E}_2^{\otimes}$ are Kan complexes.

We can then take the operadic nerve $N^\otimes(\mathbb{E}_2)$, which, by definition, is the simplicial nerve of the simplicial category $\mathbb{E}_2^\otimes$. Since the mapping simplicial sets of $\mathbb{E}_2^\otimes$ are Kan complexes the operadic nerve $N^\otimes(\mathbb{E}_2)$ is an $\infty$-operad, by \cite[2.1.1.27]{HA}. 

The inclusions $\Sigma_n= \mathbb{E}_2(n)_0\hookrightarrow \mathbb{E}_2(n)$ extend to a map of simplicial operads $\Assoc\to \mathbb{E}_2$, where we view $\Assoc$ as a constant simplicial operad.  Taking categories of operators and simplicial nerves\footnote{Note that we are implicitly using the equivalence of $\infty$-categories between the simplicial nerve of $\Assoc^\otimes$ and the usual nerve of the 1-category $\Assoc^\otimes$.}, this produces a map of $\infty$-operads 
\begin{equation*}
    f^\otimes\colon \Assoc^\otimes \to N^\otimes(\mathbb{E}_2).
\end{equation*} 
We apply \cref{Monoidal envelopes without fibration} to define the $\mathbb{E}_2$-monoidal $\infty$-category $\Env_{\mathbb{E}_2}(\Assoc)$. Explicitly, if we let  $\mathrm{Act}\bigl(N^\otimes(\mathbb E_2)\bigr)
\subset \Fun\bigl(\Delta^1,N^\otimes(\mathbb E_2)\bigr)$ 
be the full subcategory spanned by the active edges, and let
\[
\mathrm{Act}_{\mathbf1}\bigl(N^\otimes(\mathbb E_2)\bigr)
:=
\mathrm{Act}\bigl(N^\otimes(\mathbb E_2)\bigr)
\times_{N^\otimes(\mathbb E_2),\,\operatorname{ev}_1}
\{\mathbf1\}.
\]
Thus $\mathrm{Act}_{\mathbf 1}$ is the fiber of target evaluation over
$\mathbf 1$. Its morphisms are squares whose right-hand edge is
$\id_{\mathbf1}$. Then 
\[
    \Env_{\mathbb{E}_2}(\Assoc) \simeq \widetilde{\Assoc}^\otimes \times_{N^\otimes(\mathbb{E}_2)} \mathrm{Act}_{\mathbf{1}}(N^\otimes(\mathbb{E}_2)).
\]
Recall that $\widetilde{\Assoc}^{\otimes}$ is a replacement for $\Assoc^{\otimes}$ such that $\widetilde{f}^{\otimes}\colon\widetilde{\Assoc}^{\otimes}\to N^{\otimes}(\mathbb{E}_2)$ is a  fibration of $\infty$ operads.

% \noindent We first prove an auxiliary proposition, showing that $\mathrm{Env}_{\mathbb{E}_2}(\mathrm{Assoc})$ is a 1-category.

Using this setup, we prove a crucial auxiliary proposition.
\begin{proposition}\label{1-category}
The $\infty$-category $\mathrm{Env}_{\mathbb{E}_2}(\mathrm{Assoc})$ is equivalent to a $1$-category. 
\end{proposition}
\begin{proof}
First note that since the mapping simplicial sets in the simplicial category $\mathbb{E}_2^\otimes$ are Kan complexes, 
 for any objects $\mathbf{m}$ and $\mathbf{n}$, we have an equivalence of Kan complexes
$$\Map_{\mathbb{E}_2^\otimes}(\mathbf{m}, \mathbf{n})\simeq \Map_{N^\otimes(\mathbb{E}_2)}(\mathbf{m}, \mathbf{n}),$$ which are thus 1-truncated since $\pi_i\mathbb{E}_2(k)=0$ for all $i>1$.

The mapping space $\Map_{\mathrm{Env}_{\mathbb{E}_2}(\mathrm{Assoc})}((c, \alpha), (c', \alpha'))$ is computed as a homotopy pullback 
$$\Map_{\widetilde{\Assoc}^\otimes}(c,c')\times^h_{\Map_{N^\otimes(\mathbb{E}_2)}(\mathbf{n}, \mathbf{n'})}\Map_{\mathrm{Act}_{\mathbf{1}}(N^\otimes(\mathbb{E}_2))}(\alpha, \alpha').$$

A 1-simplex  in  $\mathrm{Act}_{\mathbf{1}}(N^\otimes(\mathbb{E}_2))$ is given by $(\beta, H)$, where $\beta\colon \mathbf{n}\to \mathbf{n'}$ is in $\mathbb{E}_2(\mathbf{n}, \mathbf{n'})_0$ and $H$ is a  square $H\colon \Delta^1\times \Delta^1\to N^\otimes(\mathbb{E}_2)$ of the form 
\begin{equation}\label{square-diagram}
\begin{tikzcd}
\mathbf{n} \arrow[r, "\alpha" ] \arrow[d, "\beta"']
  & \mathbf{1} \arrow[d, "\id"] \\
\mathbf{n'} \arrow[r, "\alpha'"']
  & \mathbf{1},
\arrow[Rightarrow, from=1-1, to=2-2, shorten <=8pt, shorten >=8pt, "\, H" ]
\end{tikzcd}
\end{equation}
where $\alpha, \alpha'$ are active maps. We will later unravel what this data means, and in particular we will see that this forces $\beta$ to also be an active map.

The homotopy fiber of the map $$\Map_{\mathrm{Act}_{\mathbf{1}}(N^\otimes(\mathbb{E}_2))}(\alpha, \alpha')\to \Map_{N^\otimes(\mathbb{E}_2)}(\mathbf{n}, \mathbf{n'})$$ over a fixed $\beta\colon \mathbf{n}\to \mathbf{n'}$ is equivalent to $\mathrm{Path}(\alpha, \alpha'\circ \beta),$ the path space in $\Map_{N^\otimes(\mathbb{E}_2)}(\mathbf{n}, \mathbf{1})$, which we saw is 1-truncated, so the path space is 0-truncated. 

Hence the homotopy fiber of the map $$\Map_{\mathrm{Env}_{\mathbb{E}_2}(\mathrm{Assoc})}((c, \alpha), (c', \alpha'))\to \Map_{\widetilde{\mathrm{Assoc}}^\otimes}(c,c')$$  is then also 0-truncated.

Moreover, since we have an equivalence of $\infty$-categories $\Assoc^\otimes\simeq \widetilde{\Assoc}^\otimes$, the mapping spaces of $\widetilde{\Assoc}^\otimes$ are equivalent to discrete sets, thus 0-truncated. Therefore, $\Map_{\mathrm{Env}_{\mathbb{E}_2}(\mathrm{Assoc})}((c, \alpha), (c', \alpha'))$ is then also 0-truncated, so $\mathrm{Env}_{\mathbb{E}_2}(\mathrm{Assoc})$ is equivalent to a 1-category.
\end{proof}

\begin{proof}[Proof of \cref{thm:envelope}]

By \cref{1-category} it suffices to compare the universal properties of $\DB_+$ and $h\mathrm{Env}_{\mathbb{E}_2}(\Assoc)$. The universal property of $\DB_+$ is spelled out in \cref{braided-observation}. 

As seen in \cref{universal property of envelope}, in the $\infty$-categorical setting, the $\mathbb{E}_2$-monoidal envelope of $\mathrm{Assoc}$ has an analogous universal property: for any $\mathbb{E}_2$-monoidal $\infty$-category $\cC$, restriction along the canonical map $\Assoc^\otimes \to \Env_{\mathbb{E}_2}(\Assoc)$, induces a natural equivalence
\[
\Fun_{\mathbb{E}_2}^\otimes(\Env_{\mathbb{E}_2}(\Assoc), \cC)\simeq \Alg_{\mathbb{E}_1}(\cC),
\]
between the $\infty$-categories of strong $\mathbb{E}_2$-monoidal functors $\Env_{\mathbb{E}_2}(\Assoc)\to \cC$ and $\mathbb{E}_1$-algebras in $\cC$ (see \cite[2.2.4.9]{HA}). 

If $\cD$ is a braided monoidal category, then its  nerve $N(\cD)$ carries a canonical $\mathbb E_2$-monoidal structure induced by the braided monoidal structure on $\cD$~\cite[5.1.2.4]{HA}. By \cref{1-category}, the envelope $\Env_{\mathbb{E}_2}(\Assoc)$ is equivalent to a 1-category, the same comparison translates the universal property of the envelope into a natural equivalence 
$$\Fun_\mathrm{br}^\otimes(h\mathrm{Env}_{\mathbb{E}_2}(\mathrm{Assoc}), \cD)\simeq \mathrm{Mon}(\cD).$$
To make this comparison explicit, note that the homotopy category functor $h$ preserves finite products, and thus the $\mathbb{E}_2$-monoidal structure on $\mathrm{Env}_{\mathbb{E}_2}(\mathrm{Assoc})$ induces a braided monoidal structure on $h\mathrm{Env}_{\mathbb{E}_2}(\mathrm{Assoc})$. Also, by \cref{1-category}, the map $\mathrm{Env}_{\mathbb{E}_2}(\mathrm{Assoc})\to N(h\mathrm{Env}_{\mathbb{E}_2}(\mathrm{Assoc}))$ is an equivalence, which is $\mathbb{E}_2$-monoidal with the induced structures. Thus, for any braided monoidal category $\cD$, we have equivalences
\[\Fun_\mathrm{br}^\otimes(h\mathrm{Env}_{\mathbb{E}_2}(\mathrm{Assoc}), \cD)\simeq h\Fun^\otimes_{\mathbb{E}_2}(\mathrm{Env}_{\mathbb{E}_2}(\mathrm{Assoc}),N(\cD))\simeq h\mathrm{Alg}_{\mathbb{E}_1}(N(\cD))\simeq \mathrm{Mon}(\cD)\]
where the first equivalence follows from the fact that $h$ preserves finite products. 
% Consequently, once \cref{1-category} is established, the two
% universal properties of  $h\Env_{\mathbb{E}_2}(\Assoc)$ and $\DB_+$ give an abstract proof of the equivalence
% asserted in \cref{thm:envelope}. The purpose of
% \cref{explicitequiv} is to identify the resulting comparison functor
% explicitly. Therefore, we first choose a model for $\mathbb{E}_2$ as an operad in simplicial sets, which in turn arises by taking the nerves of an operad in groupoids.  
\end{proof}

We provide an alternate proof and an explicit comparison functor in
\cref{explicitequiv}, which is independent of \cref{braided-observation}. Again, the key step is \cref{1-category}. 

% The proposition implies that the $\infty$-category $\mathrm{Env}_{\mathbb{E}_2}(\mathrm{Assoc})$ is equivalent to the nerve of its homotopy category $h\mathrm{Env}_{\mathbb{E}_2}(\mathrm{Assoc})$~\cite[Proposition~2.3.4.14]{HTT}. 
% Now we can deduce \cref{thm:envelope} from the universal properties of $\DB_+$ and $h\mathrm{Env}_{\mathbb{E}_2}(\Assoc)\simeq \DB_+$.

% \begin{proof}[Proof of \cref{thm:envelope}]
% By \cref{1-category}, it suffices to prove that there is an equivalence of braided monoidal 1-categories $h\mathrm{Env}_{\mathbb{E}_2}(\Assoc)\simeq \DB_+$. 
% \end{proof}

\subsection{Explicit identification of the braided envelope}\label{explicitequiv} 

In this section, we spell out explicitly the equivalence of braided monoidal 1-categories $h\mathrm{Env}_{\mathbb{E}_2}(\Assoc)\simeq \DB_+$, thereby providing an alternate proof of Theorem~\ref{thm:envelope}. 

For any $\mathbf{n}$, there is a unique active map $\mathbf{n}\to \mathbf{1}$ in $\Fin_\ast$, which we will denote by $a_n$. Note that the component indexed by $a_n\colon \mathbf{n}\to \mathbf{1}$ in $\Map_{\mathbb{E}_2^\otimes}(\mathbf{n},\mathbf{1})$, under the decomposition in \eqref{eq: maps in E_2}, is exactly $\mathbb{E}_2(n).$ Therefore, a 0-simplex in $\mathrm{Act}_{\mathbf{1}}(N^\otimes(\mathbb{E}_2))$ is of the form $(a_n, \sigma)$, where $\sigma\in \Sigma_n$. Thus, the objects of $h\mathrm{Env}_{\mathbb{E}_2}(\mathrm{Assoc})$ are of the form $(c, \alpha)$, where $c$ is an object of $\widetilde{\Assoc}^\otimes$, and $\alpha=(a_n, \sigma)$ is an active morphism in $N^\otimes(\mathbb{E}_2)$, and $\mathbf{n}=\widetilde{f}^{\otimes}(c)$, for $\widetilde{f}^\otimes\colon \widetilde{\Assoc}^\otimes\to N^\otimes(\mathbb{E}_2)$.

To understand the morphisms in $h\mathrm{Env}_{\mathbb{E}_2}(\mathrm{Assoc})$, we first recall the morphisms in $\mathrm{Act}_{\mathbf{1}}(N^\otimes(\mathbb{E}_2))$. A 1-simplex  in  $\mathrm{Act}_{\mathbf{1}}(N^\otimes(\mathbb{E}_2))$ is given by $(\beta, H)$, where $\beta\colon \mathbf{n}\to \mathbf{n'}$ is in $\Map_{\mathbb{E}_2^\otimes}(\mathbf{n}, \mathbf{n'})_0=\Assoc^\otimes(\mathbf{n}, \mathbf{n'})$ and $H$ is a  square $H\colon \Delta^1\times \Delta^1\to N^\otimes(\mathbb{E}_2)$ as in \eqref{square-diagram}.
Here, $\alpha=(a_n, \sigma)$ with $\sigma\in \Sigma_n$, $\alpha'=(a_{n'}, \sigma')$ with $\sigma'\in \Sigma_{n'}$, $\beta= (\psi, \times \tau_i)$ with $\psi\colon \mathbf{n} \to \mathbf{n'}$ in $\Fin_\ast$, and $\tau_i\in \Sigma_{|\psi^{-1}(i)|}$ for all $i\in \mathbf{n}'$. The homotopy $H$ is an element in simplicial level 1 of $\Map_{\mathbb{E}_2^\otimes}(\mathbf{n}, \mathbf{1})$, which is given by
\[
    \Map_{\mathbb{E}_2^\otimes}(\mathbf{n}, \mathbf{1})_1= \coprod_{\phi\colon \mathbf{n}\to \mathbf{1}} \Sigma_{|\phi^{-1}(1)|}\times B_{|\phi^{-1}(1)|}\,.
\]
Thus $H=(\phi, \gamma, b)$ is an element in this set for which $d_0(H)=\alpha'\circ \beta$ and $d_1(H)=\alpha$. From this we obtain the following facts.
\begin{itemize}
    \item We must have $\phi=a_n=a_{n'}\circ \psi$, which then forces $\psi$ and $\phi$ to also be active maps.
    \item This, in turn, implies $(\gamma, b)\in \Sigma_n\times B_n$ and we get that $\gamma=\sigma$.
    \item  From the formula for composition in $\Assoc^\otimes$, we get  
    \[
        \pi(b)\gamma= \big( \sigma'(|\psi^{-1}(1)|,\dots,|\psi^{-1}(n')|)\cdot (\tau_1\oplus\dots\oplus\tau_{n'}) \big)\sigma_\psi,
\] 
where $\sigma_\psi$ is the permutation in $\Sigma_n$ which converts the natural ordering of $\{1,\dots, n\}$ to its ordering obtained by regarding it as blocks $\psi^{-1}(i)$, so that the block $\psi^{-1}(i)$ precedes the block $\psi^{-1}(i')$ if $i<i'$, and each block is ordered by the natural ordering $\{1,\dots ,n\}$.
\end{itemize}  

A morphism in $h\mathrm{Env}_{\mathbb{E}_2}(\mathrm{Assoc})$ from $(c, \alpha)$ to $(c', \alpha')$ is represented by a square as in \eqref{square-diagram}, with $\beta=p(u)$ for a map $u\colon c\to c'$ in $\widetilde{\Assoc}^\otimes.$

% Recall that we have a categorical equivalence $i\colon \Assoc^\otimes \simeq \widetilde{\Assoc}^\otimes$ over $\Fin_\ast$. This implies every object in $h\mathrm{Env}_{\mathbb{E}_2}(\Assoc)$ is isomorphic to one of the form $c_n\coloneqq(i(\mathbf{n}), (a_n, \id))$. 

Since $i$ is essentially surjective over $\Fin_\ast$, the first
coordinate of every object is equivalent to some $i(\mathbf n)$.
Moreover, $\mathbb E_2(n)=N(\Sigma_n//B_n)$ is connected, so there is
a path from $\sigma$ to $\id$. Together these give an isomorphism in
the homotopy category from every object $(c,(a_n,\sigma))$ to
$c_n\coloneqq\bigl(i(\mathbf n),(a_n,\id)\bigr).$
For $n=0$, we let $c_0=(i(\mathbf{0}), (a_0, \ast))$, where $a_0\colon \mathbf{0}\to \mathbf{1}$ is the unique pointed map and $\ast$ is the unique element in $\mathbb{E}_2(0)$.

The canonical $\mathbb E_1$-algebra in $\Env_{\mathbb E_2}(\Assoc)$ induces a monoid object in $h\Env_{\mathbb E_2}(\Assoc)$ whose underlying object is identified with $c_1$. By the universal property of $\DB_+$, this monoid
determines a strong braided monoidal comparison functor
\[
\DB_+\longrightarrow h\Env_{\mathbb E_2}(\Assoc).
\]
Under our chosen representatives, its object map is $\langle n\rangle\mapsto c_n$. We now identify this essentially surjective comparison functor explicitly on morphisms and show that it is fully faithful.

% We can define an essentially surjective assignment on objects $\DB_+\to h\Env_{\mathbb{E}_2}(\Assoc)$ by sending the object $[n-1]$ in $\DB_+$ (which we denote by $\langle n \rangle$) to the object $c_n$. First note that the object $c_0$ corresponding to $\langle 0\rangle$ is indeed initial in $h\Env_{\mathbb{E}_2}(\Assoc)$. We now show that this assignment extends on morphisms to a fully faithful functor.

 First, we compute $\Hom_{h\mathrm{Env}_{\mathbb{E}_2}(\Assoc)}(c_n, c_m)$, which by definition is $\pi_0 \Map_{\mathrm{Env}_{\mathbb{E}_2}(\Assoc)}(c_n, c_m)$.  Since  homotopy pullbacks of Kan complexes commute with coproducts, we can identify the mapping space $\Map_{\mathrm{Env}_{\mathbb{E}_2}(\Assoc)}(c_n, c_m)$ with
\begin{equation}\label{eq: mapping space}
    \Map_{\widetilde{\Assoc}^\otimes}(i(\mathbf{n}), i(\mathbf{m})) \times_{\Map_{N^\otimes(\mathbb{E}_2)}(\mathbf{n},\mathbf{m})} \Map_{\mathrm{Act}_{\mathbf{1}}(N^\otimes(\mathbb{E}_2))}(\alpha_n, \alpha_m)
\end{equation}
and we can analyze the two terms in the product separately.

Since $i\colon \Assoc^\otimes \to \widetilde{\Assoc^\otimes}$ is a categorical equivalence of $\infty$-categories,  we have isomorphisms of sets
$$\pi_0 \Map_{\widetilde{\Assoc}^\otimes}(i(\mathbf{n}), i(\mathbf{m}))\cong \pi_0\Map_{\Assoc^\otimes}(\mathbf{n}, \mathbf{m})\cong \Hom_{\Assoc^\otimes}(\mathbf{n}, \mathbf{m}),$$ where the latter is the hom-set in the 1-category $\Assoc^\otimes$. Note that for $\beta\neq \beta'$, the morphisms $i(\beta)$ and $i(\beta')$ are in different components of $\Map_{\widetilde{\Assoc}^\otimes}(i(\mathbf{n}), i(\mathbf{m}))$. Therefore:
$$ \Map_{\widetilde{\Assoc}^\otimes}(i(\mathbf{n}), i(\mathbf{m})) = \coprod_{\beta\in \Assoc^\otimes(\mathbf{n}, \mathbf{m})} \Map_{\widetilde{\Assoc}^\otimes}(i(\mathbf{n}), i(\mathbf{m}))_\beta,$$

\noindent where $\Map_{\widetilde{\Assoc}^\otimes}(i(\mathbf{n}), i(\mathbf{m}))_\beta$ is the connected component containing $i(\beta)$, which is contractible since the mapping spaces in $\widetilde{\Assoc}^\otimes$ are 0-truncated. This allows us to rewrite \eqref{eq: mapping space} as 
\[
    \coprod_{\beta\in \Assoc^\otimes(\mathbf{n}, \mathbf{m})} \Map_{\widetilde{\Assoc}^\otimes}(i(\mathbf{n}), i(\mathbf{m}))_\beta \times_{\Map_{N^\otimes(\mathbb{E}_2)}(\mathbf{n},\mathbf{m})} \Map_{\mathrm{Act}_{\mathbf{1}}(N^\otimes(\mathbb{E}_2))}(\alpha_n, \alpha_m).
\]

We fix a component $\beta$ and consider the Kan fibration 
\[\begin{tikzcd}
\Map_{\widetilde{\Assoc}^\otimes}(i(\mathbf{n}), i(\mathbf{m}))_\beta \times_{\Map_{N^\otimes(\mathbb{E}_2)}(\mathbf{n},\mathbf{m})} \Map_{\mathrm{Act}_{\mathbf{1}}(N^\otimes(\mathbb{E}_2))}(\alpha_n, \alpha_m) \arrow[d] \\ \Map_{\widetilde{\Assoc}^\otimes}(i(\mathbf{n}), i(\mathbf{m}))_\beta. 
\end{tikzcd}\]
Let $F_\beta$ be the fiber over $i(\beta)\in \Map_{\widetilde{\Assoc}^\otimes}(i(\mathbf{n}), i(\mathbf{m}))_\beta $. From the induced long exact sequence of homotopy groups we get an isomorphism
\[
    \pi_0(F_\beta)\cong \pi_0\big(\Map_{\widetilde{\Assoc}^\otimes}(i(\mathbf{n}), i(\mathbf{m}))_\beta \times_{\Map_{N^\otimes(\mathbb{E}_2)}(\mathbf{n},\mathbf{m})} \Map_{\mathrm{Act}_{\mathbf{1}}(N^\otimes(\mathbb{E}_2))}(\alpha_n, \alpha_m)\big),
\] 
since $\Map_{\widetilde{\Assoc}^\otimes}(i(\mathbf{n}), i(\mathbf{m}))_\beta$ is contractible.
The fiber $F_\beta$ is a discrete set consisting of those morphisms of the form $(i(\beta), H)$. \footnote{\parbox[t]{0.95\linewidth}{As an aside, note that in $\pi_0$ of $\Map_{\mathrm{Act}_{\mathbf{1}}(N^\otimes(\mathbb{E}_2))}(\alpha_n, \alpha_m)$, there are homotopy identifications of morphisms $H\simeq H'$: if we consider the Kan fibration
$\Map_{\mathrm{Act}_{\mathbf{1}}(N^\otimes(\mathbb{E}_2))}(\alpha_n, \alpha_m) \to \Map_{N^\otimes(\mathbb{E}_2)}(\mathbf{n},\mathbf{m}),$ note that for $\beta=(\psi, \times \tau_i)$:
$$\textstyle\pi_1(\Map_{N^\otimes(\mathbb{E}_2}(\mathbf{n},\mathbf{m})), \beta)\cong \prod_{1\leq i\leq m} \pi_1(\mathbb{E}_2(|\psi^{-1}(i)|), \tau_i)\cong \prod_{1\leq i\leq m} PB_{|\psi^{-1}(i)|}.$$
 In the fiber over $\beta$ in this Kan fibration, namely squares $H$ with side $\beta$, two such squares would get identified in $\pi_0 \Map_{\mathrm{Act}_{\mathbf{1}}(N^\otimes(\mathbb{E}_2))}(\alpha_n, \alpha_m)$ precisely when they are in the same orbit of the $\pi_1$ action by the product of pure braid groups. 
 However, note that such a path would leave the fiber $F_\beta$ since a path in $F_\beta$ would have to project to the degenerate loop at $\beta$ in $\Map_{N^\otimes(\mathbb{E}_2)}(\mathbf{n},\mathbf{m})$.}} 
 Therefore, we conclude that the elements of $\Hom_{h\mathrm{Env}_{\mathbb{E}_2}(\Assoc)}(c_n, c_m)$ are in correspondence with squares $H$ in $N^\otimes(\mathbb{E}_2)$ as in \eqref{square-diagram} with left-hand side $\beta$, where $\beta\colon\mathbf{n}\to \mathbf{m}$ is an active map in $\Assoc^\otimes$, namely an active map $\psi\colon\mathbf{n}\to \mathbf{m}$ in $\Fin_\ast$ and a tuple of permutations $\tau_i\in \Sigma_{|\psi^{-1}(i)|}$ for all $1\leq i\leq m$. Since the source and targets are $\alpha=(a_n, \id)$  and $\alpha'=(a_m, \id)$,  for $H=(\gamma, b)$ filling the square, we must have $\gamma=\id$ and  $\pi(b)=(\tau_1 \oplus\dots \oplus \tau_m)\sigma_\psi \in \Sigma_n$. 

 Since $\psi\colon\mathbf{n}\to \mathbf{m}$ is an active map, we can identify it with the composite of a permutation in $\Sigma_n$, uniquely determined by the given permutations $\tau_i$ of the preimages, and a map $\phi\colon [n-1]\to [m-1]\in \Delta_+$, with the ordering uniquely determined  by the sizes of the preimages of $\psi$. Since the $\tau_i$ are completely determined by $\pi(b)$ and $\phi$, we can identify a morphism $H\colon c_n\to c_m$ with a pair $(\phi, b)$. 
 Conversely, given a pair $(\phi, b)$, define an active map in $\Fin_\ast$ as $\psi\coloneqq \phi\circ \pi(b)$, and let $\beta$ be the 1-simplex in $\Assoc^\otimes$ defined by $\psi$ and the tuple of permutations $\tau_i\in \Sigma_{|\psi^{-1}(i)|}$ which records the reorderings of preimages. By construction, these assignments are inverses of each other. 
Therefore, we get an isomorphism of hom-sets $$\Hom_{h\mathrm{Env}_{\mathbb{E}_2}(\Assoc)}(c_n, c_m)\cong \Hom_{\Delta_+}([n-1], [m-1])\times B_n=\Hom_{\DB_+}(\langle n\rangle, \langle m\rangle).$$

We explicitly show that along these isomorphisms, composition in $h\mathrm{Env}_{\mathbb{E}_2}(\Assoc)$ agrees with composition in $\DB_+$. Recall that in  $\DB_+$, for a morphism $ (\phi, b)\colon \langle n\rangle \to \langle m\rangle$, and a morphism $(\phi', b')\colon \langle m\rangle \to \langle p\rangle$, from the axiom for a crossed simplicial group, it follows that  the composition is given by the rule
$$(\phi', b')\circ (\phi, b)= \big( \phi'\circ b'^\ast\phi, \phi^\ast (b')b \big).$$
Now consider morphisms $H=(b, \beta)\colon c_n\to c_m$ and  $H'=(b', \beta')\colon c_m\to c_p$ defined by the following squares 
\begin{equation*}
\begin{tikzcd}
\mathbf{n} \arrow[r, "\alpha_n" ] \arrow[d, "\beta"']\arrow[dr, phantom, "b"]
  & \mathbf{1} \arrow[d, "\id"] \\
\mathbf{m} \arrow[r, "\alpha_m"] \arrow[d, "\beta'"']\arrow[dr, phantom, "b'"]
  & \mathbf{1} \arrow[d, "\id"] \\
  \mathbf{p} \arrow[r, "\alpha_p"'] 
  & \mathbf{1}
\end{tikzcd}
\end{equation*}
where $\alpha_i$ denotes $(a_i, \id)$, and $\beta, \beta'$ are active maps in $\Assoc^\otimes$, so they are active maps $\psi, \psi'$ in $\Fin_\ast$ together with choices of permutations on the preimages. We saw that  we can write them each uniquely as composites $\psi=\phi\circ\pi(b)$ and $\psi'= \phi'\circ\pi(b')$, where $\phi, \phi'$ are maps in $\Delta_+$.

The composite $H'\circ H$ is the homotopy class of a square whose lefthand side is $\beta'\circ \beta$, where the composition occurs in $\Assoc^\otimes$, which is isomorphic to $\DS_+$; for an explicit calculation of how the two composition rules agree, see \cite[Theorem~3.17, Example~3.18]{THG}. Explicitly, in $\DS_+$, we get that 
$$(\phi',\pi(b'))\circ ( \phi,\pi(b))= (\ \phi'\circ b'^\ast\phi, \phi^\ast(\pi(b'))\pi(b)).$$
So the $\Delta$ component of the composite is $\phi'\circ b'^\ast\phi$, as desired. The braid $b\in B_n$ represents a path in $\Map_{N^\otimes(\mathbb{E}_2)}(\mathbf{n},\mathbf{1})$ from $\alpha_n$ to $\alpha_m\circ \beta$, while $b'\in B_m$ represents a path in $\Map_{N^\otimes(\mathbb{E}_2)}(\mathbf{m},\mathbf{1})$ from $\alpha_m$ to $\alpha_p\circ \beta'$. In order to compose these, we need to consider the precomposition map 
$$-\circ \beta \colon \Map_{N^\otimes(\mathbb{E}_2)}(\mathbf{m},\mathbf{1}) \to \Map_{N^\otimes(\mathbb{E}_2)}(\mathbf{n},\mathbf{1}),$$ which determines a path $\beta^\ast b'$ in $\Map_{N^\otimes(\mathbb{E}_2)}(\mathbf{n},\mathbf{1})$ from $\alpha_m\circ \beta$ to $\alpha_p\circ \beta'\circ \beta$. By the definition of the simplicial nerve, such a path is a 1-simplex of $\Map_{\mathbb{E}_2^\otimes}(\mathbf{n}, \mathbf{1})$. Explicitly, we apply the degeneracy $s_0$ to the 0-simplex $\beta$ in the mapping simplicial set $\Map_{\mathbb{E}_2^\otimes}(\mathbf{n}, \mathbf{m})$, and we precompose with $s_0 \beta=(\psi, (\tau_1, \id), \dots, (\tau_m, \id))$ at simplicial level 1 of the simplicial category $\mathbb{E}_2^\otimes$ using the composition map 
$$\Map_{\mathbb{E}_2^\otimes}(\mathbf{m}, \mathbf{1})_1\times \Map_{\mathbb{E}_2^\otimes}(\mathbf{n}, \mathbf{m})_1\to \Map_{\mathbb{E}_2^\otimes}(\mathbf{n}, \mathbf{1})_1.$$ From the definition of composition in the category of operators and the operad map definition on simplicial level 1 of $\mathbb{E}_2$, 
$$(a_m, (\id, b'))\circ (\psi, (\tau_1, \id), \dots, (\tau_m, \id))$$
is  $\big(a_m, \ ((\tau_1\oplus \dots\oplus \tau_m)\sigma_\psi, \omega(b', \id, \dots, \id) )\big)$. 
Recall that $\omega(b', \id, \dots, \id)$ is the braid in $B_n$ obtained by replacing the $i$-th strand of $b'$ with $|\psi^{-1}(i)|$ parallel strands, so it is precisely $\phi^\ast (b')$ since $\phi$ was defined as the order preserving factor of $\psi$, so $|\psi^{-1}(i)|= |\phi^{-1}(i)|$.

The braid corresponding to $H'\circ H$ is the one representing the composed path $\beta^\ast b' \circ b$. Since composition in the underlying groupoid is given by multiplication, this is $\phi^\ast(b')b$, as desired.

%Lastly, we show that the equivalence of  categories between $h\Env_{\mathbb{E}_2}(\Assoc)$ and $\DB_+$ is braided monoidal.
It remains to identify the monoidal structure and braiding of the
comparison functor concretely in our chosen model. For $\DB_+$, the braided monoidal structure was described in \cref{augmented-braid}. Note that the active map $a_2\colon \mathbf{2}\to \mathbf{1}$ in $\Fin_\ast$ is the map which encodes multiplication. The monoidal structure on $\Env_{\mathbb{E}_2}(\Assoc)$ is by definition given by 
$$\Env_{\mathbb{E}_2}(\Assoc)\times \Env_{\mathbb{E}_2}(\Assoc) \simeq \Env_{\mathbb{E}_2}(\Assoc)^\otimes_{\mathbf{2}}\xrightarrow{\mu_\ast} \Env_{\mathbb{E}_2}(\Assoc),$$
where the second map is the functor determined by the multiplication map $\mu=(a_2, \id)$ in $N^\otimes(\mathbb{E}_2)$ lifted along the cocartesian fibration $\Env_{\mathbb{E}_2}(\Assoc)^\otimes\to N^\otimes(\mathbb{E}_2)$, which is precomposed with the Segal map \cite[Remark 2.1.2.16]{HA}. Let $a_{n,m}\colon \mathbf{n+m}\to \mathbf{2}$ be the active map in $\Fin_\ast$ which sends the first $n$ nonzero inputs to $1$ and the last $m$ inputs to $2$, and let $\alpha_{n,m}=(a_{n,m}, \id_n, \id_m)$ be the morphism in $N^\otimes(\mathbb{E}_2)$, where $\id_n$ and $\id_m$ are the identity permutations in $\Sigma_n$, and $\Sigma_m$, respectively. 

Under the Segal equivalence $\Env_{\mathbb{E}_2}(\Assoc)\times \Env_{\mathbb{E}_2}(\Assoc) \simeq \Env_{\mathbb{E}_2}(\Assoc)^\otimes_{\mathbf{2}}$, the object $(c_n, c_m)$ corresponds to the object $c_{n,m}=(i(\mathbf{n+m}), \alpha_{n,m})$, which maps to $c_{n+m}$ under the functor $\mu_\ast$. Thus the monoidal structure  in $h\Env_{\mathbb{E}_2}(\Assoc)$ is given on objects by $c_n\otimes c_m=c_{n+m}$, which corresponds to the monoidal structure in $\DB_+$ under the identification of $c_n$ with $\mathbf{n}$. 

% Now consider the other active edge $\mu_\tau=(a_2, \tau)$ in $N^\otimes(\mathbb{E}_2)$, where $\tau$ is the transposition in $\Sigma_2$: it has the same underlying active map $a_2\colon \mathbf{2}\to \mathbf{1}$ in $\Fin_\ast$, but it swaps the two inputs. Note that $\mu_\tau=\mu\circ t$, where $t$ is the morphism in $N^\otimes(\mathbb{E}_2)$ given by the swap map $\mathbf{2}\to \mathbf{2}$ in $\Fin_\ast$, with the unique labels in $\mathbb{E}_2(1)$. To compute  $c_m\otimes c_n$ we can either apply $\mu_\ast$ to the object $c_{m,n}\in \Env_{\mathbb{E}_2}(\Assoc)^\otimes_{\mathbf{2}}$ corresponding to $(c_m, c_n)$ under the Segal map, or equivalently we can apply $(\mu_\tau)_\ast\colon\Env_{\mathbb{E}_2}(\Assoc)^\otimes_{\mathbf{2}}\to \Env_{\mathbb{E}_2}(\Assoc)$, the cocartesian lift of the twisted multiplication map $\mu_\tau$, to $ t_\ast c_{n,m}\simeq c_{n,m}$. We can thus compute $c_m\otimes c_n$ to be $(i(\mathbf{n}+\mathbf{m}), (a_{n+m}, \tau_{n,m})),$ where $\tau_{n,m}$ is the block permutation in $\Sigma_{n+m}$ swapping the first $n$ and last $m$ inputs. Note that this element is identified with its representative $c_{n+m}$ in the homotopy category $h\Env_{\mathbb{E}_2}(\Assoc)$.

Now consider the other active edge
$\mu_\tau=(a_2,\tau)$ in $N^\otimes(\mathbb E_2)$, where
$\tau\in\Sigma_2$ is the transposition. It has the same underlying
active map $a_2\colon\mathbf2\to\mathbf1$ in $\Fin_\ast$ as $\mu$,
but it interchanges the two inputs. More precisely,
$\mu_\tau=\mu\circ t$, where $t\colon\mathbf2\to\mathbf2$ is the
morphism in $N^\otimes(\mathbb E_2)$ given by the transposition in
$\Fin_\ast$, with the unique labels in $\mathbb E_2(1)$. Transport along $t$ interchanges the two Segal factors, and therefore
$t_\ast c_{n,m}\simeq c_{m,n}.$
Since $\mu_\tau=\mu\circ t$, it follows that $$(\mu_\tau)_\ast c_{n,m}\simeq
\mu_\ast t_\ast c_{n,m}\simeq \mu_\ast c_{m,n} \simeq c_m\otimes c_n.$$
In our chosen model, $(\mu_\tau)_\ast c_{n,m}$ is represented by $\bigl(i(\mathbf{n+m}),(a_{n+m},\tau_{n,m})\bigr),$
where $\tau_{n,m}\in\Sigma_{n+m}$ is the block permutation
interchanging the first $n$ inputs with the last $m$ inputs. Thus the
$\tau_{n,m}$-labelled object is an alternative representative of
$c_m\otimes c_n$; its standard representative is
$c_{m+n}=c_{n+m}$.

Let $\sigma \in B_2$ be the generator where the first strand goes over the second. Then $(\id, \sigma)$ is a 1-simplex in $\mathbb{E}_2(2)$ from $\id$ to $\pi(\sigma)=\tau$, which gives a 1-simplex in the mapping space $\Map_{N^\otimes(\mathbb{E}_2)}(\mathbf{2}, \mathbf{1})$ from $\mu$ to $\mu_\tau$, which defines the braiding isomorphism in $\Env_{\mathbb{E}_2}(\Assoc)$. Under the operad map in simplicial level 1,
$$\mathbb{E}_2(2)_1\times \mathbb{E}_2(n)_1 \times \mathbb{E}_2(m)_1\to \mathbb{E}_2(n+m)_1, $$
the 1-simplex $(\id, \sigma)$, combined with the degenerate 1-simplices $s_0(\id_n), s_0(\id_m)$ yields  the 1-simplex $(\id, b_{n,m})$, where $b_{n,m}$ is the braid defined in \cref{augmented-braid}, where the first $n$ strands cross over the last $m$ strands. The braiding isomorphism $c_n\otimes c_m \cong c_m\otimes c_n$ is then precisely a square with braid label $b_{n,m}$.

\subsection{\texorpdfstring{$\mathbb{E}_2$}{E2}-ring structure on topological braid homology} Recall that for the symmetric category $\DS$, we have the equivalence $\DS_+\simeq \Env(\Assoc)$ with the symmetric monoidal envelope of $\Assoc$, see \cite[Theorem~3.17]{THG}. 
Therefore given any symmetric monoidal $\infty$-category $\cC$, there are natural equivalences
\[
\Fun_{\mathbb{E}_2}^\otimes(\Env_{\mathbb{E}_2}(\Assoc), \cC)\simeq \Alg_{\mathbb{E}_1}(\cC)\simeq \Fun_{\mathbb{E}_\infty}^\otimes (\Env(\Assoc), \cC),
\]
between the $\infty$-categories of strong $\mathbb{E}_2$-monoidal functors $\Env_{\mathbb{E}_2}(\Assoc)\to \cC$, $\mathbb{E}_1$-algebras in $\cC$, and strong symmetric monoidal functors $\Env(\Assoc)\to \cC$.
In particular, picking $\cC$ to be $\Env(\Assoc)$, we obtain a natural strong $\mathbb{E}_2$-monoidal functor $\Env_{\mathbb{E}_2}(\Assoc)\to \Env(\Assoc)$ corresponding to the identity functor on $\Env(\Assoc)$.

\begin{corollary}\label{the map DB to DS}
Under the explicit equivalence
$\DB_+\simeq\Env_{\mathbb E_2}(\Assoc)$
constructed in \cref{explicitequiv}, the canonical functor
$\Env_{\mathbb E_2}(\Assoc)\longrightarrow\Env(\Assoc)$
corresponds to the nerve of the projection $\DB_+\longrightarrow\DS_+.$
  % The strong $\mathbb{E}_2$-monoidal functor 
  %   \(\Env_{\mathbb{E}_2}(\Assoc)\to \Env(\Assoc)\) 
  %   is precisely the nerve of the braided functor $\DB_+\to \DS_+$ induced by the canonical projection $\DB\to \DS$.
\end{corollary}

\begin{proof}
Under the universal property of $\Env(\Assoc)$, the identity functor
corresponds to the canonical monoid object of $\DS_+$ whose underlying
object is $\langle1\rangle$. Its unit and multiplication are the
canonical maps
\[
\langle0\rangle\longrightarrow\langle1\rangle,
\qquad
\langle2\rangle\longrightarrow\langle1\rangle.
\]
Under the universal property of $\DB_+$, the projection
$\DB_+\to\DS_+$ corresponds to the same monoid object: it carries the
canonical unit and multiplication in $\DB_+$ to the two maps above.
Consequently, the two strong $\mathbb E_2$-monoidal functors classify
the same monoid object in $\DS_+$ and are naturally equivalent.
\end{proof}

% \begin{proof}
%  As $\Env(\Assoc)$ is the nerve of $\DS_+$, then the identity functor on $\Env(\Assoc)$ is precisely picking an $\mathbb{E}_1$-algebra in $\Env(\Assoc)$, i.e.\ a monoid object in $\DS_+$, which is the monoid $\langle 1 \rangle\in \DS_+$, see notation in \cref{def: augmented simplex cat}.
%  Similarly, the braided functor $\DB_+\to \DS_+$ also precisely determines a monoid object in $\DS_+$, which is determined by where $\langle 1 \rangle\in \DB_+$ is sent to, which is $\langle 1 \rangle \in \DS_+$. Therefore, the result follows.
% \end{proof}

\begin{corollary}\label{TB as homotopy colimit from monoidal envelope}
Given an $\mathbb{E}_1$-ring spectrum $R$, the topological braid homology $\TB(R)$ can equivalently be obtained as the homotopy colimit of $F_R^{\mathbb{E}_2}\colon \Env_{\mathbb{E}_2}(\Assoc)\to \Sp$, the adjunct of the map of $\infty$-operads $R^\otimes\colon \Assoc^\otimes\to \Sp^\otimes$ defining $R$ as in \cref{strong monoidal functor associated to algebras}.
\end{corollary}

\begin{proof}
Observe that the following diagram commutes in small categories:
\[
\begin{tikzcd}
    \DB \ar[hook]{r} \ar{d} & \DB_+ \ar{d}\\
    \DS \ar[hook]{r}  & \DS_+.
\end{tikzcd}
\]
Therefore, $\TB(R)$ is equivalently the spectrum obtained as the homotopy colimit of
 \[
    \begin{tikzcd}
        \DB \ar[hook]{r} & \DB_+ \ar{r} & \DS_+\simeq \Env(\Assoc) \ar{r}{F_R} & \Sp.
    \end{tikzcd}
    \]
Similar to \cref{lemma:THG}, the inclusion $\DB \hookrightarrow \DB_+$ is cofinal. Indeed, the comma category of the inclusion $\DB\times_{\DB_+}(\DB_+)_{[n]/}$ 
is equivalent to the over category $\DB_{[n]/}$ when $n\ge 0$ and $\DB$ when $n=-1$. When $n\ge 0$, this is contractible with initial object the identity on $[n]$, while for $[n]=[-1]$ the category is equivalent to $\DB$, which is contractible by~\cite[Example 6 in Section 1.5]{FL91} and the proof of \cite[Proposition 5.8]{FL91}. Therefore, the inclusion $\DB\to \DB_+$ is homotopy cofinal \cite[4.1.3.1]{HTT} by Quillen's Theorem A for $\infty$-categories, and the homotopy colimit of a diagram indexed over $\DB_+$ is equivalent to the homotopy colimit of the diagram indexed over its restriction on $\DB$, see \cite[4.1.1.8]{HTT}.

Finally, by \cref{the map DB to DS}, the induced composition
\[
\begin{tikzcd}
    \Env_{\mathbb{E}_2}(\Assoc)\simeq \DB_+ \ar{r} & \DS_+\simeq \Env(\Assoc) \ar{r}{F_R} & \Sp 
\end{tikzcd}
\]
is precisely the strong $\mathbb{E}_2$-monoidal functor $F_R^{\mathbb{E}_2}\colon\Env_{\mathbb{E}_2}(\Assoc)\to \Sp$ of \cref{strong monoidal functor associated to algebras}.
\end{proof}

\begin{corollary}\label{TB-rep-homology}
The topological braid homology $\TB(R)$ of an $\mathbb{E}_1$-ring spectrum has a canonical $\mathbb{E}_2$-ring spectrum structure. Additionally, given an associative algebra $R$ in $\Sp$, there is an equivalence of $\mathbb{E}_2$-ring spectra
\begin{equation*}
\mathrm{TB}(R)\simeq \mathrm{HR}^{\mathbb{E}_2}(R) \,.
\end{equation*}
\end{corollary}

\begin{proof}
By \cref{TB as homotopy colimit from monoidal envelope}, $\TB(R)$ is the homotopy colimit of the strong $\mathbb{E}_2$-monoidal functor $F_R^{\mathbb{E}_2}\colon\Env_{\mathbb{E}_2}(\Assoc)\to \Sp$.
By \cite[Proposition~3.16, Example~3.18]{Keenan-Peroux}, the colimit has a canonical structure of an $\mathbb{E}_2$-ring spectrum. This proves the first statement. 
The second statement follows from \cref{representation homology thm} and \cref{TB as homotopy colimit from monoidal envelope}.
\end{proof}

\begin{remark}\label{rem-formula-TB}
Let $\mathrm{Free}_{\mathbb{E}_1}$ denote the left adjoint of the forgetful functor $\mathrm{Alg}(\Sp)\to \Sp$ and let $X$ be a spectrum. By the Bousfield--Kan formula for the colimit in \cref{TB-def}, we have an equivalence 
\[
\TB(\mathrm{Free}_{\mathbb{E}_1}(X)) \simeq \coprod_{n\ge -1}([n]\downarrow \DB)\times \mathrm{Free}_{\mathbb{E}_1}(X)^{\times n+1}/\sim   \,.
\]
As in~\cite[Lemma~6.8]{THG}, we can identify $([n]\downarrow \DB)=EB_{n+1}$\,. Notice also
\begin{align*}
    EB_{n+1}\times_{B_{n+1}}\Sigma_{n+1}& \simeq (*\times_{B_{n+1}}B_{n+1})\times_{B_{n+1}}\Sigma_{n+1}\\*
    & \simeq *\times_{B_{n+1}}(B_{n+1}\times_{B_{n+1}}\Sigma_{n+1})\\*
       & \simeq *\times_{B_{n+1}}\Sigma_{n+1}\\*
    & \simeq \mathbb{E}_2(n+1) \,.
    \end{align*}
Since the $B_{n+1}$-action on $\Free_{\mathbb{E}_1}(X)^{\times n+1}$ is given by restriction along the group homomorphism $B_{n+1}\to \Sigma_{n+1}$ and every map in $\DB$ factors as a composite of an element in the braid group followed by a map in $\Delta$,  
we have an explicit presentation
\[
\mathrm{TB}(\Free_{\mathbb{E}_1}(X))\simeq \coprod_{n\ge -1}\mathbb{E}_2(n+1)\times_{\Sigma_{n+1}} X^{\times n+1}
\]
of the homotopy colimit defining $\mathrm{TB}(\Free_{\mathbb{E}_1}(X))$, cf.~\cite[Proposition~6.10]{THG}. One can then use the fact that both sides commute with sifted colimits and choose a free resolution in order to provide a different proof of \cref{TB-rep-homology}.
\end{remark}

\subsection{Topological braid homology of loop spaces}\label{TB-loop}
In this section, we compute the topological braid homology of group-like $\mathbb{E}_1$-spaces. Let $J = \Omega\Sigma$ be the monad on based spaces whose algebras are group-like $\mathbb{E}_1$-spaces. Let $\mathrm{Free}_{\mathbb{E}_2}$ denote the left adjoint of the forgetful functor from the $\infty$-category of $\mathbb{E}_2$-algebras in spaces to the $\infty$-category of spaces.

We first recall some definitions and a construction. 
\begin{definition}[{\cite[Definitions~2.2,~9.4]{May72}}]
Let $C$ be a monad in topological spaces with structure maps $\mu \colon CC\to C$ and $\eta \colon \mathrm{id}\to C$. A $C$-functor is $(F,\lambda)$ in topological spaces is a functor $F\colon \mathscr{S}\to \mathscr{S}$ together with a natural transformation of functors $\lambda\colon  FC\to F$ such that $\lambda \circ F\eta =\mathrm{id}$ and $\lambda \circ F\mu= \mu \circ \lambda C$. A $C$-algebra in topological spaces is a topological space $X$ together with a map $\xi\colon CX\to X$ such that $\xi\circ \eta_{X} =\mathrm{id}$ and $\xi \circ \mu_{X} = \xi \circ C\xi$  
\end{definition}
\begin{example}
The functors $J$,   $\mathrm{Free}_{\mathbb{E}_2}$ and $\Omega^2\Sigma^2$ and $\Sigma$ are $J$-functors. 
\end{example}
\begin{construction}[{\cite[Construction~9.6]{May72}}]
Let $(C,\mu,\eta)$ be a monad, let $(F,\lambda)$ be a $C$-functor, and let $X$ be a $C$-algebra. We can form a simplicial object
\( [q] \mapsto FC^{q}X \)
with face maps given by $\delta_0=\lambda$ where $\lambda \colon FC^{q}X\to FC^{q-1}X$,  $\delta_i=FC^{i-1}\mu$ where $\mu \colon C^{q-i+1}X\to C^{q-i}X$ for $0< i<q$, and $d_q=FC^{q}\xi$. The  degeneracy maps are given by $s_i=FC^i\eta$ where $\eta \colon C^{q-i}X\to C^{q-i+1}X$ for $0\le i\le q$. We then define 
\[ B(F,C,X)\]
to be the geometric realization of this simplicial space. 
\end{construction}
\begin{remark}\label{free-E2-on-E1}
   If we write $\mathrm{HR}^{\mathbb{E}_2}(M/\mathscr{S})$ for the free $\mathbb{E}_2$-algebra on a group-like $\mathbb{E}_1$-algebra $M$ in spaces, then $B(\mathrm{Free}_{\mathbb{E}_2},J,M)$ is an explicit model for  $\mathrm{HR}^{\mathbb{E}_2}(M/\mathscr{S})$
\end{remark}

\begin{proposition}\label{main prop for TP of loop spaces}
Let $M$ be a group-like $\mathbb{E}_1$-space. There is an equivalence of $\mathbb{E}_2$-ring spectra 
\[ \mathrm{TB}(\Sigma_+^{\infty}M)\simeq \Sigma_{+}^{\infty}B(\Free_{\mathbb{E}_2},J,M)\,.
\]
\end{proposition}
\begin{proof}
Since $M$ is a group-like $\mathbb{E}_1$-space there is an equivalence 
$B(J,J,M)\simeq M$
and consequently an equivalence 
$\TB( \Sigma_+^{\infty}B(J,J,M))\simeq \TB( \Sigma_+^{\infty}M)$.
Commuting colimits with colimits we produce an equivalence
\begin{equation}\label{commuting colimits with colimits}
    \mathrm{colim}_{\DB}B_{\DB}^{\bullet}(\Sigma_+^{\infty}B(J,J,M))\xrightarrow{\simeq} B(\mathrm{colim}_{\DB}B_{\DB}^{\bullet}J,J,M)
\end{equation}
By \cref{TB-rep-homology}, we can identify 
$B(\mathrm{colim}_{\DB}B_{\DB}^{\bullet}J,J,M)\simeq \Sigma_{+}^{\infty}B(\Free_{\mathbb{E}_2},J,M)$. This equips the right hand side of \eqref{commuting colimits with colimits} with an $\mathbb{E}_2$-ring structure. By inspection, the map \eqref{commuting colimits with colimits} inducing the equivalence can be identified with the adjunct to the map $\Sigma^{\infty}_+M\to \Sigma_{+}^{\infty}B(\Free_{\mathbb{E}_2},J,M)$ given by inclusion of zero simplices so \eqref{commuting colimits with colimits} is an equivalence of $\mathbb{E}_2$-rings. 
\end{proof}
\begin{remark}
In light of \cref{free-E2-on-E1}, we can view \cref{main prop for TP of loop spaces} as stating that 
\[ \mathrm{HR}^{\mathbb{E}_2}(\Sigma^{\infty}_+M)\simeq \Sigma^{\infty}_+\mathrm{HR}^{\mathbb{E}_2}(M/\mathscr{S})\,.
\]
\end{remark}

\begin{theorem}
Let $M$ be a group-like $\mathbb{E}_1$-space. Then there is an equivalence of $\mathbb{E}_2$-ring spectra
\[ \mathrm{TB}(\Sigma_{+}^{\infty}M)\simeq \Sigma_{+}^{\infty}\Omega^2\Sigma BM \,. 
\]
\end{theorem}
\begin{proof}
Explicitly, by \cref{main prop for TP of loop spaces}, we know that there is an equivalence of $\mathbb{E}_2$-ring spectra 
$\mathrm{TB}(\Sigma_{+}^{\infty}M)\simeq \Sigma_{+}^{\infty}B(\Free_{\mathbb{E}_2},J,M)$.
Since $M$ is a group-like $\mathbb{E}_1$-space there is an equivalence $\Sigma_{+}^{\infty}B(\Free_{\mathbb{E}_2},J,M)\simeq \Sigma_{+}^{\infty}B(\Omega^2\Sigma^2,J,M)$ of $\mathbb{E}_2$-ring spectra, where we equip the right hand side with an $\mathbb{E}_2$-ring spectrum structure using the $\mathbb{E}_2$-ring spectrum structure on the left side. 
By~\cite[Lemma~9.7, Theorems~12.3,~12.4]{May72}, there is an equivalence 
\[B(\Omega^2\Sigma^2,J,M)\simeq \Omega^2\Sigma B(\Sigma,J,M)\]
of $\mathbb{E}_2$-spaces. 
By~\cite[Corollary~7.9]{Fie84}, we know that $B(\Sigma,J,M)\simeq BM$,
finishing the proof. 
\end{proof}

In light of May's recognition theorem~\cite{May72}, we have the following.

\begin{corollary}
Let $X$ be a pointed connected space. There is an equivalence of $\mathbb{E}_2$-ring spectra
\[ 
\mathrm{TB}(\Sigma_{+}^{\infty}\Omega X)\simeq \Sigma_{+}^{\infty}\Omega^2\Sigma X \,. 
\]
\end{corollary}

\begin{remark}
An algebraic analogue of this result appears in unpublished work of Fiedorowicz~\cite[\S~3]{Fie}. 
\end{remark}

We also note that there is a canonical map from topological braid homology to topological symmetric homology. 
\begin{proposition}\label{prop:TBtoTS}
For any $\mathbb{E}_1$-ring spectrum $R$, there is a canonical map of $\mathbb{E}_2$-ring spectra 
\[ \TB(R)\to \TS(R)\,.\]
\end{proposition}
\begin{proof}
The functor $\DB\to \DS$ induces a canonical map $\TB(R)\to \TS(R)$. 
\end{proof}

\section{Computations}\label{sec:computations}

In this section, we introduce computational tools and techniques for topological symmetric and braid homology.  Using these, we give some explicit computations. A number of the results in this section apply to homology theories for general crossed simplicial groups, not just symmetric and braid homology.

We first recall the construction of topological $\Delta \mathbf{G}$-homology of a ring spectrum with twisted $G$-action. We then construct spectral sequences for computing topological twisted symmetric homology, $\TF$, which is associated to any group with parity.  We then prove some general results about topological $\Delta \mathbf{G}$-homology of Thom spectra and we specialize to $\DS$ and $\DB$ to obtain explicit computations.  

\subsection{Topological \texorpdfstring{$\DG$}{Delta G}-homology}

We begin by recalling the definition of groups with parity.

\begin{definition}
    A \emph{group with parity} is a pair $(G,\varphi)$ where $G$ is a group and $\varphi\colon G\to C_2=\{1,-1\}$ is a group homomorphism.
\end{definition}

Given a group with parity, an $\mathbb{E}_1$-\emph{ring spectrum with twisted $G$-action} is a common generalization of a ring spectrum with anti-involution and a ring spectrum with $G$-action. For the definition, see \cite[3.20]{THG}. 

In previous work, the first, fourth and fifth authors introduce the \textit{twisted symmetric category} $\Delta \varphi\wr \Sigma$, a crossed simplicial group associated to a group with parity $(G, \varphi)$ that recovers $\DS$ if $\varphi$ is the trivial group homomorphism $1\to C_2$ \cite[2.9]{THG}.
Moreover, they prove that the nerve of the augmented twisted symmetric category $(\Delta \varphi\wr \Sigma)_+$ is equivalent to the symmetric monoidal envelope of the operad $\Assoc_{\varphi}$ that encodes rings with twisted $G$-actions \cite[3.17]{THG}. In other words, the data of an $\mathbb{E}_1$-algebra $R$ with twisted $G$-action in a symmetric monoidal $\infty$-category $\cC$ is the same data as a symmetric monoidal functor 
\[ F_R\colon  (\Delta \varphi\wr \Sigma)_+\simeq \mathrm{Env}_{}(\mathrm{Assoc}_{\varphi})\to \mathscr{C} \,. \] 

Given any crossed simplicial group $\DG$, there is an associated group homomorphism $\lambda_0\colon G_0\to C_2$, called the \emph{canonical parity} of $\DG$ \cite[2.19]{THG}, and a natural functor $\widetilde{\lambda}\colon \DG\to \Delta \lambda_0\wr \Sigma$ \cite[2.22]{THG}.

\begin{definition}[{\cite[4.28]{THG}}]\label{def: DG-homology}
Let $\Delta \mathbf{G}$ be a crossed simplicial group with canonical parity $\lambda_0 \colon G_0\to C_2$. Let $\mathscr{C}$ be a cocomplete symmetric monoidal $\infty$-category. 
    The \textit{positive topological $\DG$-homology} $\mathrm{H}\mathbf{G}^+(R/\mathscr{C})$ of an $\mathbb{E}_1$-algebra $R$ with twisted $G_0$-action in $\mathscr{C}$ is the homotopy colimit of the functor 
    \[
    \begin{tikzcd}
       \BDG(R) \colon \Delta \mathbf{G} \ar{r} &  (\Delta \lambda_0\wr \Sigma)_+\simeq \Env_{}(\Assoc_{\lambda_0}) \ar{r}{F_R} & \mathscr{C} \,.
    \end{tikzcd}
    \]
When $\mathscr{C}=\mathrm{Sp}$, we write $\mathrm{T}\mathbf{G}^+(R):=\mathrm{H}\mathbf{G}^+(R/\mathrm{Sp})$. When $A$ is an $\mathbb{E}_{\infty}$-ring spectrum and $\mathscr{C}=\mathrm{Mod}_A$, we write $\mathrm{T}\mathbf{G}^+(R/A):=\mathrm{H}\mathbf{G}^+(R/\mathrm{Mod}_A)$. 
The functor $\BDG(R)$ can be pictured diagrammatically as 
\[
\begin{tikzcd}
   R \ar[out=120, in=60, loop, looseness=6]{}{G_0} \arrow[r, altstackar=3] & R^{\otimes_{k} 2} \ar[out=120, in=60, loop, looseness=6]{}{G_1} \arrow[altstackar=5]{r} & R^{\otimes_{k} 3} \ar[out=120, in=60, loop, looseness=6]{}{G_2} \arrow[altstackar=7]{r} & \cdots. 
\end{tikzcd}
\]
    In the case where $\Delta\mathbf{G}=\Delta  \varphi\wr \Sigma$ for some group with parity $\varphi\colon G\to C_2$, we write $\mathrm{T}\mathbf{G}^+(R/A)$ as $\TF(R/A)$ and refer to it as the \textit{topological twisted symmetric homology} for $\mathbb{E}_1$-ring spectra with twisted $G$-action. When $k$ is a discrete commutative ring and $R$ is a $Hk$-algebra, we write  $\mathrm{H}\varphi(R/k)\coloneqq\mathrm{T}\varphi(R/Hk)$ and refer to it as \textit{twisted symmetric homology}. We write $\mathrm{H}\varphi_s(R/k)=\pi_s\mathrm{H}\varphi(R/k)$ and we simply write $\mathrm{H}\varphi(R)$ and $\mathrm{H}\varphi_s(R)$ when $k=\mathbb{Z}$. 

    When $\Delta \mathbf{G}$ is equipped with a duality $\Delta \mathbf{G}^{\textup{op}}\xrightarrow{\cong} \Delta \mathbf{G}$, we define the \textit{topological $\DG$-homology} of an $\mathbb{E}_1$-algebra with twisted $G_0$-action in $\mathscr{C}$, which we denote by $\mathrm{TH}\mathbf{G}(R/\mathscr{C})$, as the homotopy colimit of the composite
     \[
    \begin{tikzcd}
     \Delta^{\textup{op}}\to \Delta \mathbf{G}^{\textup{op}} \to \Delta \mathbf{G} \ar{r} &  (\Delta \lambda_0\wr \Sigma)_+\simeq \Env(\Assoc_{\lambda_0}) \ar{r}{R} & \mathscr{C} \,.
    \end{tikzcd}
    \]
When $\mathscr{C}=\mathrm{Mod}_{A}$, we write $\mathrm{TH}\mathbf{G}(R/A)$ and when $\mathscr{C}=\Sp$ we write  $\mathrm{TH}\mathbf{G}(R)$.
\end{definition}

\begin{remark}
    Topological symmetric homology $\TS$ is universal among all topological $\DG$-homologies $\mathrm{T}\mathbf{G}^+$ with trivial parity: for any crossed simplicial group $\DG$ where $G_0=1$ is the trivial group and any $\mathbb{E}_1$-ring spectrum $R$, there is a natural map of spectra $\mathrm{T}\mathbf{G}^+(R)\to \TS(R)$.
    More generally, given any crossed simplicial group $\DG$ with canonical parity $\lambda_0\colon G_0\to C_2$ and any $\mathbb{E}_1$-ring spectrum $R$ with twisted $G_0$-action, there is a natural map of spectra $\mathrm{T}\mathbf{G}^+(R)\to \mathrm{T}\lambda_0(R)$.
\end{remark}

\subsection{Degree zero computations}
In this subsection, we study the $0$-th homotopy group of  $\mathrm{T}\varphi(R)$ for connective a $\mathbb{E}_1$-ring spectrum $R$. First, we need a lemma. 
% construct a spectral sequence that computes $\mathrm{T}\varphi(R)$ using the Whitehead filtration. 
% \begin{proposition}\label{SS}
% Let $\varphi\colon G\to C_2$ be a group with parity. Let $A$ be an $\mathbb{E}_{\infty}$-ring spectrum and let $R$ be a connective $\mathbb{E}_1$-algebra over $A$ with twisted $G$-action. There is a strongly convergent spectral sequence with ${E}^1$-page 
% \[ E_{s,t}^1=\pi_s\mathrm{colim}_{\Delta \varphi\wr \Sigma}(H\pi_{t}B_{\Delta\varphi\wr \Sigma}^{\smbullet}(R/A))\implies \pi_{s+t}\mathrm{T}\varphi(R/A)\,.
% \]
% \end{proposition}
% \begin{proof}
% The spectral sequence is associated to the filtered spectrum 
% \[
% \begin{tikzcd}
%    \tau_{\ge \bullet} R \ar[out=120, in=60, loop, looseness=6]{}{G} \arrow[r, altstackar=3] & \tau_{\ge \bullet}(R^{\otimes_{A} 2}) \ar[out=120, in=60, loop, looseness=6]{}{G\wr \Sigma_2} \arrow[altstackar=5]{r} & \tau_{\ge \bullet}(R^{\otimes_{A} 3}) \ar[out=120, in=60, loop, looseness=6]{}{G\wr \Sigma_{3}} \arrow[altstackar=7]{r} & \cdots .
% \end{tikzcd}
% \]
% It is conditionally convergent because the Whitehead filtration is exhaustive and complete. Since $R$ is connective, it is concentrated in the first quadrant with exiting differentials so it is strongly convergent.   
% \end{proof}

\begin{lemma}\label{connectivity-lemma}
Let $\varphi\colon G\to C_2$ be a group with parity with associated crossed simplicial group $\Delta \varphi \wr \Sigma$. Let $X^{\smbullet}\to Y^{\smbullet}$ be a map of functors $\Delta \varphi \wr \Sigma\to \mathrm{Sp}$ such that the map $X^q\to Y^q$ of spectra is $n$-connective for each $q\ge 0$, then the map 
\[ \mathrm{colim}_{\Delta \varphi \wr \Sigma} X^{\smbullet} \to  \mathrm{colim}_{\Delta \varphi \wr \Sigma} Y^{\smbullet} \]
is $n$-connective. 
\end{lemma}

\begin{proof}
This follows because colimits of $n$-connective spectra are $n$-connective. 
\end{proof}

\begin{remark}\label{map of graded commutative rings}
Let $S$ be an $\mathbb{E}_{\infty}$-ring spectrum, let $A$ be an $\mathbb{E}_{\infty}$-$S$-algebra, and let $R$ be an $\mathbb{E}_{1}$-$A$-algebra with twisted $G$-action for some group with parity $\varphi \colon G\to C_2$. The induced map $\mathrm{T}\varphi (R/S)\to \mathrm{T}\varphi (R/A)$
is a map of $\mathbb{E}_{\infty}$-ring spectra, by the same argument as that of~\cref{Einfty-structure}.
\end{remark}

\begin{corollary}\label{cor:low-degrees}
Let $\varphi\colon G\to C_2$ be a group with parity. Let $R$ be a discrete ring with twisted $G$-action. There is a map 
$\pi_*\mathrm{T}\varphi(HR)\to \mathrm{H}\varphi_*(R)$
of graded commutative rings that induces isomorphisms in degree $0$. In particular, there is a map of graded commutative rings 
\[
\pi_*\TS(HR)\to \HS_*(R)
\]
that induces an isomorphism in degree $0$.
\end{corollary}
\begin{proof}
% We consider the spectral sequences of \cref{SS} for $HR$ with $A=\mathbb{S}$ and $A= H\mathbb{Z}$, respectively. 
The unit map $\mathbb{S}\to H\mathbb{Z}$ induces natural %a map of spectral sequences.  In total degree $0\leq m\leq 2$ there are 
$G\wr \Sigma_n$-equivariant equivalences 
\begin{align*}
\tau_{\le 0}(HR^{\otimes q+1})& \simeq \tau_{\le 0}(HR^{\otimes_{H\mathbb{Z}} q+1})
\end{align*}
The claim then follows from \cref{connectivity-lemma}.
\end{proof}

\begin{corollary}\label{cor:pi0-of-TS}
Let $\varphi \colon G\to C_2$ be a group with parity. 
If a map $R\to S$ of $\mathbb{E}_1$-ring spectra with twisted $G$-action is $n$-connective, then the induced map 
\[ \mathrm{T}\varphi (R)\to \mathrm{T}\varphi (S)\]
is $n$-connective. In particular, if $R$ is a connective $\mathbb{E}_1$-ring spectrum with twisted $G$-action then the map 
\[ 
\pi_s\TS(R)\to \HS_s(\pi_0R)\,
\]
is an isomorphism when $s=0$ and surjective when $s=1$. 
\end{corollary}
\begin{proof}
This is immediate from \cref{connectivity-lemma} and \cref{cor:low-degrees}. 
\end{proof}
If $R$ is a discrete ring, let $([R,R])$ denote the ideal generated by all elements of the form $rs-sr$ for $r,s\in R$. 
\begin{remark}
    We emphasize that $([R,R])$ is an ideal, not a subgroup.  The quotient $R\to R/([R,R])$ is the initial ring map from $R$ to a commutative ring.  In particular $R/([R,R])$ is, generally, smaller than $\HH_0(R)$. 
\end{remark}

\begin{corollary}\label{cor:pi0}
Let $R$ be a connective $\mathbb{E}_1$-ring spectrum. There is an isomorphism 
\[ 
\pi_0\TS(R)\cong \pi_0R/([\pi_0R,\pi_0R])\,.
\]
Moreover, $\pi_0R=([\pi_0R,\pi_0R])$ if and only if $\TS(R)=0$. 
\end{corollary}
\begin{proof}
This follows from~\cref{cor:pi0-of-TS} and~\cite[Theorem~86]{Aul10}. 
\end{proof}

From this, we can also conclude that $\TS$ is not Morita invariant. This should not come as a surprise, since the algebraic counterpart is also not Morita invariant. 
\begin{corollary}\label{cor:not-morita}
 The functor $\TS$ is not Morita invariant as a functor on connective $\mathbb{E}_1$-ring spectra.
\end{corollary}
\begin{proof}
By \cref{cor:pi0}, $\TS_0(HM_n(\mathbb{Z}))=0$ for $n\ge 2$ and consequently  $\TS(HM_n(\mathbb{Z}))=0$ for all $n\ge 2$ whereas $\TS_0(H\mathbb{Z}) \cong \mathbb{Z}/([\mathbb{Z},\mathbb{Z}])\cong \mathbb{Z}$.
The claim now follows from the fact that $\mathbb{Z}$ and $M_n(\mathbb{Z})$ are Morita equivalent.
\end{proof}

\begin{remark}
Similar arguments to those above imply that topological hyperoctahedral homology $\mathrm{TO}$ is not Morita invariant as a functor on connective $\mathbb{E}_1$-ring spectra with anti-involution using~\cite[Theorem~5.8]{Gra23}. 
\end{remark}

We end this section by showing that there is also a B\"okstedt-type spectral sequence which computes the homology of $\TS(R)$. 

\begin{proposition}\label{prop:Bokstedt}
Let $(G,\varphi)$ be a group with parity. 
Let $R$ be a connective $\mathbb{E}_1$-ring spectrum with twisted $G$-action and let $k$ be a field. Then there is a multiplicative strongly convergent B\"okstedt-type spectral sequence 
\begin{equation}\label{general-BokSS}
    E^2 = \mathrm{H}\varphi_*(H_*(R;k)/k)\implies H_*(\mathrm{T}\varphi(R);k)\,.
\end{equation}
In particular, there is a strongly convergent B\"okstedt-type spectral sequence 
\[ 
    E^2 = \HS_*(H_*(R;k))\implies H_*(\TS(R);k)
\]
where $\HS$ is defined as in \cref{derived HS} here. 
\end{proposition}
\begin{proof}
This is the spectral sequence associated to the functor
\[
\begin{tikzcd}
    \tau_{\ge \bullet} (Hk\otimes R)  \arrow[r, altstackar=3] \ar[out=120, in=60, loop, looseness=6]{}{G}&   (\tau_{\ge \bullet} (Hk\otimes R))^{\otimes_{Hk} 2} \ar[out=120, in=60, loop, looseness=6]{}{G\wr \Sigma_2} \arrow[altstackar=5]{r} & \cdots    
\end{tikzcd}
\]
from $\Delta\varphi\wr \Sigma$ to commutative algebras in filtered spectra. Here, we abuse notation and write $\otimes_{Hk}$ for relative Day convolution over the $\mathbb{E}_\infty$-algebra $\tau_{\ge \bullet}Hk$ in filtered spectra. The fact that this filtration is complete and exhaustive follows from the fact that $\tau_{\ge \bullet}$ is complete and exhaustive. The identification of the ${E}^2$-page comes from the K\"unneth isomorphism, and \cref{remark: HS as homotopy groups}. Since $R$ is connective, the spectral sequence is a first quadrant spectral sequence with exiting differentials and so it strongly converges. 
\end{proof}

\begin{remark}
When $\varphi=\id_{C_2}$, then the input of \eqref{general-BokSS} can be identified with hyperoctahedral homology~\cite{Gra22}. There is work in progress of Graves and Whitehouse studying twisted symmetric homology more generally, which is the input of \eqref{general-BokSS} for a general group with parity. 
\end{remark}

\subsection{Topological \texorpdfstring{$\DG$}{Delta G}-homology of Thom spectra}

For a ring spectrum $R$ which arises as the Thom spectrum over a base space $X$, work of Blumberg, Cohen, and Schlichtkrull gives a formula for $\THH(R)$ in terms of $R$ and $X$ \cite{BCS}. In this subsection, we extend their results to an arbitrary crossed simplicial group $\Delta \mathbf{G}$, proving a general result for $\mathrm{TH}\mathbf{G}$ and $\mathrm{T}\mathbf{G}^+$ of Thom spectra. In the next subsections, we specialize the result to get computations of the topological symmetric homology of Thom spectra of infinite loop maps and of the topological braid homology of Thom spectra of $2$-fold loop maps.  

Let $E$ be an $\mathbb{E}_\infty$-ring and let $\mathrm{Pic}_{\mathbb{S}}=\mathrm{Pic}(\mathrm{Mod}_{E}^{\otimes})$ denote the space of invertible $E$-modules. This is a group-like $\mathbb{E}_{\infty}$-space with a canonical, strong monoidal functor $\mathrm{Pic}_{E}\to \mathrm{Mod}_{E}$.  Let $BF$ be the classifying space of stable spherical fibrations. Note that $BF\simeq B\mathrm{GL}_1\mathbb{S}$ is the path component of the identity in $\mathrm{Pic}_{\mathbb{S}}$. In the introduction, we phrased our results for Thom spectra associated to maps $f\colon X\to BF$, but here we prove our results in slightly larger generality for maps $f\colon X\to \mathrm{Pic}_{E}$. 
 
\begin{definition}[{\cite[Definition 1.4]{ABGHR14}}]
    Let $E$ be an $\mathbb{E}_\infty$-ring. Let $f\colon X\to \mathrm{Pic}_{E}$ be a map of spaces. The Thom spectrum of $f$, denoted $Mf$ is the homotopy colimit of the composite
    \[
        X\xrightarrow{f} \mathrm{Pic}_{E}\to \mathrm{Mod}_E.
    \]
    This is a functor $M\colon \mathscr{S}_{/\mathrm{Pic}_{E}}\to \mathrm{Mod}_E$, where the domain is the $\infty$-category of spaces over $\mathrm{Pic}_{E}$.
\end{definition}

Since $\mathrm{Pic}_{E}$ is an $\mathbb{E}_{\infty}$-space, the slice category $\mathscr{S}_{/\mathrm{Pic}_{E}}$ inherits the structure of a symmetric monoidal $\infty$-category \cite[Remark 2.2.2.5]{HA}.  
Informally, for a pair of objects $f\colon X\to \mathrm{Pic}_{E}$ and $g\colon Y\to \mathrm{Pic}_{E}$ in $\mathscr{S}_{/\mathrm{Pic}_{E}}$, the product of these two objects is given by the composite
\[
    X\times Y\xrightarrow{f\times g} \mathrm{Pic}_{E}\times \mathrm{Pic}_{E}\to \mathrm{Pic}_{E}
\]
where the second map is the multiplication on $\mathrm{Pic}_{E}$. 

\begin{lemma}\label{lemma: Th is strong monoidal}
    The Thom spectrum functor 
    $M\colon \mathscr{S}_{/\mathrm{Pic}_{E}}\to \mathrm{Mod}_E\,$ is strong monoidal and preserves all colimits.
\end{lemma}
\begin{proof}
The fact that $M$  preserves colimits follows from the definition of a Thom spectrum as a colimit~\cite[Definition 1.4]{ABGHR14}. The fact that it is strong monoidal follows from~\cite[Theorem~1.6]{ABG18}.  This fact is also classical, see for example~\cite{LMS,Sch09}.   
\end{proof}

\begin{lemma}\label{lem: En-algoverBF}
Let $G$ be a group with parity $\varphi \colon G\to C_2$. Let $E$ be an $\mathbb{E}_\infty$-ring. 
    An $\mathbb{E}_{1}$-algebra with twisted $G$-action in $\mathscr{S}_{/\mathrm{Pic}_{E}}$ is the same data as an $\mathbb{E}_1$-space $X$ with twisted $G$-action and a map of $\mathbb{E}_1$-spaces $f\colon X\to \mathrm{Pic}_{E}$ with twisted $G$-action. More precisely, there is a homotopy equivalence 
    \[
    \mathrm{Map}_{\Alg_{\varphi}(\mathscr{S})}(X, \mathrm{Pic}_{E}) \simeq \{X\}\times_{\Alg_{\varphi}(\mathscr{S})}\Alg_{\varphi}(\mathscr{S}_{/\mathrm{Pic}_{E}})
    \]
    of spaces. 
\end{lemma}
\begin{proof}
This follows from~\cite[Lemma 2.12]{ACB19} with, in the notation of \textit{loc.\ cit.,} $\cO=\mathrm{Fin}_*$, $\cD$ is the symmetric monoidal $\infty$-category $\mathscr{S}^{\otimes}$, $\mathscr{C}=\mathrm{Env}(\Alg_{\varphi})$, $F=X$,  $p\colon \mathrm{Env}(\Alg_{\varphi})^{\otimes}\to \mathrm{Fin}_*$ exhibits $\mathrm{Env}(\Alg_{\varphi})$ as symmetric monoidal and $A=\mathrm{Pic}_{E}$ is an $\mathbb{E}_{\infty}$-algebra in spaces regarded as an object $\mathrm{Alg}_{\varphi}(\mathscr{S})$ by restriction along $p$ (which we omit in our statement).
\end{proof}

\begin{proposition}
    \label{thm: HG commutes with strong monoidal functors}
    Let $\cC$ and  $\cD$ be presentably symmetric monoidal $\infty$-categories, and let $F\colon \cC\to \cD$ be a strong monoidal functor which preserves colimits.  For any crossed simplicial group $\DG$, and $R$ a monoid in $\cC$ with twisted $G_0$-action, there is an equivalence 
    \[
        \mathrm{H}\mathbf{G}^+(F(R)/\cD)\simeq F(\mathrm{H}\mathbf{G}^+(R/\mathscr{C})).
    \]
    Moreover, if $\DG$ is self-dual then there is an equivalence 
     \[
        \mathrm{TH}\mathbf{G}(F(R)/\cD)\simeq F(\mathrm{TH}\mathbf{G}(R/\mathscr{C})).
    \]
\end{proposition}
\begin{proof}
    There is a diagram
\[\begin{tikzcd}
	&&&& {\Env(\cC)} && {\cC} \\
	{\Delta \mathbf{G}} & {\Delta\lambda_0\wr \Sigma} & {\Env(\Assoc^{\lambda_0})} \\
	&&&& {\Env(\cD)} && {\cD}
	\arrow["\otimes", from=1-5, to=1-7]
	\arrow["{\Env(F)}", from=1-5, to=3-5]
	\arrow["F", from=1-7, to=3-7]
	\arrow["{\widetilde{\lambda}}", from=2-1, to=2-2]
	\arrow[from=2-2, to=2-3]
	\arrow["{\Env(R)}", from=2-3, to=1-5]
	\arrow["{\Env(F(R))}"', from=2-3, to=3-5]
	\arrow["\otimes", from=3-5, to=3-7]
\end{tikzcd}\]
of $\infty$-categories where the right square commutes (up to natural equivalence) because the map $\Env(\cC)\to \cC$ is natural in strong monoidal functors and the triangle commutes because $\Env$ is a functor.  By definition, the upper composite of this diagram is $F(\BDG(R))$ and the lower composite is $\BDG(F(R))$ so we obtain an equivalence $\BDG(F(R))\simeq F(\BDG(R))$.  Taking colimits, and using the hypothesis that $F$ commutes with colimits, proves the claim.

When $\Delta \mathbf{G}$ is self-dual, the same argument works for $\mathrm{TH}\mathbf{G}$ by precomposing with the map $\Delta^{\op}\to \DG^{\op}\to \DG$ where the second map is the self-duality isomorphism of $\Delta \mathbf{G}$. 
\end{proof}

\begin{corollary}\label{BCS}
Let $\DG$ be a crossed simplicial group and let $E$ be an $\mathbb{E}_\infty$-ring.  Let $f\colon X\to \Pic_{E}$ be a map of $\mathbb{E}_1$-spaces with twisted $G_0$-action.  Then there is an equivalence 
\[
    \mathrm{T}\mathbf{G}^+(Mf) \simeq M(\mathrm{H}\mathbf{G}^+(f/\mathscr{S}_{/\Pic_E}))\,.
\]    
Moreover, if $\DG$ is self-dual then there is an equivalence 
\[
    \mathrm{TH}\mathbf{G}(Mf)\simeq M(\mathrm{TH}\mathbf{G}(f/\mathscr{S}_{/\Pic_E})) \,.
\] 
\end{corollary}
\begin{proof}
    This is immediate from \cref{thm: HG commutes with strong monoidal functors}, in light of \cref{lemma: Th is strong monoidal} and \cref{lem: En-algoverBF}.
\end{proof}

\begin{example}\label{example-counit-map}
Let $E$ be a $\mathbb{E}_\infty$-ring spectrum. 
Note that if $X$ is an $\infty$-loop space and $f\colon X\to \Pic_{E}$ is an $\infty$-loop map, then $\TS(f/\mathscr{S}_{/\Pic_E}):=\mathrm{H}\Sigma^+(f/\mathscr{S}_{/\Pic_E})\in \mathscr{S}_{/\Pic_E} $ is given by the composite
\[
\TS(X)\xrightarrow{\TS(f)} \TS(\Pic_{E})\xrightarrow{\epsilon} \Pic_{E},
\]
where the map $\epsilon$ is the $\infty$-loop map that is adjunct to the identity map, which exists by \cref{representation homology thm} since $\Pic_{E}$ is an $\infty$-loop space. If $X$ is $2$-fold loop space, then

\[\TB(f/\mathscr{S}_{/\Pic_E}):=\mathrm{H}B^+(f/\mathscr{S}_{/\Pic_E})\in \mathscr{S}_{/\Pic_E} \] 
is defined in the same way. 
\end{example}

\subsection{Topological symmetric homology of Thom spectra} In this section, we give an explicit formula for $\TS(Mf)$ for infinite loop maps $f\colon X\to \mathrm{Pic}_{E}$ where $E$ is an $\mathbb{E}_\infty$-ring. 
For any based space $X$, we write $QX:= \Omega^{\infty}\Sigma^{\infty}X$.
If $X=\Omega^{\infty} Z$, the counit $\varepsilon_Z\colon \Sigma^{\infty}\Omega^{\infty}Z\to Z$ of the $\Sigma^{\infty}\dashv \Omega^{\infty}$ adjunction induces a map of infinite loop spaces
\[ 
    \epsilon = \Omega^{\infty}\varepsilon_{Z} \colon   QX\to X 
\]
whose homotopy fiber we call $\widetilde{Q}X$. Note that as the fiber of an infinite loop map, $\widetilde{Q}X$ is also an infinite loop space, and the canonical map $\iota\colon \widetilde{Q}X\to QX$ is an infinite loop map. After applying $\Sigma^{\infty}_+$, we can identify $\widetilde{Q}X$ with $\bigvee_{k\ge 2}(X^{\otimes  k})_{h\Sigma_k}$. 

More generally, if $X$ is an $n$-fold loop space we write $Q_nX \coloneqq \Omega^n\Sigma^nX$, and this comes equipped with a canonical $n$-fold loop map $Q_nX\to X$ whose homotopy fiber we will denote by $\widetilde{Q}_nX$.  

\begin{lemma}\label{lem:Qtilde}
    There is an equivalence of spaces $\widetilde{Q}_nX\times X\to Q_nX$ for any $n$-fold loop space $X$ (including $n=\infty$).  Moreover, this equivalence is natural in maps of $n$-fold loop spaces. 
\end{lemma}

\begin{proof}
    We will omit the $n$ from the notation, as the same proof works in all cases. The canonical map $\epsilon\colon QX\to X$ is split up to homotopy by the unit map $\eta_X\colon X\to QX$; this is a splitting by the triangle identities.  Note that $\eta$  preserves the unit, as it is a map of based spaces.  We have a (homotopy) commutative map of fibrations
\[
\begin{tikzcd}
	{\widetilde{Q}X} && {\widetilde{Q}X\times X} && X \\
	&& {QX\times QX} \\
	{\widetilde{Q}X} && QX && X
	\arrow["{i_1}", from=1-1, to=1-3]
	\arrow["{=}"', from=1-1, to=3-1]
	\arrow["{\pi_2}", from=1-3, to=1-5]
	\arrow["{\iota\times \eta}", from=1-3, to=2-3]
	\arrow["{=}", from=1-5, to=3-5]
	\arrow["\mu", from=2-3, to=3-3]
	\arrow["\iota"', from=3-1, to=3-3]
	\arrow["\epsilon"', from=3-3, to=3-5]
\end{tikzcd}
\]
where the map $\mu$ is an $\mathbb{E}_{n}$-multiplication on $QX$. The left square commutes because $\eta$ preserves the unit. 

Homotopy commutativity of the right square follows from considering the diagram
\[\begin{tikzcd}
	{\widetilde{Q}X\times X} &&&& X \\
	{QX\times QX} && {X\times X} \\
	QX &&&& X
	\arrow["{\pi_2}", from=1-1, to=1-5]
	\arrow["{\iota\times \eta}"', from=1-1, to=2-1]
	\arrow["{*\times \mathrm{id}}", from=1-1, to=2-3]
	\arrow["{=}", from=1-5, to=3-5]
	\arrow["{\epsilon\times \epsilon}"', from=2-1, to=2-3]
	\arrow["\mu"', from=2-1, to=3-1]
	\arrow["\mu", from=2-3, to=3-5]
	\arrow["\epsilon"', from=3-1, to=3-5]
\end{tikzcd}\]
where $*$ denotes the constant maps at the basepoint/unit of $X$.  

Now, since each fibration is split, the long exact sequences on homotopy both break up into short exact sequences.  The fact that both squares above commute up to homotopy implies the diagram
\[
    \begin{tikzcd}
        0\ar[r]  & \pi_n(\widetilde{Q}X) \ar[r,"i_1"]  \ar[d,equal] & \pi_n(\widetilde{Q}X\times X) \ar[r,"\pi_2"]  \ar[d] & \pi_n(X) \ar[r] \ar[d,equal] & 0\\
        0\ar[r]  & \pi_n(\widetilde{Q}X) \ar[r,"\iota"]  & \pi_n(QX) \ar[r,"\epsilon"]  & \pi_n(X) \ar[r] & 0
    \end{tikzcd}
\]
commutes.  Thus, the middle map is an isomorphism for all $n$ by the five lemma.
\end{proof}

% \dcnoteil{Claim: the space $\widetilde{Q}_nX$ is an $\mathbb{E}_{n}$-space, and the map $\widetilde{Q}_nX\times X\to Q_nX$ is a map of $\mathbb{E}_{n-1}$-spaces. The first claim comes from the fact that $\mathbb{E}_n$-spaces admit homotopy limits and the forgetful functor to all spaces preserves limits.  The second claim... I guess I need to know that $\eta\colon X\to Q_nX$ is a map of $E_n$-spaces? the multiplication $Q_nX\times Q_nX\to Q_nX$ is a map of $E_{n-1}$ spaces I think.}
% \mpnoteil{Yes given any $\mathbb{E}_n$-algebra $A$, the multiplication $A\times A\to A$ is always a map of $\mathbb{E}_{n-1}$-algebras (hopefully it is in HA...)}
% \dcnoteil{unfortunately I think the unit map $\eta\colon X\to Q_n X$ is just a map of spaces.}

% \dcnoteil{Ref (suggestion 17) points out we don't say what the next theorem is an equivalence of.  We are doing spectra, not ring spectra here, right? That's what we say in the intro, at least.}
% \gaknoteil{I think we can prove the stronger result. I added a rough argument at the end of the proof.}

\begin{theorem}\label{thm:TS-thomspectra}
Let $E$ be an $\mathbb{E}_\infty$-ring and let $X$ be an $\infty$-loop space. 
If $f\colon X\to \mathrm{Pic}_{E}$ is an infinite loop map, then there is an equivalence of $E$-modules
\[ 
\TS(Mf)\simeq Mf\otimes \Omega \widetilde{Q}BX_+\,.
\]
\end{theorem}
\begin{proof}
Recall from \cref{example-counit-map} and \cref{thm: HG commutes with strong monoidal functors} that $\TS(Mf)$ is equivalent to the Thom spectrum of the $\infty$-loop map 
\[
\TS(X/\mathscr{S})\xrightarrow{\TS(f)}\TS(\mathrm{Pic}_E/\mathscr{S})\overset{\epsilon}{\longrightarrow} \mathrm{Pic}_E \,.
\]
Consider the commutative diagram 
\[
\begin{tikzcd}
\Omega \widetilde{Q} BX \arrow{d} \arrow{r} & \Omega \widetilde{Q} B\Pic_{E}  \arrow{dr}{\simeq *}\arrow{d} &  \\
\TS(X/\mathscr{S})=\Omega Q BX \arrow{d}[swap]{\Omega \varepsilon}\arrow{r} & \Omega Q B\Pic_{E}\arrow{d}[swap]{\Omega \varepsilon} \arrow{r} &   \Pic_{E}\\
X \arrow{r}[swap]{f}& \Pic_{E}\arrow{ur}[swap]{\mathrm{id}}
\end{tikzcd}
\]
of spaces. We know $\TS(X/\mathscr{S})\simeq X\times \Omega \widetilde{Q}BX$ by \cref{lem:Qtilde}. 
% Since $M(\TS(f))\simeq \TS(Mf)$ is the colimit of the middle composite of the diagram, we see that 
Therefore, we can compute this as the colimit of the top composite, which gives $\Sigma^{\infty}_+\widetilde{Q}BX$, tensored with the colimit of the bottom composite, which is $Mf$.  That is, 
\[  Mf\otimes \Omega \widetilde{Q}BX_+\simeq \TS(Mf)\,.\qedhere\] 

% \gaknoteil{Here is an attempt.}
% Note that there is a splitting, as a map of spaces, of the map $\Omega \varepsilon$ 
% \[ X\to  \Omega Q BX \]
% and this induces the map 
% \[ BX \to  Q BX \, \]
% which agrees with the canonical map to the free $\mathbb{E}_{\infty}$-space on group-like $\mathbb{E}_1$-space. In other words, the map 
% \[ X\to  \Omega Q BX \]
% of spaces is the adjunct to the identity map\dcnote{how can $X\to \Omega Q B X$ be the adjunct of a map with the same target?} 
% \gaknote{I was thinking of the free forget adjunction.}
% \[\Omega Q BX  \to  \Omega Q BX \]
% of $\infty$-loop spaces.\dcnote{worth noting that this does not automatically imply that $X\to \Omega Q B X$ is a map of $\mathbb{E}_{\infty}$-space.} Passing to Thom spectra this implies that the unit map 
% \[ Mf\to Mf \otimes \Omega \tilde{Q} BX_+ \]
% is the adjunct to the map 
% \[ \TS(Mf)\to Mf \otimes \Omega \tilde{Q} BX_+ \]
% inducing the equivalence above so we can conclude that the equivalence is an equivalence of $\mathbb{E}_{\infty}$-ring spectra. \gaknote{This is proof is still missing the proof that $\Omega QBX\simeq X\times \Omega \tilde{Q}BX$ is an equivalence of infinite loop spaces.}
\end{proof}

\begin{corollary}\label{cor: TS thom spectra examples}
Let $G\in\{U,Sp,O,SO,Spin,String\}$. 
Then there is an equivalence of spectra
\[
\TS(\mathrm{M}G)\simeq \mathrm{M}G\otimes \Omega \widetilde{Q}BBG_+ \,.
\]
\end{corollary}

\begin{proof}
This follows from \cref{thm:TS-thomspectra} and the fact that $\mathrm{MG}$ is the Thom spectrum of an infinite loop map $BG\to \Pic_{\mathbb{S}}$. 
\end{proof}

\begin{corollary}\label{AF-result}
There is an equivalence of spectra
\[
\TS(\mathbb{S}[x_1^{\pm 1},x_2^{\pm 1},\cdots ,x_n^{\pm 1}])\simeq \mathbb{S}[x_1^{\pm 1},x_2^{\pm 1},\cdots ,x_n^{\pm 1}]\otimes \Omega \widetilde{Q}\mathbb{T}^n_+\,.
\]
where $\mathbb{T}^n$ is the $n$-torus. 
\end{corollary}
\begin{proof}
We can view $\mathbb{S}[x_1^{\pm 1},x_2^{\pm 1},\cdots ,x_n^{\pm 1}]$ as the Thom spectrum of the constant map $\mathbb{Z}^{n}\to \Pic_{\mathbb{S}}$, which is an infinite loop map so this follows from~\cref{thm:TS-thomspectra}.
\end{proof}
\begin{remark}
\cref{AF-result} is a topological analogue of a result in unpublished work of Ault--Fiedorowicz~\cite[Corollary~1]{AF}. 
\end{remark}

\subsection{Topological braid homology of Thom spectra}
We are also able to prove formulas for the topological braid homology of Thom spectra. Let $Y$ be a $1$-fold loop space with delooping $BY$.  The counit of the $\Sigma\dashv \Omega$ adjunction gives us a map
\[
    JY \coloneqq \Omega \Sigma Y\simeq \Omega \Sigma\Omega BY\to \Omega BY = Y.
\]
Let $\widetilde{J}Y$ denote the fiber of this map.

\begin{theorem}\label{thm:TB-thomspectra}
Let $E$ be an $\mathbb{E}_\infty$-ring. If $f\colon X\to \Pic_{E}$ is a $2$-fold loop map, then there is an equivalence 
\[ \TB(Mf)\simeq Mf\otimes \Omega \widetilde{J}BX_+\]
of $E$-modules. 
\end{theorem}
\begin{proof}
Recall from \cref{example-counit-map} and \cref{thm: HG commutes with strong monoidal functors} that $\TS(Mf)$ is equivalent to the Thom spectrum of the $2$-fold loop map 
\[
\TB(X/\mathscr{S})\overset{\TB(f)}{\longrightarrow} \TB(\mathrm{Pic}_E/\mathscr{S})\overset{\epsilon}{\longrightarrow} \mathrm{Pic}_E \,.
\]
Consider the commutative diagram
\[
\begin{tikzcd}
    \Omega \widetilde{J} BX \arrow{d} \arrow{r} & \Omega \widetilde{J} B\Pic_{E}  \arrow{dr}{\simeq *}\arrow{d} &  \\
    \TB(X/\mathscr{S})=\Omega J BX \arrow{d}[swap]{\Omega\varepsilon} \arrow{r} & \Omega J B\Pic_{E}\arrow{d}[swap]{\Omega \varepsilon} \arrow{r} &   \Pic_{E}\\
    X \arrow{r}[swap]{f}& \Pic_{E} \arrow{ur}[swap]{\mathrm{id}}
\end{tikzcd}
\]
of spaces. We know $\TB(X/\mathscr{S})\simeq X\times \Omega \widetilde{J}BX$ by \cref{lem:Qtilde}
and consequently 
\[  Mf\otimes \Omega \widetilde{J}BX_+\simeq \TB(Mf)\,.\qedhere\]
\end{proof}

% Note the same result holds $p$-locally for $2$-fold loop maps $f:X\to \mathrm{Pic}_{\mathbb{S}_{(p)}}.$ 
In the special case of spherical group rings, this result previously appeared in unpublished work of Fiedorowicz~\cite[Proposition~3.2]{Fie}. 

Using \cref{thm:TB-thomspectra} and known computations of Thom spectra, we can deduce the following computations of topological braid homology.

\begin{corollary}\label{cor:TB-Thom-spectra-examples}
Let $G$ be a stabilized Lie group, such as $U$, $Sp$, $O$, $SO$ and $Spin$. There are equivalences of spectra
\begin{align*}
    \mathrm{TB}(\mathrm{M}G)&\simeq \mathrm{M}G\otimes \Omega \widetilde{J}BBG_+ \,,\\
    \mathrm{TB}(X(k))&=X(k)\otimes \Omega \widetilde{J}SU(k)_+ \,,\\
    \mathrm{TB}(H\mathbb{F}_p)&\simeq H\mathbb{F}_p\otimes \Omega \widetilde{J}\Omega S^3_+ \,,\\
    \mathrm{TB}(H\mathbb{Z})&\simeq H\mathbb{Z}\otimes \Omega \widetilde{J}\Omega S^3\langle 4\rangle_+ \,,\\
    \mathrm{TB}(H\mathbb{Z}/2^m)&\simeq H\mathbb{Z}/2^m\otimes \Omega \widetilde{J}\mathrm{fib}(\Omega S^3\to K(\mathbb{Z}/2^{m-2},2))_+   \text{ and } \\
    \mathrm{TB}(H\mathbb{Z}/p^n)&\simeq H\mathbb{Z}/p^n\otimes \Omega \widetilde{J}\mathrm{fib}(\Omega S^3\to K(\mathbb{Z}/p^{n-1},2))_+ \text{ for } p>2
\end{align*}
where $\ell \ge 1$, $n\ge 2$ and $m\ge 3$. Here $X\langle z\rangle$ denotes the $z$-th connected cover of $X$. 
\end{corollary}
\begin{proof}
The case of $\mathrm{M}G$ follows as in the case of symmetric homology. The case of $X(k)$ follows since $X(k)$ is the Thom spectrum of a $2$-fold loop map $\Omega SU(k)\to \Pic_{\mathbb{S}}$ \cite[\S 5.3]{Rav86}. 
The case of $H\mathbb{F}_p$ follows from~\cite[\S~2.3]{Mah79} at $p=2$ and~\cite[Theorem~5.1]{ACB19} at odd primes. The case of $\mathbb{Z}$ follows from~\cite[Proposition~2.8]{Mah79}. 
The case of $\mathbb{Z}/p^n$ follows from~\cite[Theorems~3.2--3.3]{Kit20}. 
\end{proof}

\bibliographystyle{alpha}
\bibliography{bib}

\end{document}